\documentclass{article}
\usepackage[utf8]{inputenc}

 \usepackage{graphicx}
 \usepackage[utf8]{inputenc}
 \usepackage[T1]{fontenc}
 \usepackage[normalem]{ulem}
 \usepackage{verbatim}
 \usepackage{bbm}
 \usepackage{hyperref}
 \usepackage{enumerate}
 \usepackage{enumitem}
 \usepackage{stmaryrd}
  \usepackage{mathrsfs}
 \usepackage{amsmath}
 \usepackage{amssymb}
 \usepackage{dsfont}
\usepackage{amsthm}
\usepackage{color}
\usepackage{esint}
\usepackage{braket}
\usepackage{ulem}

\usepackage{xcolor}

\usepackage{amsmath}%
\usepackage{amsfonts}%
\usepackage{amssymb}%
\usepackage{graphicx}
\usepackage{dsfont}
\theoremstyle{plain} %plain, definition, remark
\newtheorem{theorem}{Theorem}[section]
\newtheorem*{theorem*}{Theorem}
\newtheorem{lemma}[theorem]{Lemma}
\newtheorem*{lemma*}{Lemma}
\newtheorem{corollary}[theorem]{Corollary}
\newtheorem*{corollary*}{Corollary}
\newtheorem{proposition}[theorem]{Proposition}
\newtheorem*{proposition*}{Proposition}
\newtheorem{assumptions}[theorem]{Assumptions}
\newtheorem{assumption}[theorem]{Assumption}
\newtheorem*{assumption*}{Assumption}

\newtheorem{definition}[theorem]{Definition}
\newtheorem*{definition*}{Definition}

\newtheorem*{example*}{Example}
\newtheorem{remark}[theorem]{Remark}

\newtheorem*{remark*}{Remark}
\newtheorem*{remarks*}{Remarks}

\newcommand{\E}{\mathbb{E}}
\newcommand{\N}{\mathbb{N}}

\newcommand{\R}{\mathbb{R}}

\renewcommand{\P}{\mathbb{P}}
\newcommand{\x}{\mathbf{x}}
\newcommand{\y}{\mathbf{y}}

\newcommand{\mcal}[1]{\mathcal{#1}}
\newcommand{\mscr}[1]{\mathscr{#1}}

\newcommand{\ii}{\mathrm{i}}

\newcommand{\msf}[1]{\mathsf{#1}}

\newcommand{\mrm}[1]{\mathrm{#1}}

\newcommand{\vertiii}[1]{{\left\vert\kern-0.25ex\left\vert\kern-0.25ex\left\vert #1 
    \right\vert\kern-0.25ex\right\vert\kern-0.25ex\right\vert}}

\DeclareMathOperator{\var}{Var}

\DeclareMathOperator{\Span}{span}

\renewcommand{\leq}{\leqslant}
\renewcommand{\geq}{\geqslant}
\renewcommand{\epsilon}{\varepsilon}
\newcommand{\red}[1]{{\textcolor{red}{#1}}}
\newcommand{\blue}[1]{{\textcolor{blue}{#1}}}
\newcommand{\defi }{\overset{(\mrm{def})}{=}}

\def\R{{\mathbb R}}

\author{Charlie Dworaczek Guera\footnote{Universit\'e de Lyon, ENSL, CNRS,  France \newline
\textit{email:} chadg@kth.se}, Ronan Memin\footnote{Institut de mathématiques de Toulouse, Toulouse, France \newline
\textit{email:} ronan.memin@math.univ-toulouse.fr}}

\title{CLT for real $\beta$-ensembles at high temperature\thanks{This project   has received funding from the European Research Council (ERC) under the European Union
Horizon 2020 research and innovation program (grant agreement No. 884584).}}

\begin{document}
\date{}
\maketitle

%\begin{abstract} We establish the central limit theorem for the fluctuations of the empirical measure in the real beta-ensemble of size $N$ with temperature proportional to $N$ and convex, smooth potential whose two first derivatives have subexponential growth, adapting part of the proof by Hardy and Lambert of the analogue CLT for the circular beta ensemble \cite{hardy2021clt}. 
%\end{abstract}

%\begin{abstract}We establish a central limit theorem for the fluctuations of the empirical measure in the real beta-ensemble at a temperature proportional to $N$ and with convex, smooth potential. We adapt part of the proof of Hardy and Lambert for the analog CLT in the case of the circular beta ensemble \cite{hardy2021clt}. The main difference in replacing the circle by the real line consists in dealing with a rapidly decreasing equilibrium density instead of a compactly supported one. Our result covers test functions that fall in the range of an integro-differential operator.\end{abstract}

\begin{abstract}We establish a central limit theorem for the fluctuations of the linear statistics in the $\beta$-ensemble of dimension $N$ at a temperature proportional to $N$ and with confining smooth potential. In this regime, the particles do not accumulate in a compact set as in the fixed $\beta>0$ case which results in an equilibrium measure supported on the whole real line. The space of test functions for which the CLT holds includes bounded $\mcal{C}^2$ functions. The method that we use is based on a change of variables in the partition function introduced in \cite{johansson} and allows to deduce the convergence of the Laplace transform of the recentred linear statistics towards the Laplace transform of the normal distribution. It is obtained by the inversion of the master operator, which is the main contribution of the present paper, by following the scheme developed in \cite{hardy2021clt} in the compact case. In the high-temperature regime, the master operator contains an additional differential term due to entropic effects which makes it an unbounded operator. The techniques used in this article involve Schrödinger operators theory as well as concentration of measure. 
\end{abstract}

\tableofcontents

\section{Introduction and main result}
\subsection{Setting of the problem}
The $\beta$-ensemble of dimension $N\geq 1$ with parameter $\beta>0$ and potential $V$, is the probability measure on $\R^N$ given by:
\begin{equation}
\label{eq:betaEnsemble}
    d\mathbf{P}^{\beta,V}_N(x_1,\ldots, x_N)=\frac{1}{Z_N(V,\beta)}\prod_{i<j}|x_i-x_j|^\beta e^{-\frac{N\beta}{2}\sum_{i=1}^NV(x_i)}d x_1\ldots d x_N\,.
\end{equation}
The potential $V$ has to be chosen so that the partition function
$$Z_N(V,\beta) = \int_{\R^N}\prod_{i<j}|x_i-x_j|^\beta e^{-\frac{N\beta}{2}\sum_{i=1}^NV(x_i)}d x_1\ldots d x_N $$
is finite. This is the case for example if
\begin{equation}
\label{condition de base sur le potentiel}
  \liminf_{|x|\to \infty}\frac{V(x)}{2\ln|x|}>1\,,
\end{equation}
see \cite[equation (2.6.2)]{AGZ}. The parameter $\beta$, which is allowed to depend on $N$, is the so-called inverse temperature.

Under the special choice of the quadratic potential $V_G(x)=x^2/2$, the measure \eqref{eq:betaEnsemble} can be seen as the joint law of the (unordered) eigenvalues of certain matrix models:
\begin{itemize}
    \item For $\beta=1$ (resp. $\beta=2$), it is the law of the eigenvalues of the Gaussian Orthogonal Ensemble (resp. Gaussian Unitary Ensemble), rescaled by a factor $1/\sqrt{N}$, see \cite[Theorem 2.5.2]{AGZ}.
    \item For general $\beta>0$, potentially depending on $N$, it is the law of the spectrum of certain tri-diagonal random matrices as shown by Dumitriu and Edelman in \cite{dued}.
\end{itemize}

We consider here the \textit{high temperature regime} where $\beta$ scales as $1/N$, and write $\beta=2P/N$ for some $P>0$. Up to absorbing a factor $2P$ in the potential, the corresponding measure is therefore
\begin{equation}
    \label{coulomb}
    d \P^{V,P}_N(x_1,\ldots,x_N) \defi  \frac{1}{\mcal{Z}_N^{P}[V]}\prod_{i<j}|x_i-x_j|^\frac{2P}{N} e^{-\sum_{i=1}^NV(x_i)}d x_1\ldots d x_N\,,
\end{equation}
with partition function 
\begin{equation}
    \label{fonction partition}
    \mcal{Z}_N^{P}[V]\defi \int_{\R^N}\prod_{i<j}|x_i-x_j|^\frac{2P}{N} e^{-\sum_{i=1}^NV(x_i)}d x_1\ldots d x_N\,.
\end{equation}
It was shown in \cite{Garcia} that under $\P^{V,P}_N$, the sequence of empirical measures 
$$\hat{\mu}_N = \frac{1}{N}\sum_{i=1}^N\delta_{x_i} $$
satisfies a large deviation principle at speed $N$ with strictly convex, good rate function. This result can be seen as the generalization of Sanov and Varadhan theorems dealing with the singular behavior of the logarithmic interaction. As a consequence, $\hat{\mu}_N$ converges almost surely in distribution towards a deterministic measure $\mu^P_V$ as $N$ goes to infinity, meaning that almost surely, for every bounded continuous $f:\R \to \R$,
$$ \int_\R f(x)d \hat{\mu}_N(x) \underset{N\rightarrow\infty}{\longrightarrow}\int_\R f(x)d \mu^P_V(x)\,.$$
The limiting measure $\mu^P_V$ can be seen to have a density $\rho^P_V$ which satisfies for almost every $x\in \R$
\begin{equation}
\label{eq:mesure equilibre}
V(x) - 2P\int_\R \ln|x-y|\rho^P_V(y)dy + \ln \rho^P_V(x) = \lambda_V^P\,,
\end{equation}
where $\lambda_V^P$ is constant (see \cite[Lemma 3.2]{GMToda} for example).

\subsection{Connection with the literature}
\paragraph{The high-temperature regime}
The $\beta$-ensemble in the regime $\beta N=2P>0$ has drawn a lot of attention from the random matrix and statistical physics communities lately. This regime was first considered by \cite{cepa1997diffusing} with the study of Dyson Brownian motion with vanishing repulsive coefficient scaled like $N^{-1}$. Gases of vortices were also studied with temperature proportional to $N$ in \cite{bodineau1999stationary}. The limiting density was then described in the case of the quadratic potential in \cite{AllezBouchaudGuionnet}, as a crossover between the Wigner semicircle law (fixed $\beta>0$ case) and the Gaussian density (case $\beta=0$). The local statistics of the system in the bulk were studied for quadratic potentials in \cite{benaych2015poisson,NakanoTrinh} and later generalized to more general potentials in \cite{NakanoTrinh20}. The edge was also studied in \cite{pakzad2018poisson}. These results both in the bulk and at the edge were generalized in \cite{lambert2021poisson} to more general geometric settings and interactions. In both cases the limiting microscopic behavior is shown to be Poissonian which can be seen as the behavior dictated by the Sine-$\beta$ and Airy $\beta$ process in the fixed $\beta>0$ regime when $\beta\rightarrow0$ \cite{allez2014sine}. Instead, at high temperature we recover the microscopic behavior of independent identically distributed particles, \textit{i.e.} at the microscale we do not see the Wigner-Gauss crossover mentioned above. At the edge, which is, up to corrections, the same as in the iid case \cite{lambert2021poisson}, it was also shown that the fluctuations follow a Gumbel distribution as opposed to the Tracy-Widom distribution in the fixed $\beta$-case which is again in accordance with \cite{allez2014tracy}. At intermediary temperatures when $\beta\gg N^{-1}$, the large deviation principle was established in \cite{pakzad2020large}. Finally the fluctuations of the linear statistics were shown to be Gaussian in the case of quadratic potentials in \cite{NakanoTrinh}. Their proof relied on the tridiagonal representation and martingale methods. In the circular case with general potential, a quantitative CLT was shown in \cite{hardy2021clt} relying on the inversion of the master operator and Stein's method. It is the only result with the present article where an explicit expression for the variance is given. In the circular case, the authors were also able to show that the high temperature variance interpolates between the $H^{1/2}$-norm which is the variance in the $\beta$-fixed case and the $L^2$-norm as in the iid case.

\paragraph{Coulomb gases} The equilibrium measure characterized by \eqref{eq:mesure equilibre} is often called the \textit{thermal equilibrium measure} in the context of Coulomb gases due to the study of \cite{armstrong2022thermal} and later in \cite{padilla2023concentration,garcia2024generalized} in the case of more general interactions. Its investigation is motivated by the fact that this measure captures more of the behavior of the system when $N$ is large and that this description is better as the temperature grows. The description obtained for this measure were then shown to be useful in this literature \cite{armstrong2021local, serfaty2023gaussian, padilla2023asymptotes, padilla2024emergence}.

\paragraph{Link with integrable systems}

Recently, Spohn uncovered in \cite{Spohn1} a link between the study of the \textit{classical Toda chain} -- a $N$ particles Hamiltonian \textit{integrable} system, in the sense that it has $N$ independent conserved quantities (in opposition to generic Hamiltonian systems, that only present a few conservation laws) -- and the $\beta$-ensemble in the high temperature regime, showing that the limiting density of states of the \textit{Lax matrix} of the classical Toda  chain, distributed according to a \textit{generalized Gibbs ensemble} with polynomial potential, can be computed by means of the limiting empirical measure of the $\beta$-ensemble at high temperature. In \cite{mazzuca2021mean}, the author established this relation using the matrix representation of the $\beta$-ensemble and a moment method, and in \cite{GMToda} the authors proved a large deviation principle for the empirical measure of the Toda chain, establishing the previous result for potentials with polynomial growth. See also \cite{spohn2022hydrodynamicAL,grava2023generalized,mazzucamemin} for a similar link between the Ablowitz-Ladik lattice and the circular $\beta$-ensemble at high temperature.\\
The relation between the statistical properties of the Toda chain and the high temperature $\beta$ ensemble can be further pushed to compute the limiting averages of the \textit{currents} of the Toda chain through the central limit theorem for the empirical measure in the $\beta$ ensemble. The computation of these currents is a crucial step in the derivation of a hydrodynamic equation for the Toda chain, and to the analysis of the correlations of the locally conserved quantities at equilibrium through \textit{linearized hydrodynamics}. The derivation of the currents via the central limit theorem is justified in the recent paper \cite{MazzucaMeminCLT}, where the authors show that on the one hand, the Lax matrix of the Toda chain, distributed with respect to the generalized Gibbs ensemble with pressure parameter $P$ and polynomial potential $V$, and on the other hand the $\beta$-ensemble at inverse temperature $2P/N$, both satisfy a central limit theorem of the form
$$ \sqrt{N}\left( \int_\R \phi(x) d\hat\mu_N(x) - \int_\R \phi(x) d\mu_{\infty}(x) \right) \overset{(d)}{\underset{N\to \infty}{\longrightarrow}}\mathcal{N}(0,\sigma^2(\phi,V,P))\,.$$
This result holds for polynomial test functions $\phi$, and $\hat\mu_N$ denotes, depending on the model, the empirical measure of eigenvalues of the Toda Lax matrix or the empirical measure of the $\beta$ ensemble at high temperature, and $\mu_\infty$ denotes the limiting measure (depending on $V,P$) of the corresponding model. The authors prove that the limiting variances are linked through the formula
$$ \sigma^2_\text{Toda}(\phi,V,P)=\partial_P\bigg(P\sigma^2_\text{HT}(\phi,V,P)\bigg)\,,$$
giving access to the limiting variance $\sigma^2_\text{Toda}(\phi,V,P)$, directly linked with the average of the currents of the Toda chain. Indeed, \cite[Lemma 2.5, Remark 2.7]{MazzucaMeminCLT} show that for integer $n$, the $n$-th limiting average current $J^{[n]}$ can be expressed \cite[(6.10)]{Spohn4}, denoting $h_n(x)=x^n$, by the limiting covariance in the $\beta$-ensembles at inverse temperature $\beta=2P/N$ and polynomial potential $V$
$$ J^{[n]}= P\lim_{N\rightarrow+\infty} \mathrm{Cov}_{V}^P\bigg(\sqrt{N}\int_\R h_1(x) d\hat\mu_N(x),\sqrt{N}\int_\R h_n(x) d\hat\mu_N(x)\bigg)$$
which can then be recast as
\begin{equation}\label{eq:formule courants}
    J^{[n]}=\frac{P}{2}\left(\sigma^2_\text{HT}(h_1+h_n,V,P)-\sigma^2_\text{HT}(h_1,V,P)-\sigma^2_\text{HT}(h_n,V,P)\right)\,.
\end{equation}
However, an explicit form for $\sigma^2_\text{HT}(\phi,V,P)$ is still lacking and the present paper goes in this direction, by computing it for a family of bounded test functions $\phi$. The interested reader can refer to \cite{Spohn4}, where the author explains the strategy to derive the hydrodynamic equations for the Toda lattice.\\

\paragraph{CLT's for $\beta$-ensembles with fixed $\beta$}

The Central Limit Theorem for the fluctuations of the linear statistics of $\beta$-ensembles was first established by \cite{johansson} for polynomial potentials, then generalized and further developed in \cite{BoG1} in the one-cut case by obtaining the asymptotic expansion of the partition function via the analysis of the so-called loop equations. The same approach was used to obtain the fluctuations in \cite{Shc2,BoG2} for the multi-cut case. Some assumptions have to be made to ensure that the fluctuations are indeed Gaussian due to the additional order 1 effects of the tunneling of the particles between the different cuts. In \cite{bekerman2018clt}, the authors proposed a new proof based on a transport method, with relaxed assumptions on $V$ and which allow for criticality behavior in the bulk and at the edge. In \cite{LambertLedouxWebb}, the authors manage, by Stein's method arguments, to obtain quantitative bounds between the fluctuations and the normal distribution in Wassertein distance. This approach was also simplified with optimal rates of convergences in \cite{angst2024sharp}. Finally, note that the fluctuations were obtained: for irregular test-functions, \textit{i.e.} when the usual $H^{1/2}$-norm variance blows up in \cite{bourgade2022optimal}, at mesoscopic scales in \cite{bekerman2018mesoscopic, peilen2024local} and by determinantal structure based methods in the case $\beta=2$ \cite{lambert1511clt,breuer2017central}.
\begin{comment}
Also an optimal local law was found in this regime in \cite{bourgade2022optimal}}. \sout{The CLT was obtained in the high-temperature regime $\beta N \to 2P>0$ by Nakano and Trinh in} \cite[Theorem 4.9]{NakanoTrinh} \sout{for quadratic $V$, relying on the tridiagonal representation for the $\beta$-ensemble with quadratic potential in \cite{dued}. See also \cite{nakano2023limit} for a refinement of these methods using orthogonal polynomials, and \cite{nakano2023beta} for a similar study involving $\beta$-Jacobi ensembles.

In \cite{hardy2021clt}, the authors prove a CLT for the smooth linear statistics of the \textit{circular} $\beta$-ensemble at high temperature with general potential, using a normal approximation method involving the spectral analysis of the \red{Dyson generator}. Their method allowed them to derive a Berry-Esseen bound, \textit{i.e.} a speed of convergence of the fluctuations towards a Gaussian variable.
\end{comment}
\ \\ 

In this paper, we adapt part of the arguments of \cite{hardy2021clt} to our setup. More precisely, we show that for a class of regular potentials $V$ satisfying a growth condition of the type 
$$\lim_{|x|\to +\infty} V(x)=\lim_{|x|\to +\infty} |V'(x)| = +\infty,\quad \lim_{|x|\to \infty}\frac{V''(x)}{V'(x)^2} = 0,$$
denoting $\mrm{fluct}_N = \hat{\mu}_N-\mu^P_V$ and considering test functions $f$ belonging to the range of a certain integro-differential operator, the rescaled fluctuations of $\hat{\mu}_N$, defined by
\begin{equation}
    \label{def:fluct}
    \sqrt{N}\mrm{fluct}_N(f) \defi  \sqrt{N}\left(\int_\R f(x)d \hat{\mu}_N(x)-\int_\R f(x)d \mu^P_V(x)\right)\,,
\end{equation}
converge in law towards a centered Gaussian law with variance depending on $f$.

\begin{comment}
The limiting measure $\mu^P_V$ can be seen to have a density $\rho^P_V$ which satisfies for almost every $x\in \R$
\begin{equation}
\label{eq:mesure equilibre}
V(x) - 2P\int_\R \ln|x-y|\rho^P_V(y)dy + \ln \rho^P_V(x) = \lambda_V^P\,,
\end{equation}
where $\lambda_V^P$ is constant (see \cite[Lemma 3.2]{GMToda} for example).
\end{comment}
It is shown in \cite[Theorem 2.6.1]{AGZ} that the empirical measure satisfies a large deviation principle, and the limiting measure is characterized in \cite[Lemma 2.6.2]{AGZ} by an equation similar to \eqref{eq:mesure equilibre}. In fact, the entropic term $\ln\rho^P_V$ in the left-hand side of \eqref{eq:mesure equilibre} is the only difference in the equation characterizing the limiting measure in the fixed $\beta$ case, due to the subdominant role of the entropy. We point out the very similar characterization of the equilibrium measure corresponding to the minimization problem arising in \cite{BGK2}. There again, the limiting measure is compactly supported. The entropic term is of prime importance because its presence implies that the support of $\rho^P_V$ is the whole real line. This fact implies several technicalities such as dealing with weighted Sobolev spaces of the form
$$H^k(\mu^P_V)\defi \Big\{f\in L^2(\mu^P_V),\,f^{(i)}\in L^2(\mu^P_V),\,\forall i \in\llbracket 1,k\rrbracket\Big\}\,.$$

\subsection{Main result}

We now present the assumptions imposed on the potential function $V$. First, recall that a probability measure $\mu$ supported on $\R$ satisfies the Poincaré inequality if there exists $C>0$ such that for all $f\in\mathcal{C}^1(\R)$ with compact support:
\begin{equation}
    \label{def:PoincareMesureProba}
    \mathrm{Var}_\mu(f)\defi \int_\R \left(f(x)-\int_\R f(y) d\mu(y)\right)^2d\mu(x) \leq C\int_\R |f'(x)|^2 d\mu(x)\,.
\end{equation}
To state the main results, we will also need the definition of the so-called \textit{master operator} $\Xi$ which acts on sufficiently smooth functions $\phi$ as:
\begin{equation}
    \label{def:Xi}
    \Xi[\phi](x) \defi 2P\int_\R \frac{\phi(x)-\phi(y)}{x-y}d\mu^P_V(y)+\phi'(x) - V'(x)\phi(x)-2P\int_\R \mcal{H}[\phi\rho^P_V](y)d\mu^P_V(y)\,.
\end{equation}
Showing that this operator is invertible with regular inverse will be of prime importance in the following. To achieve this task, we will need the following Hilbert-space $\msf{H}$:
\begin{equation}\label{def:H}
    \mathsf{H}=\left\{u\in L^2(\mu^P_V)\ \Big|\ u'\in L^2(\mu^P_V), \int_\R u(x)d\mu^P_V(x)=0\right\},\qquad \Braket{u,v}_\mathsf{H}=\Braket{u',v'}_{L^2(\mu^P_V)}\,.
\end{equation}

\begin{assumptions}
\label{assumptions}
The potential $V$ satisfies:
\begin{enumerate}[label=(\roman*)]
    \item\label{ass1} $V\in \mathcal{C}^3(\R)$, $V(x)\underset{|x|\rightarrow+\infty}{\longrightarrow}+\infty$, $|V'(x)|\underset{|x|\rightarrow+\infty}{\longrightarrow}+\infty$ and is such that $\mu^P_V$ satisfies the Poincaré inequality \eqref{def:PoincareMesureProba}.
    \item\label{ass2} For all polynomial $Q\in\R[X]$ and $\alpha>0$, $Q\bigg(V'(x)\bigg)e^{-V(x)}=\underset{|x|\to \infty}{o}(x^{-\alpha})$\,.
    \item\label{ass3} Furthermore, for any sequence $(x_N)_N$ such that $|x_N|$ goes to infinity, and for all real $a<b$,
    $$ \frac{1}{V'(x_N)^2}\sup_{a\leq x\leq b}|V''(x_N+x)| \underset{N\rightarrow\infty}{\longrightarrow} 0\,.$$
    \item\label{ass4} The function $x\mapsto V'(x)^{-2}$ is integrable at infinity, $\dfrac{V''(x)}{V'(x)}=\underset{|x|\rightarrow\infty}{O}(1)$ and $\dfrac{V^{(3)}(x)}{V'(x)}=\underset{|x|\rightarrow\infty}{O}(1)$.
\end{enumerate}
\end{assumptions}

Since we need another assumption to state our result, we postpone to Remark \ref{remarkConv} explicit examples of potentials admissible for our main result. We now discuss the previous assumptions.

%Taking $V=V_\mathrm{conv}+\phi$ with $V_\mathrm{conv},\,\phi\in\mathcal{C}^3(\R)$ such that $\phi^{(k)}$ is bounded for $k=0,\dots,3$, $V_\mathrm{conv}$ is convex with $|V_\mathrm{conv}'|\to +\infty$ at infinity, satisfying hypotheses \textit{ii)}, \textit{iii)} and \textit{iv)} such that there exists $\varepsilon>0$ 
 %such that $V_\mrm{conv}-2Pf_\varepsilon$ is convex (see Lemma \ref{lem:logpotential}), then $V$ satisfies Assumptions \ref{assumptions}.
\begin{itemize}
    \item Because \textit{\ref{ass1}} implies that $V$ goes to infinity faster than linearly, we will see that it ensures exponential decay at infinity of $\rho^P_V$. Recalling the sufficient condition for $\mathbb{P}^{V,P}_N$ of equation \eqref{condition de base sur le potentiel} to be defined, this first assumption implies that
    $$\liminf_{|x|\to \infty}\frac{V(x)}{|x|} =+\infty.$$
    This guarantees in particular that the $\beta$-ensemble \eqref{coulomb} is well-defined for all $N\geq 1$ and $P\geq0$. Finally, we will use the fact that $\mu^P_V$ satisfies the Poincaré inequality to ensure that the space $\mathsf{H}$ endowed with $\braket{\cdot,\cdot}_{\mathsf{H}}$ is a Hilbert space. 
    \item The second assumption ensures that any power of $V'$ (and of $V''$ by \textit{\ref{ass4}}) is in $L^2(\mu^P_V)$ and that $\rho^P_V$, which behaves like $e^{-V}$ up to a sub-exponential factor, belongs to the Sobolev space $H^2(\R)\subset\mcal{C}^1(\R)$. Indeed, for $k\leq2$, using \textit{\ref{ass4}}, $\rho^P_V\,^{(k)}$ behaves at infinity like $(V')^k\rho^P_V$ as shown in Lemma \ref{lem:regularitedensite} which is in $L^2(\R)$ by assumption \textit{ii).} 
    \item Assumption \textit{\ref{ass3}} will be used to localize the minimum/maximum point of a typical configuration $(x_1,\ldots,x_N)$ following the law $\P^{V,P}_N$: this will be done in Corollary \ref{cor:edge}, which comes as a consequence of \cite[Theorem 3.4]{lambert2021poisson}. More precisely, Corollary \ref{cor:edge} establishes that for some specific sequences $(\alpha_N^+)_N,(\alpha_N^-)_N$ and $(E_N^-)_N, (E_N^+)_N$, all going to infinity, the random variables
    $$ \alpha_N^+ \left( \max_{1\leq j\leq N} x_j - E_N^+ \right)\qquad\text{and }\qquad \alpha_N^- \left( \max_{1\leq j\leq N} x_j - E_N^- \right) $$
    converge in distribution. For large $N$, the scalars $E_N^+$ and $E_N^-$ can thus be seen as the edges of a typical configuration, and verify
    \begin{equation}
        \label{eq:equivalentEdge}
        V(E_N^\pm)\underset{N\rightarrow\infty}{\sim} \ln N\,.
    \end{equation}
    We refer to Section \ref{section:localization} for detailed statements, which we use to lift the result of Proposition \ref{prop:transfo laplace} from compactly supported functions to more general functions. 
    \item Finally, we use assumption \textit{\ref{ass4}} to control integral remainders in the proof of Theorem \ref{thmreg}, ensuring that $\Xi^{-1}$ is regular enough \textit{i.e.} that for sufficiently smooth functions $f$, $\Xi^{-1}[f]$ is also smooth.
\end{itemize}

We will need another technical assumption to ensure that Taylor remainders arising in the proof of Theorem \ref{thm:LaplaceGeneral} are negligible.

\begin{assumption}
\label{assumption2}
With the notations of Theorem \ref{thm:poissonLambert}, we have
$$ \sup_{d(x,I_N)\leq 1} \left|V^{(3)}(x)\right|=\underset{N\rightarrow\infty}{o}(\sqrt{N})\,,$$
where $I_N=\left[E_N^- -2;E_N^+ +2\right]$.
\end{assumption}

\begin{remark}
\label{remarkConv}
    Taking $V=V_\mathrm{conv}+\phi$ with $V_\mathrm{conv},\,\phi\in\mathcal{C}^3(\R)$ such that $\phi^{(k)}$ is bounded for $k=0,\dots,3$, $V_\mathrm{conv}$ is convex with $|V_\mathrm{conv}'|\to +\infty$ at infinity, satisfying hypotheses \textit{\ref{ass2}}, \textit{\ref{ass3}}, \textit{\ref{ass4}} and assumption \ref{assumption2}, such that there exists $\varepsilon>0$ 
 such that $V_\mrm{conv}-2Pf_\varepsilon$ is convex (see Lemma \ref{lem:logpotential}), then $V$ satisfies Assumptions \ref{assumptions} and \ref{assumption2}. The main point that needs to be checked is that the measure $\mu^P_V$ satisfies the Poincaré inequality, this will be done in Proposition \ref{prop:poincare}. The function $f_\varepsilon$ is introduced as a function that behaves like $\ln|x|$ at infinity, but has second derivative which is as small as desired.
\end{remark}

The type of functions $V_\mrm{conv}$ that one can consider is typically the convex polynomials or $\cosh(\alpha x)$. On the other hand, a \textit{composed} potential like $e^{x^2}$ has a derivative growing faster at infinity than the original function itself and so doesn't satisfy assumptions \textit{\ref{ass3}} and \textit{\ref{ass4}}.
We are now able to state the main result, \textit{i.e.} the central limit theorem for bounded smooth functions.
\begin{theorem}
\label{thm:thmppal}
Assume that $V$ satisfies Assumptions \ref{assumptions} and Assumption \ref{assumption2}. Then for all $\phi\in \mcal{C}^2(\R)$ with $\phi$, $\phi'$ and $\phi''$ bounded, we have the convergence in law towards a normal distribution
\begin{equation}
\label{eq:TCLppal}
\sqrt{N}\mrm{fluct}_N(\phi)\overset{\mrm{law}}{\underset{N\rightarrow+\infty}{\longrightarrow}} \mathcal{N}\Big(0,(\sigma^P_V)^2(\phi)\Big)\,,
\end{equation}
where $\mrm{fluct}_N(\phi)$ is defined in \eqref{def:fluct} and where the limiting variance $(\sigma^P_V)^2(\phi)$ is given by
\begin{comment}
\begin{multline}
\label{eq:variance}
    (\sigma^V_P)^2(\phi)=\braket{\phi,\mcal{L}^{-1}[\phi]}_\msf{H}=\int_\R\Bigg(\mcal{L}^{-1}[\phi]''(x)^2+V''(x)\mcal{L}^{-1}[\phi]'(x)^2\Bigg)d\mu^P_V(x)\\+P\iint_{\R^2}\Bigg(\frac{\mcal{L}^{-1}[\phi]'(x)-\mcal{L}^{-1}[\phi]'(y)}{x-y}\Bigg)^2d\mu^P_V(x)d\mu^P_V(y)\,.
\end{multline}
\end{comment}

\begin{multline}
\label{eq:variance}
    (\sigma^P_V)^2(\phi)=\int_\R \phi'(x)\Xi^{-1}\circ\mcal{X}[\phi](x)d\mu_V^P(x)=\int_\R\Bigg(\Xi^{-1}[\phi]'(x)^2+V''(x)\Xi^{-1}[\phi](x)^2\Bigg)d\mu^P_V(x)\\+P\iint_{\R^2}\Bigg(\frac{\Xi^{-1}[\phi](x)-\Xi^{-1}[\phi](y)}{x-y}\Bigg)^2d\mu^P_V(x)d\mu^P_V(y)\,.
\end{multline}
where $\mcal{X}[\phi]\defi\phi-\displaystyle\int_\R\phi(x)d\mu_V^P(x)$.

\end{theorem}

\begin{remark}\ \\
\begin{enumerate}
    \item  In particular, let $\phi_1,\ldots,\phi_d \in \mcal{C}^2(\R)$, all satisfying the hypotheses of Theorem \ref{thm:thmppal}. Considering for $\mathbf{t}=(t_1,\ldots,t_d)\in \R^d$ the function $\phi=\sum_{j=1}^d t_j\phi_j$, $\phi$ also satisfies the hypotheses. Then we have
    $$ \E\left[e^{\ii \sqrt{N}(t_1\mrm{fluct}_N(\phi_1)+\ldots+t_d\mrm{fluct}_N(\phi_d)}\right]= \E\left[e^{i\sqrt{N}\mrm{fluct}_N(\phi)}\right]\,,$$
    and the last expectation converges towards the characteristic function of a centered Gaussian variable. Thus, we see that the vector $\sqrt{N}(\mrm{fluct}_N(\phi_1),\ldots,\mrm{fluct}_N(\phi_d))$ converges towards a centered Gaussian vector whose covariance matrix is given by $\left(\braket{\phi_i',\Xi^{-1}\circ\mathcal{X}[\phi_j]}_{L^2(\mu_V^P)}\right)_{1\leq i,j\leq d}$.
    \item Note that the formula of the variance we obtain, namely
      $$\int_\R \phi'(x)\cdot\Xi^{-1}\circ\mcal{X}[\phi](x)d\mu_V^P(x)\,,$$
      has the same structure as the limiting variance in the $\beta$ fixed case \cite[Theorem 1]{bekerman2018clt}. Indeed, if one uses instead the equilibrium measure $\mu_{\mrm{fix}}$ and the inverse of the master operator $\Xi_{\mrm{fix}}$  of this regime, where the latter is defined by 
      $$\Xi_{\mrm{fix}}[\phi](x)\defi \beta\int_\R \frac{\phi(x)-\phi(y)}{x-y}d\mu_\mathrm{fix}(y) - V'(x)\phi(x)\,,$$
      one recovers the variance in the fixed temperature regime.
      \item The CLT was also established for polynomial test-functions, at least in the case of polynomial potentials \cite{NakanoTrinh,MazzucaMeminCLT}. Since $\sigma_V^P(h_k)$ makes sense for polynomial test-function $h_k(x)=x^k$, this defines a conjectured expression for the averages of the limiting currents described in \eqref{eq:formule courants}.
\end{enumerate}
\end{remark}

\begin{comment}
\begin{remark}
    In Theorem \ref{thm:thmppal}, it's possible to consider functions such that $\int_\R\phi d\mu^P_V\neq 0$ if one replaces $\phi$ by $\phi-\int_\R\phi d\mu^P_V$. \sout{Also, the assumption \ref{assumptions}
\textit{iv)} is restrictive. In the setting of a polynomial potential $V$, by a more careful analysis, one might improve the result, by considering less regular test-functions $\phi$.} \red{(pas très précis comme claim)} \blue{SUPPRIMER CETTE REMARQUE}\red{[a voir selon comment on ecrit le thm. Je te laisse choisir !]}\end{remark}
\end{comment}

%\begin{remark}
 %Since $\mrm{fluct}_N(\phi+c)=\mrm{fluct}_N(\phi)$ for all constant $c\in \R$, the assumption $\displaystyle\int_\R \phi(x)d\mu^P_V=0$ can be dropped by replacing $\phi$ by $\phi-\displaystyle\int_\R \phi(x)d\mu^P_V$ in the expression of the limiting variance.
%\end{remark}
A crucial step in our proof is to control the usual \textit{anisotropy term} given by
\begin{equation}
    \label{def:anisotropie}
    \zeta_N(\phi) \defi \iint_{\R^2} \frac{\phi(x)-\phi(y)}{x-y}d\left(\hat\mu_N - \mu_V^P\right)(x)d\left(\hat\mu_N - \mu_V^P\right)(y)\,.
\end{equation}
For this, we establish a concentration inequality for the empirical measure. This inequality is stated in terms of the following distance over the set of probability distributions $\mu,\mu'\in \mcal{M}_1(\R)$ .

\begin{equation}
\label{def:distance lip 1/2}
    d(\mu,\mu')= \sup_{\substack{\|f\|_{\text{Lip}}\leq 1\\ \|f\|_{1/2}\leq 1}}\left\{ \left|\int_\R f(x)d\mu(x) - \int_\R f(x)d\mu'(x)\right| \right\}\,,
\end{equation}
where $\|f\|_{\text{Lip}}$ denotes the Lipschitz constant of $f$, and $\|f\|_{1/2}^2 = \displaystyle\int_\R |t|\left|\mcal{F}[f](t)\right|^2d t$, where $\mcal{F}$ denotes the Fourier transform on $L^2(\R)$, expressed as $\mcal{F}[f](x)\defi\displaystyle\int_\R f(t)e^{-\ii tx}dt$ for $f$ in $L^1(\R)\cap L^2(\R)$.

This result, which is the analog of \cite[Theorem 1.4]{hardy2021clt}, goes as follows:
\begin{proposition}
    \label{thm:concentration}
    There exists $K\in \R$ (depending on $P$ and on $V$), such that for any $N\geq 1$ and $r>0$,
    \begin{equation}
    \label{ineq:concentration}\mathbb{P}^{V,P}_N\left(d(\hat{\mu}_N,\mu^P_V)>r\right) \leq e^{-Nr^2\frac{P\pi^2}{2} + 5P\ln N + K}\,.
    \end{equation}
\end{proposition}

\subsection{Strategy}
Our strategy is based on a change of variables in the partition function $\mcal{Z}_N^{P}[V]$ \eqref{fonction partition}, used for the $\beta$-ensemble at fixed temperature introduced in \cite{johansson} \cite{shcherbina2014change, bekerman2015transport, bekerman2018clt, leble2018fluctuations}, and used in \cite{BGK2, guionnet2019asymptotics} to derive the loop equations. The outline of the argument goes as follows:
Take $\phi:\R \to \R$ smooth, vanishing fast enough at infinity, and do the change of variables in $\mcal{Z}_N^{P}[V]$, $x_i=y_i+\frac{t}{\sqrt{N}}\phi(y_i)$, $1\leq i\leq N$, to get:
\begin{equation*}
    \mcal{Z}_N^{P}[V]= \int_{\R^N} \prod_{i<j}\left|y_i-y_j+ \frac{t}{\sqrt{N}}(\phi(y_i)-\phi(y_j))\right|^{2P/N}e^{-\sum_{i=1}^NV\left(y_i + \frac{t}{\sqrt{N}}\phi(y_i)\right)}\prod_{i=1}^N\left(1+\frac{t}{\sqrt{N}}\phi'(y_i) \right)d^N\y\,.
\end{equation*}
Expanding the different terms in this integral, one gets:
\begin{align*}
    \mcal{Z}_N^{P}[V] = \int_{\R^N}\prod_{i<j}|y_i-y_j|^\frac{2P}{N} e^{-\sum_{i=1}^NV(y_i)}e^{\frac{t}{\sqrt{N}}\left[\frac{2P}{N}\sum_{i<j}\frac{\phi(y_i)-\phi(y_j)}{y_i-y_j} + \sum_{i=1}^N \left(\phi'(y_i)-V'(y_i)\phi(y_i) \right) \right]}e^{-\frac{t^2}{2}q_N(\phi)}d^N\y\,,
\end{align*}
where the term $q_N(\phi)$ converges towards a limiting variance $q(\phi)$ depending on $\phi$, $P$ and $V$. After dividing both parts of the equation by $\mcal{Z}_N^{P}[V]$, and because of Equation \eqref{eq:mesure equilibre} characterizing $\mu^P_V$, one deduces the convergence of the Laplace transform
\begin{equation}
\label{eq heuristique}
    \E_N^{V,P}\left[ e^{t\sqrt{N}\left(\mrm{fluct}_N(\Xi [\phi]) + \text{error term}\right)}\right]\underset{N\rightarrow\infty}{\longrightarrow} \exp\Big(\frac{t^2}{2}q(\phi)\Big)\,,
\end{equation}
where $\Xi$ is the so-called \textit{master operator}, a linear operator acting on test functions and defined in \eqref{def:Xi}.
The error term involves the anisotropy term $\zeta_N$ given by \eqref{def:anisotropie} which is shown in Corollary \ref{cor:controleZeta} to be negligible at the scale of the CLT, as a consequence of Proposition \ref{thm:concentration}.
Equation \eqref{eq heuristique} then establishes the central limit theorem for test functions of the form $\Xi [\phi]$. The goal is then to invert $\Xi$.

Following \cite{hardy2021clt}, the operator $\mathcal{L}$ given by
\begin{equation}
    \label{def:operateurL}
    \mathcal{L}[\phi] = \Xi[\phi']
\end{equation}
can be inverted using Hilbert space techniques. In particular, the operator $\mathcal{L}$, seen as an unbounded operator on the Hilbert space $\mathsf{H}$ defined in \eqref{def:H}, can be decomposed as
$$-\mathcal{L} = \mathcal{A} + 2P\mathcal{W}\,,$$
where $\mathcal{A}$ is a positive Sturm-Liouville operator and $\mathcal{W}$ is positive and self-adjoint. Such a writing allows us to show that $-\mathcal{L}$ is invertible, see Theorem \ref{thm:InverL} and hence that so is $\Xi$.

\medbreak
\textbf{The paper is organized as follows.} In Section \ref{section:regularity_eq_measure} we discuss the regularity of the equilibrium density $\rho^P_V$ under Assumption \ref{assumptions}. In Section \ref{section:concentration}, we prove the concentration inequality Proposition \ref{thm:concentration}. Section \ref{section:localization} is dedicated to the localization of the edge of a typical configuration, mentioned in the discussion preceding the statement of Assumption \ref{assumption2}. We next prove in Section \ref{section:laplace_transform} the convergence of the Laplace transform of $\sqrt{N}\mrm{fluct}_N(\mcal{L}[\phi])$ for general functions $\phi$ which establishes Theorem \ref{thm:thmppal} for functions in the range of $\mcal{L}$. Section \ref{section:spectral_theory} is dedicated to the diagonalization and inversion of $\mathcal{L}$ given by \eqref{def:operateurL}. In Section \ref{s7}, we show regularity properties of $\mcal{L}^{-1}$ to establish Theorem \ref{thm:thmppal}. We detail in Appendix \ref{app A} elements of proof for the spectral theory of Schrödinger operators, used in Section \ref{section:spectral_theory}.
\medbreak

\textbf{Acknowledgements} The authors wish to thank Alice Guionnet, Karol Kozlowski and an anonymous referee for their helpful suggestions. We also thank Arnaud Debussche  for pointing out the link with Schrödinger operators theory and Gautier Lambert for pointing out \cite{lambert2021poisson}. We would also like to thank Jeanne Boursier, Corentin Le Bihan and Jules Pitcho for their intuition about the regularity of the inverse operator. We would like to thank Jean-Christophe Mourrat for telling us about a more general framework for Poincaré inequalities.

\section{Regularity of the equilibrium measure and Hilbert transform}
\label{section:regularity_eq_measure}
In this section, we discuss the regularity properties of the equilibrium density $\rho^P_V$, namely its decay at infinity and its smoothness, and give formulas for its two first derivatives.

The Hilbert transform, whose definition we recall, plays a central role in the analysis of the equilibrium measure. It is first defined on the Schwartz class through $\forall\phi\in \mcal{S}(\R),\,\forall x \in\R,\;$

\begin{equation}
\label{def:HilbertTransform}
\mcal{H}[\phi](x)\defi \fint_\R\dfrac{\phi(t)}{t-x}d t=\lim_{\varepsilon\downarrow 0}\int_{|t-x|>\varepsilon}\dfrac{\phi(t)}{t-x} d t=\int_0^{+\infty}\dfrac{\phi(x+t)-\phi(x-t)}{t}d t,
\end{equation}
where $\displaystyle\fint$ denotes the Cauchy principal value integral, and then extended to $L^2(\R)$ thanks to property \textit{ii)} of Lemma \ref{lemma:HilbertPropriétés}: $\|f\|_{L^2(\R)}=\dfrac{1}{\pi}\| \mathcal{H}[f] \|_{L^2(\R)}$. The last expression in \eqref{def:HilbertTransform} is a definition where the integral converges in the classical sense. We also recall the definition of the logarithmic potential $U^f$ of a density of probability $f:\R \to \R$, given for $x\in\R$ by
\begin{equation}
\label{def:potentielLogarithmique}
	U^{f}(x)\defi-\int_\R\ln|x-y|f(y)dy\,.
\end{equation}
Because we assume $f\in L^1(\R)$ to be nonnegative, $U^f$ takes values in $[-\infty,+\infty)$. If $f$ integrates the function $\ln$, \textit{i.e} $\int_\R \ln|x|f(x)d x <+\infty$, then $U^f$ takes real values. Additionally, one can check that the logarithmic potential and the Hilbert transform of $f$ are linked through the distributional identity $\big(U^{f}\big)'=\mathcal{H}[f]$.

We recall in the next lemma some properties of the Hilbert transform that we will use in the rest of the paper.

\begin{lemma}[Properties of the Hilbert transform]\ 
\label{lemma:HilbertPropriétés}
\begin{itemize}
    \item[i)] Fourier transform: For all $\phi\in L^2(\R)$,  $\mcal{F}\Big[\mcal{H}[\phi]\Big](\omega)=\ii\pi\text{sgn}(\omega)\mcal{F}[\phi](\omega)$ for all $\omega\in\R$.
    \item[ii)] As a consequence, $\dfrac{1}{\pi}\mathcal{H}$ is an isometry of $L^2(\R)$, and $\mathcal{H}$ satisfies on $L^2(\R)$ the identity $\mathcal{H}^2=-\pi^2\msf{id}$.
    \item[iii)] Derivative: For any $f\in H^1(\R)$, $\mathcal{H}[f]$ is also $H^1(\R)$ and $\mcal{H}[f]'=\mcal{H}[f']$.
    \item[iv)] For all $p>1$, the Hilbert transform can be extended as a bounded operator $\mathcal{H}:L^p(\R)\to L^p(\R)$.
    \item[v)] Skew-self adjointness: For any $f,g\in L^2(\R)$, $\Braket{\mcal{H}[f],g}_{L^2(\R)}=-\langle f,\mcal{H}[g] \rangle_{L^2(\R)}$.
\end{itemize}
\end{lemma}

\begin{proof}
We refer to \cite{Hilberttransforms} for the proofs of these properties.
\end{proof}

As a consequence of \cite{Garcia}, $\hat{\mu}_N$ converges almost surely under $\P^{V,P}_N$ towards the unique minimizer of the energy-functional $\mathcal{E}_V^P$, defined for $\mu\in \mathcal{P}(\R)$ by
\begin{equation}
\label{def:energy}
\mathcal{E}_V^P(\mu)\defi\begin{cases}\displaystyle\int_\R\Big[ V+\ln\Big(\dfrac{d\mu}{dx}\Big)\Big]d\mu-P\iint_{\R^2}\ln\big|x-y\big|d\mu(x)d\mu(y)\text{ if }\mu\ll dx
\\+\infty\text{ otherwise }
\end{cases}\,.
\end{equation}
(Here we wrote $\mu \ll dx$ for "$\mu$ is absolutely continuous with respect to Lebesgue measure") 

Consequently, following \cite[Lemma 3.2]{GMToda}, the density $\rho^P_V$ of $\mu^P_V$ satisfies equation \eqref{eq:mesure equilibre}, which we rewrite here for convenience.

\begin{equation}
\label{eq:mesure equilibre rappel}
V(x) - 2P\int_\R \ln|x-y|\rho^P_V(y)dy + \ln \rho^P_V(x) = \lambda_V^P\,,
\end{equation}
where $\lambda_V^P$ is a constant (depending on $V$ and $P$).
Using this equation, we will show in the next lemma that $\rho^P_V$ decays exponentially and is twice continuously differentiable via the representation:

$$\forall x\in\R,\;\rho^P_V(x)=\exp\Big(-V(x)-2PU^{\rho^P_V}(x)-\lambda_P^V\Big)$$
In the Gaussian potential case \textit{i.e.} $V_G(x)=\dfrac{x^2}{2}$, an explicit formula has been found first in \cite{AllezBouchaudGuionnet}, and in \cite{duy2015mean} via a different method:
$$\rho_{V_G}^P(x)=\dfrac{e^{-\frac{x^2}{2}}}{\sqrt{2\pi}}\dfrac{1}{|\hat{f_\alpha}(x)|^2},\hspace{2cm}\hat{f_\alpha}(x)\defi  \sqrt{\dfrac{P}{\Gamma(P)}}\displaystyle\int_0^{+\infty}t^{P-1}e^{-\frac{t^2}{2}+\ii xt}dt.$$
It has been established in \cite{benaych2015poisson} that $\sqrt{P+1}\rho_P^{V_G}(\sqrt{P+1}x)$ converges to the Gaussian distribution when $P$ goes to zero and the semi-circle law when $P$ goes to infinity. So in the Gaussian case, $\mu^P_V$ can be seen as an interpolation between the Gaussian distribution and the semi-circular one. In fact, this interpolation result holds true for general potentials, see \cite[Remark 2.1]{NakanoTrinh20}.
\ \\ \ \\
\textbf{We now lighten the notations $\rho^P_V$ and $\mu^P_V$ and write $\rho_V$ and $\mu_V$ for convenience.} \ \\

In the next lemma, we prove that $\rho_V$ has the same regularity as $V$.

\begin{lemma}
\label{lem:regularitedensite}
Under Assumption \ref{assumptions},
\begin{itemize}
    \item The support of $\mu_V$ is $\R$ and there exists a constant $C^V_P$ such that for all $x\in \R$, 
    $$ \rho_V(x) \leq C^V_P (1+|x|)^{2P}e^{-V(x)}\,.$$
    \item The density $\rho_V$ is in $\mcal{C}^3(\R)$ and we have
\begin{equation}\label{deriv1}
    \rho_V'=-\Big(V'+2 P\mcal{H}[\rho_V]\Big)\rho_V
\end{equation}
and
\begin{equation}\label{deriv2}
   \rho_V''=\Big(-2P\mcal{H}[\rho_V]'-V''+V'^2+4P^2\mcal{H}[\rho_V]^2+4 PV'\mcal{H}[\rho_V]\Big)\rho_V\,.
\end{equation}
\end{itemize}

\end{lemma}

\begin{proof}
For the first point, \cite[Lemma 3.2]{GMToda} establishes that the support of $\mu_V$ is the whole real axis, and that under the first condition of Assumptions \ref{assumptions}, we have the bound, valid for all $x\in \R$
\begin{equation}
    \label{ineq:densite a priori}
    \rho_V(x) \leq \frac{K^V_P}{(1+|x|)^2}\,,
\end{equation}
with $K^V_P$ a positive constant. Using \eqref{eq:mesure equilibre rappel} and the fact that
$$\ln|x-y|\leq\ln\big(1+|x|\big)+\ln\big(1+|y|\big)\,,$$
we see that for all $ x\in\R$, 
    \begin{equation}
        \label{ineq:exp bound density}
        \rho_V(x)\leq C^V_P \exp\Big(-V(x)+2P\ln(1+|x|)\Big)\,,
    \end{equation}
with 
$$C^V_P=\exp\Big(2P\int_\R \ln(1+|y|)\rho_V(y)d y +\lambda_P^V\Big)$$ which is indeed finite by \eqref{ineq:densite a priori}.

	For the second point, we use that $\big(U^{\rho_V}\big)'=\mcal{H}[\rho_V]$ weakly and equation \eqref{eq:mesure equilibre rappel} to conclude on the distributional identity
	\begin{equation*}
	\rho_V'=\Big(-V'-2P\mcal{H}[\rho_V]\Big)\rho_V\,.
	\end{equation*}
	By the second point of Assumption \ref{assumptions}, $V'(x)e^{-V(x)+2P\ln(1+|x|)}=o(x^{-1})$ as $|x|\rightarrow\infty$, thus by \eqref{ineq:exp bound density}, $V'\rho_V\in L^2(\R)$. Also since $\rho_V$ is $L^2(\R)$ and bounded, we deduce, by using that $\mathcal{H}\big[L^2(\R)\big]= L^2(\R)$, that $\mcal{H}[\rho_V]\rho_V\in L^2(\R)$. Adding up these terms we get $\rho_V\in H^1(\R)$. Because $\mcal{H}[\rho_V]'=\mcal{H}[\rho_V']$ in a weak sense by Lemma \ref{lemma:HilbertPropriétés}, $\mathcal{H}[\rho_V]\in H^1(\R)$. By the classical fact that $H^1(\R)$ is contained in the set of $1/2$-Hölder functions $\mathcal{C}^{1/2}(\R)$, we have $\mcal{H}[\rho_V]\in\mcal{C}^{1/2}(\R)$ and so $U^{\rho_V}\in\mcal{C}^{1,1/2}(\R)$, the set of functions in $\mcal{C}^1(\R)$ with derivative of class $1/2$-Hölder.\\
	Using the fact that $V$ is continuously differentiable, the previous equation for the weak derivative of $\rho_V$ then ensures that $\rho_V\in\mcal{C}^{1}(\R)$ and equation \eqref{deriv1} holds in the strong sense.
	
	Differentiating (in a weak sense) equation \eqref{deriv1} we obtain
	\begin{equation*}
	\rho_V''=\Big(-2P\mcal{H}[\rho_V]'-V''+V'^2+4P^2\mcal{H}[\rho_V]^2+4PV'\mcal{H}[\rho_V]\Big)\rho_V\,.
	\end{equation*}
	The three first terms belong to $L^2(\R)$ for the same reasons as before. Since $\rho_V\in H^1(\R)$n by Lemma \ref{lemma:HilbertPropriétés}\textit{iii)} so is $\mcal{H}[\rho_V]\in H^1(\R)$, it is then bounded over $\R$ hence the two last terms are in $L^2(\R)$ when multiplied by $\rho_V$. Finally, we can conclude that $\rho_V\in H^2(\R)$ and so that $\mcal{H}[\rho_V]\in H^2(\R)$ with $\mcal{H}[\rho_V]''=\mcal{H}[\rho_V'']$ (in a weak sense). As before, we conclude that $\rho_V\in\mcal{C}^2(\R)$ and that equation \eqref{deriv2} holds in a strong sense. By the exact same method, we can show that $\rho_V\in \mcal{C}^3(\R)$.
\end{proof}

We next show that the Hilbert transform of $\rho_V$ is continuous and decays at infinity.

\begin{lemma}
\label{lemme:boundedHilbert}
Let $u\in L^2(\R)$ such that $\int_\R u(t)dt\defi \lim_{A\to \infty}\int_{-A}^A u(t)dt$ exists and $f:t\mapsto tu(t)\in H^1(\R)$ then $$\mcal{H}[u](x)\underset{|x|\rightarrow\infty}{\sim}\displaystyle\dfrac{-\int_\R u(t)dt}{x}.$$
Moreover if  $\displaystyle\int_\R u(t)dt=0$, $\int_\R f(t)dt$ exists and $g:t\mapsto t^2u(t)\in H^1(\R)$, then $$\mcal{H}[u](x)\underset{|x|\rightarrow\infty}{\sim}\displaystyle\dfrac{-\int_\R tu(t)dt}{x^2}.$$
 As a consequence, we obtain that $\mcal{H}[\rho_V](x)\underset{|x|\rightarrow\infty}{\sim}-x^{-1}$ and the logarithmic potential $U^{\rho_V}$ is Lipschitz bounded, with bounded derivative $\mathcal{H}[\rho_V]$.
\end{lemma}

\begin{proof}
    Let $u\in L^2(\R)$, such that $\int_\R u(t)dt$ exists and $f:t\mapsto tu(t)\in H^1(\R)$. Then
    \begin{equation*}
x\mcal{H}[u](x)+\int_\R u(t)dt=\int_\R\Big[\dfrac{xu(x+t)-xu(x-t)}{2t}+\dfrac{u(x+t)}{2}+\dfrac{u(x-t)}{2}\Big]dt=\mcal{H}[f](x).
\end{equation*}
Since $f\in H^1(\R)$, so is $\mcal{H}[f]$, proving that it goes to zero at infinity. Hence $$\mcal{H}[u](x)\underset{|x|\rightarrow\infty}{\sim}\displaystyle\dfrac{-\int_\R u(t)dt}{x}$$
Moreover if $\displaystyle\int_\R u(t)dt=0$, $\int_\R f(t)dt$ exists and $g:t\mapsto t^2u(t)\in H^1(\R)$, then by the same argument:
$$x^2\mcal{H}[u](x)=x\mcal{H}[f](x)=\mcal{H}[g](x)-\int_\R f(t)dt$$
where $g(t)=t^2u(t)$. We deduce that $\mcal{H}[u](x)\underset{|x|\rightarrow\infty}{\sim}\displaystyle\dfrac{-\int_\R tu(t)dt}{x^2}$ since $\mcal{H}[g]$ goes to zero at infinity.
\end{proof}

\begin{lemma}[Asymptotic of the logarithmic potential]
    \label{lem:logpotential}
    We have the following asymptotic expansion at infinity $U^{\rho_V}(x)=-\ln|x|+\underset{|x|\rightarrow\infty}{O}(x^{-1})$.
\end{lemma}

\begin{proof}
    Since $\mcal{H}[\rho_V](x)=-x^{-1}+\underset{|x|\rightarrow\infty}{O}(x^{-2})$, and recalling that $U^{\rho_V}$ (defined by \eqref{def:potentielLogarithmique}) satisfies $(U^{\rho_V})'(x)=\mathcal{H}[\rho_V](x)$, we deduce the result by integrating $t\mapsto \mathcal{H}[\rho_V](t)+1/t$ in a neighborhood of infinity.
\end{proof}

We conclude this section by stating the Poincaré inequality for the measure $\mu_V$ under the assumption that $V$ is a bounded perturbation of a strictly convex potential $V_\mathrm{conv}$. 

\begin{lemma}
\label{lem:fonctionconvexe}
    Let $\varepsilon>0$, there exists a function $f\in \mathcal{C}^2(\R)$ such that $ f_\varepsilon(x)+\ln|x|=\underset{|x|\rightarrow\infty}{O}(1)$, and $\|f_\varepsilon''\|_{\infty}\leq \varepsilon$.
\end{lemma}

\begin{proof}

\begin{comment}
    \red{Je crois que la fonction suivante définit par la convolé avec in mollifier suffit} $$f_\varepsilon(x)\defi \int_\R\log|x-y|g_\varepsilon(y)dy$$ with $g_\varepsilon$ which verifies
    $$\int_\R g_\varepsilon=1,\;g_\varepsilon\in\mcal{C}^2_c(\R),\;\|g_\varepsilon''\|_\infty<\varepsilon, supp(g_\varepsilon)=[-1,1]$$
    Then
$$|f_\varepsilon''(x)|\leq\varepsilon\int_{[0,1]}\log\dfrac{|x-y|}{1+|x|}dy<\varepsilon M$$
    By puting $h_\varepsilon=f_{\varepsilon/M}$, there exists a function of class $\mcal{C}^2(\R)$ such that $\|h_\varepsilon''\|_\infty<\varepsilon$ and behaving like $\log|x|$ at infinity.
    \end{comment}

    Indeed, for fixed $\varepsilon>0$, let $$f_\varepsilon(x)\defi-\log\left(\sqrt{\varepsilon^{-1}+x^2}\right)\,.$$ 
    It is straightforward to verify that for all $x\in \R$,
    \begin{equation*} |f_\varepsilon''(x)|=\left|\frac{\varepsilon^{-1}-x^2}{(\varepsilon^{-1}+x^2)^2}\right| \leq \varepsilon\,.\end{equation*}
    \end{proof}\begin{proposition}
\label{prop:poincare}
Assume that $V=V_\mathrm{conv}+\phi$, where $V_\mathrm{conv}\in\mathcal{C}^3(\R)$ with $V_\mathrm{conv}$ convex and $\phi$ bounded. Assume that there exists $\varepsilon>0$ such that $V_\mathrm{conv}-2Pf_\varepsilon$ is convex ($f_\varepsilon$ being given by Lemma \ref{lem:fonctionconvexe}).
Then, the measure $\mu_V$ satisfies the Poincaré inequality: there exists a constant $C>0$ such that for all $f\in \mathcal{C}^1_c(\R)$, 
\begin{equation}
\label{ineq:Poincare}
    \var_{\mu_V}(f)\defi\int_\R \bigg(f(x)-\int_\R f(y) d\mu_V(y)\bigg)^2 d\mu_V(x)\leq C\int_\R |f'|^2(x)d \mu_V(x)\,.
\end{equation}
\end{proposition}

\begin{proof}
  We use the fact that if $\mu_1,\mu_2$ are two absolutely continuous probability measures supported on $\R$ such that  $\dfrac{1}{C}\leq\dfrac{d\mu_1}{d\mu_2}\leq C$ for some $C>0$ and $\mu_1$ satisfies Poincaré inequality with constant $C_1$ then so does $\mu_2$ for some other constant. Indeed, in that case let $f\in\mcal{C}_c^1(\R)$, we have
  \begin{align*}
  \var_{\mu_2}(f)=\inf_a \int_\R \left(f(x)-a\right)^2d\mu_2(x) &\leq C\var_{\mu_1}(f)\leq C^2C_1\int_\R f'(x) d\mu_2(x).
  \end{align*}
  Here we take $d\mu_2(x)\defi\rho_V(x)dx$ and we want to compare it to a measure $\mu_1$ supported on $\R$ defined by $d\mu_1(x)=\dfrac{1}{Z}\exp\big(-W(x)\big)dx$ for some convex function $W$. The measure $\mu_1$ then clearly verifies the Poincaré inequality: this fact comes as a direct consequence of \cite[Corollary 1.9]{PoincareBBCG}, which states that if a probability measure $\mu$ has a log-concave density on $\R$, then it satisfies \eqref{ineq:Poincare}. With the definition $W\defi V_\mathrm{conv}-2Pf_\varepsilon$ with $\varepsilon>0$ such that $V_\mrm{conv}-2Pf_\varepsilon$ is convex, $W-V-2PU^{\rho_V}$ is bounded on $\R$. It is then not hard to see that $\dfrac{1}{C}\leq\dfrac{d\mu_1}{d\mu_V}\leq C$ for some $C>0$ which allows to conclude that $\mu_V$ satisfies the Poincaré inequality.
\end{proof}
\begin{remark}
    The previous proof is easily extended to potentials $V=V_\mathrm{conv} + \phi$, with $V_\mathrm{conv}$ satisfying: there exists $\alpha>0$ such that for $|x|$ big enough, $V_\mathrm{conv}''(x)>\alpha$; such as $V(x)=x^{2n}+\phi(x)$ with $\phi$ bounded and $n\geq1$. This potential indeed fails to satisfy the condition $V''-2P f_\varepsilon''\geq0$, and we need another trick.\\ 
    On the other side, 
    one can notice that $x\mapsto V(x)-2P f_\varepsilon(x+x_0)$ for $x_0$ big enough will be convex. Indeed, the "lack of convexity" of $-f_\varepsilon$ will be compensated since $f_\varepsilon''(x+x_0)>0$ will occur in the region where $V_\mathrm{conv}''(x)>\alpha$. 
    Therefore we can use the same argument as the one of Proposition \ref{prop:poincare} to conclude that $\mu_V$ satisfies Poincaré inequality.
\end{remark}

\begin{remark}
\label{remark:poincare}
We will apply later inequality \eqref{ineq:Poincare} to more general functions than $\mathcal{C}^1_c(\R)$, namely functions of the weighted Sobolev space $H^1(\mu_V)$, defined in Section \ref{section:spectral_theory}; which can be seen as the completion of $\mathcal{C}^\infty_c(\R)$ with respect to the norm $\|u\|_{L^2(\mu_V)}+\|u'\|_{L^2(\mu_V)}$.
\end{remark}

\section{Concentration inequality, proof of Proposition \ref{thm:concentration}}
\label{section:concentration}
\subsection{Concentration inequality}
We prove in this section the concentration Proposition \ref{thm:concentration}. Its proof is a direct adaptation of Theorem 1.4 of \cite{hardy2021clt}, which shows the analogous estimate in the circular setup. It is inspired by \cite{maida2014free} and based on a comparison between a configuration $\mathbf{x}_N = (x_1,\ldots,x_N)$ sampled with $\mathbb{P}_N^{V,P}$ and a regularized version $\mathbf{y}_N=(y_1,\ldots,y_N)$, which we describe here.

\begin{definition}
\label{def:regularized}
$y_1 \defi x_1$, and for $0\leq k \leq N-1$, $y_{k+1}\defi y_k + \max\{x_{k+1}-x_k, N^{-3}\}.$
\end{definition}

Note that the configuration $\mathbf{y}_N$ given by the previous definition satisfies $y_{k+1}-y_k\geq N^{-3}$, and $\mathbf{y}_N$ is close to $\mathbf{x}_N$ in the sense that 
\begin{equation}
\label{ineq:boundedDifferences}
    \sum_{k=1}^N |x_k-y_k| \leq\dfrac{1}{2N}\,.
\end{equation}
Indeed, by construction we have $|x_k-y_k| = y_k-x_k \leq (k-1)N^{-3}$, and we get the bound by summing these inequalities.\\
The key point of the proof of Proposition \ref{thm:concentration} is comparing the empirical measure $\hat{\mu}_N = \frac{1}{N}\sum_{i=1}^N \delta_{x_i},$ where $\mathbf{x}_N$ follows $\mathbb{P}_N^{V,P}$, to the regularized measure 
\begin{equation}
    \label{eq:tildeMu}
    \widetilde{\mu}_N\defi\lambda_{N^{-5}} \ast \frac{1}{N}\sum_{i=1}^N \delta_{y_i},
\end{equation}
\textit{i.e.} the convolution of $\lambda_{N^{-5}}$ and the empirical measure, where $\lambda_{N^{-5}}$ is the uniform measure on $[0,N^{-5}]$. The interest of introducing the measure $\widetilde{\mu}_N$ is that it is close to $\hat{\mu}_N$, while having a finite energy $\mathcal{E}_V^P(\widetilde{\mu}_N)$, given by \eqref{def:energy}. Finally, notice that the empirical measure doesn't change when reordering $x_1,\ldots, x_N$, and thus we do not lose in generality for our purposes in assuming that $x_1\leq \ldots \leq x_N$ in definition \ref{def:regularized}.

We now introduce a distance on $\mathcal{P}(\R)$ which is well-suited to our context.

\begin{definition}
For $\mu,\mu'\in\mathcal{P}(\R)$ we define the distance (possibly infinite) $D(\mu,\mu')$ by 
\begin{align}
\label{def:distanceMesures}
    D(\mu,\mu') &\defi \left( -\int_{\R^2} \ln|x-y|d(\mu-\mu')(x)d(\mu-\mu')(y) \right)^{1/2}\\
                &= \left( \int_0^{+\infty} \frac{1}{t}\big|\mcal{F}[\mu-\mu'](t) \big|^2 d t \right)^{1/2}.
\end{align}
where the Fourier transform of a signed measure $\nu$ is defined by $\mcal{F}[\nu](x)\defi\displaystyle\int_\R e^{- \ii tx}d(\mu-\mu')(x)$
\end{definition}
Let $f:\R\to \R$ with finite $1/2$ norm $\|f\|_{1/2}\defi\left(\int_\R |t|\left|\mcal{F}[f](t)\right|^2d t\right)^{1/2}$. By Plancherel theorem and Hölder inequality, for any $\mu,\mu'\in \mathcal{P}(\R)$, setting $\nu=\mu-\mu'$,
$$\left|\int_\R f(x)d\mu(x) -\int_\R f(x)d\mu'(x)\right|^2 = \left|\dfrac{1}{2\pi}\int_\R |t|^{1/2}\mcal{F}[f](t)\frac{
\overline{\mcal{F}[\nu](t)}}{|t|^{1/2}}d t \right|^2 \leq \dfrac{1}{2\pi ^2}\|f\|_{1/2}^2D^2(\mu,\mu').$$
Therefore the metric $d$ defined in \eqref{def:distance lip 1/2} is dominated by $D$:
\begin{equation}
    \label{ineq:distancesComparaison}
    d(\mu,\mu')\leq \dfrac{1}{\sqrt{2}\pi}D(\mu,\mu').
\end{equation}
The following lemma shows how the distance $D$ is related to the energy-functional $\mathcal{E}_V^P$ defined in \eqref{def:energy}, we will write $\mathcal{E}_V$ for simplicity.

\begin{lemma}
\label{lemme:diffEnergies}
We have for any absolutely continuous $\mu\in \mathcal{P}(\R)$ with finite energy $\mathcal{E}_V(\mu)$,
\begin{equation}
\label{eq:diffEnergies}
    \mathcal{E}_V(\mu)-\mathcal{E}_V(\mu_V) = PD^2(\mu,\mu_V) + \int_\R \ln\frac{d \mu}{d \mu_V}(x)d \mu(x)\, .
\end{equation}
\end{lemma}

\begin{proof}[Proof of Lemma \ref{lemme:diffEnergies}]
    Subtracting $\mathcal{E}_V(\mu)-\mathcal{E}_V(\mu_V)$ we find
    \begin{multline}
        \label{eq:diffcarac}
        \mathcal{E}_V(\mu)-\mathcal{E}_V(\mu_V)=\int_\R V(x)d(\mu-\mu_V)(x) + \int_\R\ln \frac{d \mu}{d x}(x)d\mu(x) - \int_\R\ln \rho_V(x)d\mu_V(x)\\- P\iint_{\R^2} \ln|x-y|d \mu(x)d\mu(y) + P\iint \ln|x-y|d \mu_V(x)d\mu_V(y)\,.
    \end{multline}
    Now, if $\nu$ is a signed measure of mass zero, integrating \eqref{eq:mesure equilibre rappel} we get
    $$\int_\R V(x)d \nu(x) - 2P\iint_{\R^2} \ln|x-y|d \nu(x)d\mu_V(y) + \int_\R \ln\rho_V(x)d \nu(x) = 0\,.$$
    We take $\nu=\mu-\mu_V$, and get
    \begin{multline*}
        \int_\R V(x)d(\mu-\mu_V)(x) = 2P\iint_{\R^2} \ln|x-y|d \mu(x)d\mu_V(y) - 2P\iint_{\R^2} \ln|x-y|d \mu_V(x)d\mu_V(y) \\- \int_\R \ln\rho_V(x)d \mu(x) + \int_\R \ln\rho_V(x)d \mu_V(x)\, .
    \end{multline*}
    Plugging this last identity in \eqref{eq:diffcarac}, we find
    $$\mathcal{E}_V(\mu)-\mathcal{E}_V(\mu_V)=-P\iint_{\R^2}\ln|x-y|d\nu(x)d\nu(y)+\int_\R \ln\frac{d \mu}{d \mu_V}(x)d \mu(x)$$
    which establishes the result.
\end{proof}

\begin{proof}[Proof of Proposition \ref{thm:concentration}]
We first give a lower bound for the partition function $\mcal{Z}_N^{P}[V]$ \eqref{fonction partition} of $\mathbb{P}_N^{V,P}$. We rewrite it as
$$	\mcal{Z}_N^{P}[V]=\int_{\R^N}\exp\Bigg(\dfrac{2P}{N}\sum_{i<j}\ln|x_i-x_j|-\sum_{i=1}^N\Big[V(x_i)+\ln\rho_V(x_i)\Big]\Bigg)d\mu_V(x_1)\dots d\mu_V(x_N)\, ,$$
	and apply Jensen inequality to obtain:
	\begin{align*}
	\ln \mcal{Z}_N^{P}[V]&\geq\int_{\R^N}\Bigg(\dfrac{2P}{N}\sum_{i<j}\ln|x_i-x_j|-\sum_{i=1}^N\Big[V(x_i)+\ln\rho_V(x_i)\Big]\Bigg)d\mu_V(x_1)\dots d\mu_V(x_N)\\&\geq P(N-1)\displaystyle\iint_{\R^2}\ln|x-y|d\mu_V(x)d\mu_V(y)-N\int_\R\Big[V+\ln\rho_V\Big](x)d\mu_V(x)
	\\&\geq -N\mathcal{E}_V^P\big[\mu_V\big]-P\iint_{\R^2}\ln|x-y|d\mu_V(x)d\mu_V(y).
	\end{align*}
	Using this estimate and the fact that for $1\leq i,j \leq N$ we have $|x_i-x_j|\leq |y_i-y_j|$, with $\y_N=(y_1,\ldots,y_N)$ of definition \ref{def:regularized}, we deduce the bound on the density of probability
	\begin{equation}
	\label{ineq:boundDensity}
	   \frac{d\mathbb{P}^{V,P}_N}{d\x}(x_1,\ldots,x_N)\leq e^{N\mathcal{E}_V(\mu_V)+P\iint_{\R^2} \ln|x-y|d\mu_V(x)d\mu_V(y)+\frac{P}{N}\sum_{i\neq j}\ln|y_i-y_j|- \sum_{i=1}^N V(x_i)}\, . 
	\end{equation}
	Recalling \eqref{eq:tildeMu}, we now show the following estimate
	\begin{equation}
	\label{ineq:estimeeEnergie}
	    \sum_{i\neq j}\ln|y_i-y_j|\leq2+N^2\iint_{\R^2} \ln|x-y|d\widetilde{\mu}_N(x)d\widetilde{\mu}_N(y) + 5N\ln N + \frac{3}{2}N\, .
	\end{equation}
	Let $i\neq j$ and $u,v\in [0,N^{-5}]$. Since for $x\neq 0$ and $|h|\leq \frac{|x|}{2}$, we have $\big|\ln|x+h| - \ln|x|\big|\leq \frac{2|h|}{|x|}$, we deduce
	$$\big|\ln|y_i-y_j +u- v|-\ln|y_i-y_j|\big| \leq \frac{2|u-v|}{|y_i-y_j|} \leq \frac{2N^{-5}}{N^{-3}} =\dfrac{2}{N^2}.$$
Thus, summing over $i\neq j$ and integrating with respect to $u$ and $v$, we get
	
	\begin{align*}
	    \sum_{i\neq j}\ln|y_i-y_j|&\leq 2+ \sum_{i\neq j}\iint_{\R^2}\ln|y_i-y_j+u-v| d\lambda_{N^{-5}}(u)d\lambda_{N^{-5}}(v) \\
	&= 2+ N^2\iint_{\R^2} \ln|x-y|d \widetilde{\mu}_N(x)d\widetilde{\mu}_N(y) - N\iint_{\R^2} \ln|u-v|d\lambda_{N^{-5}}(u)d\lambda_{N^{-5}}(v)\, .
	\end{align*}
	The last integral is equal to $ -\frac{3}{2} - 5\ln N$, so we deduce \eqref{ineq:estimeeEnergie}. We now combine \eqref{ineq:boundDensity} and \eqref{ineq:estimeeEnergie}. Recall \eqref{def:energy} and set
	$$c_N=P\left(\displaystyle\iint_{\R^2} \ln|x-y|d\mu_V(x)d\mu_V(y) + 3/2+2/N\right)\,.$$
 Then we get
	\begin{align*}
	   \frac{d\mathbb{P}^{V,P}_N}{d\x}(x_1,\ldots,x_N) &\leq   e^{c_N+5P\ln N}\exp{\left[N\left\{ \mathcal{E}_V(\mu_V)-\mathcal{E}_V(\widetilde{\mu}_N) + \int_\R \left(V + \ln\frac{d \widetilde{\mu}_N}{d x}\right)(x)d\widetilde{\mu}_N(x)\right\} - \sum_{i=1}^N V(x_i) \right]} \\
	   &= e^{c_N+5P\ln N}\exp{\left[-NPD^2(\widetilde{\mu}_N,\mu_V) + N\int_\R\left(V+\ln\rho_V\right)(x)d \widetilde{\mu}_N(x) -\sum_{i=1}^N V(x_i)\right]}\,
	\end{align*}
	where we used equation \eqref{eq:diffEnergies} in the last equality. Using again equation \eqref{eq:mesure equilibre rappel} we then see that the density $\dfrac{d\mathbb{P}^{V,P}_N}{d\x}(x_1,\ldots,x_N)$ is bounded by
	$$ e^{c_N+ 5P\ln N}\exp{\left[-NPD^2(\widetilde{\mu}_N,\mu_V)+2PN\iint_{\R^2} \ln|x-y|d(\widetilde{\mu}_N-\hat{\mu}_N)(x)d \mu_V(y) \right]}\prod_{i=1}^N\rho_V(x_i)\, .$$
Recalling \eqref{def:potentielLogarithmique}, we used that $\displaystyle\iint_{\R^2} \ln|x-y|d(\widetilde{\mu}_N-\hat{\mu}_N)(x)d \mu_V(y) = -\int U^{\rho_V}(x)d(\widetilde{\mu}_N-\hat{\mu}_N)(x)$. 
	As a consequence of the bound on the density $\dfrac{d\mathbb{P}^{V,P}_N}{d\x}(x_1,\ldots,x_N)$ we established, we have for all $r>0$
	\begin{equation}
	\label{ineq intermediaire}
	    \mathbb{P}^{V,P}_N\left( D^2(\widetilde{\mu}_N,\mu_V)>r \right) \leq e^{-NPr+c_N+5P\ln N}\int_{\R^N}\exp\left\{ -2PN\int_\R U^{\rho_V}(x)d(\widetilde{\mu}_N-\hat{\mu}_N)(x)\right\}\prod_{i=1}^N\rho_V(x_i)d x_i\, .
	\end{equation}
	Next, we show that $-N\int_\R U^{\rho_V}d(\widetilde{\mu}_N-\hat{\mu}_N)$ is bounded. By Lemma \ref{lemme:boundedHilbert}, $U^{\rho_V}$ is differentiable with bounded derivative $\mathcal{H}[\rho_V]$ on $\R$. As a consequence,
	\begin{align*}
	    \left|N\int_\R U^{\rho_V}d(\widetilde{\mu}_N-\hat{\mu}_N)\right| &\leq \sum_{i=1}^N \int_\R\left|U^{\rho_V}(y_i+u)-U^{\rho_V}(x_i)\right|d \lambda_{N^{-5}}(u)\\
	                &\leq \|\mathcal{H}[\rho_V]\|_{\infty}\left(\sum_{i=1}^N |y_i-x_i| + N\int_\R ud \lambda_{N^{-5}}(u) \right)\\
	                &\leq \|\mathcal{H}[\rho_V]\|_{\infty} \Big(\dfrac{1}{2N} + N^{-4}/2\Big),
	\end{align*}
	where we used \eqref{ineq:boundedDifferences} in the last inequality. Therefore, we deduce from \eqref{ineq intermediaire}
	\begin{equation}
	\label{ineq:D^2}
	    \mathbb{P}^{V,P}_N\left(D^2(\widetilde{\mu}_N,\mu_V)>r\right) \leq e^{-NPr +c_N +5P\ln N + \frac{2P}{N}\|\mathcal{H}[\rho_V]\|_{\infty}} = e^{-NPr + 5P\ln N + K_N}
	\end{equation}
	with $K_N\defi c_N+\dfrac{2P}{N}\|\mathcal{H}\left[\rho_V\right]\|_{\infty}$. Since $(c_N)_N$ is bounded, so is $(K_N)_N$.\\
	Finally, let $f$ be a Lipschitz bounded function with $\|f\|_{\text{Lip}}\leq 1$, then, we have (as we did for $U^{\rho_V}$)
	$$\left| \int_\R f(x)d \hat{\mu}_N(x) - \int_\R f(x)d \widetilde{\mu}_N(x) \right| \leq N^{-2}\,. $$
	Thus by  \eqref{ineq:distancesComparaison}
	$$d(\hat{\mu}_N,\mu_V)\leq d(\hat{\mu}_N,\widetilde{\mu}_N) + d(\widetilde{\mu}_N,\mu_V)\leq N^{-2} + \frac{1}{\sqrt{2}\pi}D(\widetilde{\mu}_N,\mu_V)\,,$$
	and for any $N$ such that $r-N^{-2}\geq r/2$ (in particular $r-N^{-2}>0$) we get
	\begin{align*}
	    \mathbb{P}^{V,P}_N\left( d(\hat{\mu}_N,\mu_V) >r \right) \leq \mathbb{P}^{V,P}_N\left( \frac{1}{2\pi^2}D^2(\widetilde{\mu}_N,\mu_V) >(r-N^{-2})^2 \right) &\leq \mathbb{P}^{V,P}_N\left( \frac{1}{2\pi^2}D^2(\widetilde{\mu}_N,\mu_V) >r^2/4 \right)\, ,
	    \end{align*}
	    and the last term is bounded by $e^{-Nr^2\frac{P\pi^2}{2} + 5P\ln N + K}$ for some $K$ large enough, which concludes the proof.
\end{proof}

\subsection{Bound on the anistropy term}

As a consequence of Proposition \ref{thm:concentration}, we are able to control the anisotropy term, whose definition we recall here for convenience,
\begin{equation}
    \label{def:zeta}
    \zeta_N(\phi) = \iint_{\R^2}\frac{\phi(x)-\phi(y)}{x-y}d(\hat{\mu}_N-\mu_V)(x)d(\hat{\mu}_N-\mu_V)(y)
\end{equation}
for a certain class of test functions $\phi$.

\begin{corollary}
\label{cor:controleZeta}
There exists $C,K>0$ such that for all $\phi\in\mcal{C}^2(\R)\cap H^2(\R)$ with bounded second derivative, we have for $\varepsilon>0$ and $N$ large enough, $$\P_N^{V,P}\left(\sqrt{N}|\zeta_N(\phi)|\leq N^{-1/2+\varepsilon}\right) \geq 1 - \exp\left\{-\frac{ PN^\varepsilon}{2C\|\psi\|_{H^2(\R)}} + 5P\ln N + K\right\}\,$$
with $N_2(\phi)= \|\phi'\|_{L^2(\R)}+\|\phi''\|_{L^2(\R)}$.
\end{corollary}
\begin{proof}
We follow the proof given in \cite[Cor. 4.16]{guionnet2019asymptotics} and adapt it to our setting. Let us denote by $\widetilde{\zeta_N}(\phi)$ the quantity
$$\iint_{\R^2}\frac{\phi(x)-\phi(y)}{x-y}d(\widetilde{\mu}_N-\mu_V)(x)d(\widetilde{\mu}_N-\mu_V)(y)\,.$$
We have the almost sure inequality, by a Taylor estimate
\begin{equation}
\label{ineq:estimationTaylor}
    |\zeta_N(\phi)-\widetilde{\zeta_N}(\phi)|\leq 2N^{-2}\|\phi''\|_{\infty}\,.
\end{equation}
Thus, for any $\delta>0$, 

\begin{align*}
    \P_N^{V,P}\left( |\zeta_N(\phi)|>\delta \right) &\leq \P_N^{V,P}\left( |\zeta_N(\phi)-\widetilde{\zeta_N}(\phi)|>\delta/2 \right) + \P_N^{V,P}\left( |\widetilde{\zeta_N}(\phi)|>\delta/2 \right)\\
    &\leq \P_N^{V,P}\left(2N^{-2}\|\phi''\|_{\infty}>\delta/2\right)+\P_N^{V,P}\left( |\widetilde{\zeta_N}(\phi)|>\delta/2 \right)\,,
\end{align*}
where the first term of the right-hand side is either $0$ or $1$. With $\delta=N^{-1+\varepsilon}$, $\varepsilon>0$, it is zero for $N$ large enough. For such a choice of $\delta$, and for $N$ large enough, $$\P_N^{V,P}\left( |\zeta_N(\phi)|>N^{-1+\varepsilon} \right)\leq \P_N^{V,P}\left( |\widetilde{\zeta_N}(\phi)|>\frac{1}{2}N^{-1+\varepsilon} \right)\,. $$
We next show that, for some $C>0$ independent of $\phi$, we have 
\begin{equation}
    \label{ineq:controleZetaTilde}
    |\widetilde{\zeta_N}(\phi)|\leq CD^2(\widetilde{\mu}_N,\mu_V)\|\phi\|_{H^2(\R)}\,.
\end{equation}
We begin by showing this inequality for $\psi \in \mathcal{S}(\R)$.
By using the inverse Fourier transform we have
\begin{align*}
    \widetilde{\zeta}_N(\psi)&=\dfrac{1}{2\pi}\iint_{\R^2} \dfrac{\int_\R d t\mcal{F}[\psi](t)e^{\ii tx}-\int_\R d t\mcal{F}[\psi](t)e^{\ii ty}}{x-y}d\big(\widetilde{\mu}_N-\mu_V\big)(x)d\big(\widetilde{\mu}_N-\mu_V\big)(y)
    \\&=\dfrac{1}{2\pi}\int_\R dt\ii t\mcal{F}[\psi](t)\iint_{\R^2} e^{\ii ty}\dfrac{e^{\ii t(x-y)}-1}{\ii t(x-y)}d\big(\widetilde{\mu}_N-\mu_V\big)(x)d\big(\widetilde{\mu}_N-\mu_V\big)(y)
    \\&=\dfrac{1}{2\pi}\int_\R dt\ii t\mcal{F}[\psi](t)\iint_{\R^2} e^{\ii ty}\int_0^1d\alpha e^{\ii\alpha t(x-y)}d\big(\widetilde{\mu}_N-\mu_V\big)(x)d\big(\widetilde{\mu}_N-\mu_V\big)(y)
    \\&=\dfrac{1}{2\pi}\int_\R dt\ii t\mcal{F}[\psi](t)\int_0^1d\alpha\int_\R e^{\ii\alpha tx}d\big(\widetilde{\mu}_N-\mu_V\big)(x)\int_\R e^{\ii (1-\alpha) ty}d\big(\widetilde{\mu}_N-\mu_V\big)(y)
\end{align*}

We then apply in order the triangular inequality, Cauchy-Schwarz inequality, a change of variable and the fact that $\left|\mcal{F}\left[\widetilde{\mu}_N-\mu_V\right]\right|^2$ is an even function.
\begin{align*}
    |\widetilde{\zeta}_N(\psi)|&\leq\dfrac{1}{2\pi}\int_\R dt\left|t\mcal{F}[\psi](t)\right|\int_0^1 d\alpha \left|\mcal{F}\left[\widetilde{\mu}_N-\mu_V\right](\alpha t)\right|.\left|\mcal{F}\left[\widetilde{\mu}_N-\mu_V\right]\big((1-\alpha) t\big)\right|
    \\&\leq\dfrac{1}{2\pi}\int_\R dt\left|t\mcal{F}[\psi](t)\right|\Big(\int_0^1 d\alpha \left|\mcal{F}\left[\widetilde{\mu}_N-\mu_V\right](\alpha t)\right|^2\Big)^{\frac{1}{2}}\Big(\int_0^1 d\alpha\left|\mcal{F}\left[\widetilde{\mu}_N-\mu_V\right]\big((1-\alpha) t\big)\right|^2\Big)^{\frac12}
    \\&\leq\dfrac{1}{2\pi}\int_\R dt\left|t\mcal{F}[\psi](t)\right|\int_0^1 d\alpha \left|\mcal{F}\left[\widetilde{\mu}_N-\mu_V\right](\alpha t)\right|^2
    \\&\leq\dfrac{1}{2\pi}\int_0^{+\infty} dt\left|t\mcal{F}[\psi](t)\right|\int_0^1 \dfrac{td\alpha}{t\alpha} \left|\mcal{F}\left[\widetilde{\mu}_N-\mu_V\right](\alpha t)\right|^2+\dfrac{1}{2\pi}\int_{-\infty}^{0} dt\left|t\mcal{F}[\phi](t)\right|\int_0^1 \dfrac{-td\alpha}{-t\alpha} \left|\mcal{F}\left[\widetilde{\mu}_N-\mu_V\right](\alpha t)\right|^2
    \\&\leq\dfrac{1}{2\pi}\int_\R dt \left|t\mcal{F}[\psi](t)\right|D^2(\widetilde{\mu}_N,\mu_V)
    \\&\leq\dfrac{1}{2\pi}\Big(\int_\R dt \left|t\mcal{F}[\psi](t)\right|^2(1+t^2)\Big)^{\frac12}\Big(\int_\R\dfrac{dt}{1+t^2}\Big)^{\frac12}D^2(\widetilde{\mu}_N,\mu_V)
    \\&\leq\dfrac{1}{2\sqrt{\pi}} D^2(\widetilde{\mu}_N,\mu_V)N_2(\psi)
    \\&\leq\dfrac{1}{2\sqrt{\pi}} D^2(\widetilde{\mu}_N,\mu_V)\|\psi\|_{H^2(\R)}
\end{align*}
%\eqref{ineq:controleZetaTilde}}
By density of $\mcal{S}(\R)$ in $H^2(\R)$, and since $\widetilde{\zeta}_N: \Big(H^2(\R),\|\cdot\|_{H^2(\R)}\Big)\rightarrow\R$ is continuous, the inequality still holds for $\phi$. Thus, using equation \eqref{ineq:D^2},
$$\P_N^{V,P}\left( |\widetilde{\zeta_N}(\phi)|>\frac{1}{2}N^{-1+\varepsilon} \right)\leq \P_N^{V,P}\left( D^2(\widetilde{\mu}_N,\mu_V)>\frac{N^{-1+\varepsilon}}{2C\|\phi\|_{H^2(\R)}}\right) \leq \exp\left\{-P\frac{N^\varepsilon}{2C\|\phi\|_{H^2(\R)}} + 5P\ln N + K\right\}\,,$$
which concludes the proof.
\end{proof}

\section{Localization of the edge of a configuration}
\label{section:localization}
In \cite[Theorem 1.8, Theorem 3.4]{lambert2021poisson}, the author was able to control the edge (\textit{i.e} the minimum and the maximum) of a typical configuration $(x_1,\ldots,x_N)$ distributed according to $\P^{V,P}_N$, by showing that the random measure $$ \Xi_N \defi \sum_{j=1}^N \delta_{\varphi_N^{-1}(x_j)} $$
converges in distribution towards a Poisson point process for a function $\varphi_N$ which takes the form
$$ \varphi_N(x)\defi E_N + \alpha_N^{-1}x\,.$$
Before being more precise on the construction of $(E_N)_N$ and $(\alpha_N)_N$, we explain, following \cite{lambert2021poisson}, how one can use this convergence to localize the edge of a typical configuration $(x_1,\ldots, x_N)$. Let us assume for a moment that $\Xi_N$ converges towards a Poisson point process with intensity $\theta(x)=e^{-x}$, with $E_N\to +\infty$. In particular, the random variable
$$ \Xi_N(t,+\infty)$$
converges in distribution towards a Poisson random variable with mean $\int_t^{+\infty} e^{-x}d x$. Combined with the equalities
\begin{align*}
    \P^{V,P}_N\bigg( \Xi_N(t,+\infty) =0 \bigg)&=\P^{V,P}_N\bigg( \forall\ 1\leq j \leq N,\ \varphi_N^{-1}(x_j) = \alpha_N(x_j-E_N) \leq t \bigg)\\
    &=\P^{V,P}_N\bigg( \alpha_N\left(\max_{1\leq j\leq N} x_j - E_N\right) \leq t \bigg)\,,
\end{align*}
we deduce that for all $t\in \R$
$$ \P^{V,P}_N\left( \alpha_N\left(\max_{1\leq j\leq N} x_j - E_N\right) \leq t \right)\underset{N\rightarrow\infty}{\longrightarrow} \exp(-e^{-t})\,.$$
Therefore, the random variable 
$$\alpha_N\left(\max_{1\leq j\leq N} x_j - E_N\right)$$
converges in distribution to the Gumbel law, showing that the maximum of a configuration is of order $E_N$. Furthermore, as will be clear from the construction of $\alpha_N$ and $E_N$, $\alpha_N$ is positive, and goes to infinity as $N$ goes to infinity.

Replacing in the previous analysis $\theta(x)=e^{x}$ and $E_N\to -\infty$, we would have deduced in the same fashion that $$\alpha_N\left(\min_{1\leq j\leq N} x_j - E_N\right)$$ converges in law.

With the above notations, we can apply \cite[Theorem 3.4]{lambert2021poisson} to our context.

\begin{theorem}
\label{thm:poissonLambert}
Let $v=\pm$. There exists $(E_N^v)_N,\,(\alpha_N^v)_N$ sequences of real numbers with $|E_N^v|\to +\infty$, $\alpha_N^v>0$ for large enough $N$, satisfying $V'(E_N^v)=\alpha_N^v v$, such that:
\begin{itemize}
    \item[a)] $\dfrac{Ne^{-V(E_N^v)+2P\ln |E_N^v|+\lambda_V^P}}{\alpha_N^v} \underset{N\rightarrow\infty}{\longrightarrow} 1$ (recall $\lambda_V^P$ is defined through equation \eqref{eq:mesure equilibre}),
    \item[b)] $\frac{\ln(\alpha_N^v)}{N} \underset{N\rightarrow\infty}{\longrightarrow} 0$ and $\alpha_N^v|E_N^v|\underset{N\rightarrow\infty}{\longrightarrow} +\infty$ ,
    \item[c)] For all compact $K\subset \R$, 
    $$ (\alpha_N^v)^{-2}\sup_{x\in K}\left| V''(\varphi_N(x)) \right| \underset{N\rightarrow\infty}{\longrightarrow} 0\,.$$
\end{itemize}
As a consequence, the random measure $\Xi_N$ converges in distribution as $N\to \infty$ to a Poisson point process with intensity $\theta(x)=e^{-vx}$.
\end{theorem}

\begin{proof}
We prove it in the case $v=+$, the case where $v=-$ being similar. We show that there exists a sequence $(E_N^+)_N$ going to $+\infty$ satisfying $f(E_N^+)=-\ln N$, where we defined the function $f$ by 
$$f(x)=-V(x)+2P\ln |x|+\lambda_P^V-\ln |V'(x)|\,.$$\\
Recalling Assumption \ref{assumptions} \textit{\ref{ass1}}, $|V'|$ goes to infinity at infinity, thus $\alpha_N^+=V'(E_N^+)\to +\infty$ (in the case $v=-1$ we would have looked for a sequence $(E_N^-)_N$ going to $-\infty$ and $\alpha_N^-=-V'(E_N^-)$). \\
As a consequence of Assumptions \textit{\ref{ass2}}, one shows that $\ln|V'|$ is negligible with respect to $V$ at infinity. Therefore, because $\dfrac{\ln |x|}{V(x)}\underset{|x|\rightarrow\infty}{\longrightarrow}0$,
\begin{equation}
    \label{eq:equivalentf}
    f(x)=-V(x)+\underset{x\rightarrow+\infty}{o}(V(x))\,.
\end{equation}
%We deduce that for $0<\varepsilon<1$ fixed, there exists $A>0$ such that for all $x>A$,
%\begin{equation}
%\label{eq:encadrementf}
%    -(1+\varepsilon)V(x)<f(x)<-(1-\varepsilon)V(x)\,,
%\end{equation}
Because $f(x)\underset{x\rightarrow+\infty}{\longrightarrow} -\infty$ there exists $(E_N^+)_N$ going to infinity such that for all $N\geq 1$, $f(E_N^+)=-\ln N$. Setting $x=E_N^+$ in \eqref{eq:equivalentf}, we obtain that $-V(E_N^+)\sim f(E_N^+)= -\ln N$.
\begin{comment}
\blue{ON PEUT ENLEVER LE BLEU EN FAISANT GAFFE QUE \eqref{eq:encadrementE_N} SOIT PAS UTILISE. By convexity of $V$ and the fact that it goes to infinity at infinity, $V$ is increasing on some $[M,+\infty[$, where $M\geq 0$. Thus $-(1\pm\varepsilon)V(x)=-\ln N$ \blue{GRAVE NIMP ICI, à changer lol} \textit{iff} $x=V^{-1}\left(\dfrac{\ln N}{1\pm\varepsilon}\right)$, where $V^{-1}$ denotes $\left(V_{|[M,+\infty[}\right)^{-1}$. We conclude by \eqref{eq:encadrementf} that such an $E_N^+$ must satisfy
\begin{equation}
\label{eq:encadrementE_N}
    V^{-1}\left(\frac{\ln N}{1+\varepsilon}\right) \leq  E_N^+ \leq V^{-1}\left(\frac{\ln N}{1-\varepsilon}\right)\,. 
\end{equation}
By convexity of $V$ and the fact that it goes to infinity at infinity, $(\alpha_N^+)_N$ is non-decreasing and goes to infinity. It is thus positive for $N$ large enough, ensuring that $\alpha_N|E_N^+|\underset{N\rightarrow\infty}{\longrightarrow} +\infty$.}\blue{Note that $(E_N)_N$ can be chosen such that $(\alpha_N^+)_N$ is non-decreasing and positive ensuring that $\alpha_N|E_N^+|\underset{N\rightarrow\infty}{\longrightarrow} +\infty$}
\end{comment}
Property c) follows from assumption \textit{\ref{ass3}}, along with the fact that $\alpha_N^{-1}$ stays bounded.

It remains to show that $\dfrac{\ln(\alpha_N^+)}{N}= \dfrac{\ln|V'(E_N^+)|}{N}\underset{N\rightarrow\infty}{\longrightarrow} 0$. By construction, we have
\begin{equation*}
    \dfrac{\ln|V'(E_N^+)|}{N}=\dfrac{\ln\Big(Ne^{-V(E_N^+)+2P\ln E_N^+ +\lambda_P}\Big)}{N}=-\dfrac{V(E_N^+)}{N}+o(1)\,.
\end{equation*}
Using that $V(E_N^+)\sim \ln N$, we can conclude that $\ln| V'(E_N^+)|=o(N)$ which concludes the proof.
\begin{comment}
As before, we use the fact that $\ln(V'(E_N))$ is negligible with respect to $V(E_N)$. By what precedes, we obtain that for all $\varepsilon>0$ and for $N$ large enough, we have
$$ V\left((1-\varepsilon)V^{-1}\left(\ln N\right)\right)\leq V(E_N) \leq V\left((1+\varepsilon)V^{-1}\left(\ln N\right)\right)\,.$$

\blue{It suffices to show that for some $\varepsilon>0$, the fraction
$$ \frac{V\left((1+\varepsilon)V^{-1}\left(\ln N\right)\right)}{N}$$
goes to zero as $N$ goes to infinity.}
\end{comment}
\end{proof}

By the discussion preceding Theorem \ref{thm:poissonLambert}, we deduce

\begin{corollary}[Edge of a configuration]
\label{cor:edge}
Let $E_N^\pm$, $\alpha_N^\pm \defi |V'(E_N^\pm)|$ be the sequences of Theorem \ref{thm:poissonLambert} associated with $v=\pm 1$.  Then, both random variables 
$$ \alpha_N^+\left( \max_{1\leq j \leq N}x_j - E_N^+ \right)  $$
and 
$$ \alpha_N^-\left( \min_{1\leq j \leq N}x_j - E_N^- \right) $$
converge to a Gumbel law, whose distribution function is given for $t\geq 0$ by $\mathcal{G}([0,t]) = \exp(e^{-t})$. Furthermore, $V(E_N^\pm)\sim \ln N$ and $\alpha_N^\pm \underset{N\rightarrow\infty}{\longrightarrow} \pm\infty$.
\end{corollary}

\begin{remark}
    Note that\cite[Theorem 3.4]{lambert2021poisson} applies for $V$ of class $\mathcal{C}^2$ outside of a compact set, allowing to take $V(x)=|x|^a$ for $a>1$. In this case, we find $E_N^\pm \sim \pm(\ln N)^{1/a}$.
     If $V(x)=\cosh(x)$, we find $E_N^+ \sim -E_N^- \sim \arg\cosh(\ln N)\sim\ln\ln N$.

\end{remark}

The next lemma will be convenient in the proof of Theorem \ref{thm:LaplaceGeneral} when dealing with error terms.
\begin{lemma}
\label{lem:mesureE_N}
With the notations of Corollary \ref{cor:edge}, we have 
$$\mu_V([E_N^-,E_N^+]^c) = o(N^{-1/2})\,.$$
\end{lemma}

\begin{proof}
\begin{comment}We start by giving an upper bound (resp. lower bound) for $E_N^+$ (resp. $E_N^-$). Denoting $E_N^+$ by $E_N$, For $\varepsilon>0$ fixed and $N$ big enough, inequality \eqref{eq:encadrementE_N} is satisfied. We then can bound $E_N$ by $\ln N$: using the fact that $\liminf_N \frac{V(x)}{|x|} \geq \alpha$ for some $\alpha>0$, we find some $\gamma>\frac{1}{\alpha}$ such that for $x$ large enough, $V^{-1}(x)\leq \gamma x$. As a consequence, we have for $N$ big enough
$$ E_N \leq \frac{\gamma}{1-\varepsilon}\ln N\,.$$
Using the same argument for $E_N^-$, we obtain that there exists a constant $C$ such that for $N$ big enough,
$$ -C\ln N \leq E_N^- \leq E_N^+\leq C\ln N\,.$$
\end{comment}
Let $0<\delta<1$, to be specified later. We have
\begin{align*}
    \int_{E_N^+}^{+\infty} \rho_V(x) dx = \int_{E_N^+}^{+\infty} \rho_V(x)^\delta\rho_V(x)^{1-\delta} dx &\leq \int_\R \rho_V(x)^\delta dx  \sup_{[E_N^+,+\infty[}\rho_V^{1-\delta}\,.
\end{align*}
By the first inequality of Lemma \ref{lem:regularitedensite}, the integral is finite. Also from the same inequality, we have for some constant $C'$ and $x$ big enough $\rho_V(x) \leq  C'e^{-\frac{3}{4}V(x)}$. Because $V$ is increasing in a neighborhood of $+\infty$, we get for $N$ large enough
$$\sup_{[E_N^+,+\infty[} \rho_V^{1-\delta} \leq C'^{1-\delta}e^{-(1-\delta)\frac{3}{4}V(E_N^+)}\,. $$
Taking $\delta>0$ such that $\frac{1}{2}-(1-\delta)\frac{3}{4}=:-\gamma <0$ and using that $V(E_N^+)= \ln N+o(\ln N)$ (established in the proof of Theorem \ref{thm:poissonLambert}),
\begin{align*}
\sqrt{N}\int_{E_N^+}^{+\infty} \rho_V(x) dx \leq Ke^{-\gamma\ln N + (1-\delta)\frac{3}{4}o(\ln N)}\,, 
\end{align*}
and the right-hand side goes to zero as $N$ goes to infinity. We deal with the integral $\int_{-\infty}^{E_N^-} \rho_V dx$ in the same way.

\end{proof}

\section{Laplace transform for smooth test functions, proof of Theorem \ref{thm:thmppal}}
 
Section \ref{section:concentration} allows us to justify in Proposition \ref{prop:transfo laplace} the heuristics we gave in equation \eqref{eq heuristique} for $\phi$ having compact support. We will then extend in Theorem \ref{thm:LaplaceGeneral} this result to a more general set of functions, by an approximation by compactly supported functions, using Corollary \ref{cor:edge}.

\label{section:laplace_transform}
\begin{proposition}
\label{prop:transfo laplace}
For $\phi \in \mathcal{C}^1(\R,\R)$ with compact support, we have for any real $t$,
\begin{equation}
\label{eq:transfo laplace}
\E^{V,P}_N\left[ \exp\Big(t\sqrt{N}\mrm{fluct}_N(\Xi[\phi])\Big)\right] \underset{N\rightarrow\infty}{\longrightarrow} \exp\left\{\frac{t^2}{2}q(\phi)\right\},
\end{equation}
where $\Xi[\phi]$ is given by equation \eqref{def:Xi}, and $q(\phi)$ is given by 
\begin{equation}
\label{eq:variance2}
q(\phi)\defi\int_\R\bigg( \phi'(x)^2+V''(x)\phi(x)^2\bigg)d\mu_V(x)+P\iint_{\R^2}\Big(\frac{\phi(x)-\phi(y)}{x-y}\Big)^2d\mu_V(x)d\mu_V(y)\,.
\end{equation}
\end{proposition}

\begin{proof} Let $\phi \in \mathcal{C}_c^1(\R,\R)$, and let $t\in \R$. We perform in equation \eqref{fonction partition} the change of variables \\$x_i=y_i+\frac{t}{\sqrt{N}}\phi(y_i)$, $1\leq i \leq N$, which is a diffeomorphism for $N$ big enough. We thus have
\begin{multline}
\label{eq changement de variable}
    \mcal{Z}_N^{P}[V]= \int_{\R^N} \prod_{1\leq i<j\leq N}\left|y_i-y_j+ \frac{t}{\sqrt{N}}\big(\phi(y_i)-\phi(y_j)\big)\right|^{2P/N}.e^{-\sum_{i=1}^NV\left(y_i + \frac{t}{\sqrt{N}}\phi(y_i)\right)}
    \\\times\prod_{i=1}^N\left(1+\frac{t}{\sqrt{N}}\phi'(y_i) \right)d^N\mathbf{y},
\end{multline}
and we develop separately the different terms of this integral. The first term can be written as:
$$\prod_{i<j}\left|y_i-y_j\right|^{2P/N}\prod_{i<j}\left|1+\frac{t}{\sqrt{N}}\frac{\phi(y_i)-\phi(y_j)}{y_i-y_j} \right|^{2P/N},$$
The second product above, setting $\Delta\phi_{i,j} \defi \frac{\phi(y_i)-\phi(y_j)}{y_i-y_j}$ and using Taylor-Lagrange theorem, equals
$$ \exp \bigg(\frac{2P}{N}\sum_{i<j}\ln\left|1+\frac{t}{\sqrt{N}}\frac{\phi(y_i)-\phi(y_j)}{y_i-y_j}\right|\bigg) = \exp\bigg(\frac{2P}{N}\sum_{i<j}\left(\frac{t}{\sqrt{N}}\Delta\phi_{i,j}-\frac{t^2}{2N}(\Delta\phi_{i,j})^2 + R_{N,1}(i,j)\right)\bigg),$$
where we noticed that $1+\frac{t}{\sqrt{N}}\Delta\phi_{i,j}\geq 1- \frac{t}{\sqrt{N}}\| \phi' \|_{\infty} >0$ if $N$ is big enough, and where 
$$|R_{N,1}(i,j)| \leq \frac{|t|^3}{3N^{3/2}}\|\phi'\|_{\infty}^3.$$\\
Again by Taylor-Lagrange theorem, the second term in \eqref{eq changement de variable} equals 
$$ \exp\bigg( -\sum_{i=1}^N\left( V(y_i)+ \frac{t}{\sqrt{N}}V'(y_i)\phi(y_i)+\frac{t^2}{2N}V''(y_i)\phi(y_i)^2+ R_{N,2}(i)\right)  \bigg) $$
where $R_{N,2}(i) = \frac{t^3}{6N^{3/2}}V^{(3)}\left(y_i+\frac{t\theta_i}{\sqrt{N}}\phi(y_i)\right)\phi(y_i)^3$ for some $\theta_i \in [0,1]$, thus for $N$ large enough
$$|R_{N,2}(i)|\leq \frac{|t|^3}{6N^{3/2}}\|\phi\|_{\infty}^3\sup_{d(x,\text{supp }\phi)\leq 1}|V^{(3)}(x)|. $$
The last term reads
$$ \prod_{i=1}^N\left(1+\dfrac{t}{\sqrt{N}}\phi'(y_i)\right)=\exp\bigg( \sum_{i=1}^N \left(\frac{t}{\sqrt{N}}\phi'(y_i) -\frac{t^2}{2N}\phi'(y_i)^2+ R_{N,3}(i)\right) \bigg), $$
with $|R_{N,3}(i)|\leq \frac{t^3}{3N^{3/2}}\|\phi'\|_{\infty}^3$. Dividing both sides of equation \eqref{eq changement de variable} by $\mcal{Z}_N^{P}[V]$ we get
\begin{multline*}
    \E_N^{V,P}\bigg[\exp\left\{t\sqrt{N}\bigg( P\iint_{\R^2} \frac{\phi(x)-\phi(y)}{x-y}d\hat{\mu}_N(x)d\hat{\mu}_N(y) + \int_\R (\phi'-V'\phi)(x)d\hat{\mu}_N(x)\bigg)\right\}\times\exp\left\{K_N(t,\phi)\right\}\\
    \times\exp\left\{\frac{t^2}{2}\left(-P\iint_{\R^2}\left(\frac{\phi(x)-\phi(y)}{x-y}\right)^2d\hat{\mu}_N(x)d\hat{\mu}_N(y)-\int_\R(V''\phi^2 + \phi'^2)(x)d\hat{\mu}_N(x) \right)\right\}\bigg] = 1,
\end{multline*}
with $|K_N(t,\phi)| \leq \frac{c(t,\phi)}{\sqrt{N}}$ where $c(t,\phi)\geq 0$ is independent of $N$. This bound implies that taking the limit $N\to \infty$ we can get rid of $K_N$:
\begin{multline*}
   \lim_{N\to \infty} \E_N^{V,P}\bigg[\exp\left\{t\sqrt{N}\bigg( P\iint_{\R^2} \frac{\phi(x)-\phi(y)}{x-y}d\hat{\mu}_N(x)d\hat{\mu}_N(y) + \int_\R (\phi'-V'\phi)(x)d\hat{\mu}_N(x)\bigg)\right\} \\\times\exp\left\{\frac{t^2}{2}\left(-P\iint_{\R^2}\left(\frac{\phi(x)-\phi(y)}{x-y}\right)^2d\hat{\mu}_N(x)d\hat{\mu}_N(y)-\int_\R(V''\phi^2 + \phi'^2)(x)d\hat{\mu}_N(x) \right)\right\}\bigg] = 1.
\end{multline*}
Using Fubini's theorem (the function $(x,y)\mapsto \frac{\phi(x)-\phi(y)}{x-y}$ being bounded continuous on $\R^2$), the first line in the expectation value can be rewritten as $e^{t\sqrt{N}\Lambda_N}$ with
\begin{equation}
\label{def Lambda}
    \Lambda_N \defi 2P\iint_{\R^2}\frac{\phi(x)-\phi(y)}{x-y}d\mu_V(x)d(\hat{\mu}_N-\mu_V)(y) + \int_\R (\phi'-V'\phi)(x)d(\hat{\mu}_N-\mu_V)(x) + P\zeta_N(\phi)
\end{equation}
where $\zeta_N(\phi)$ is given by \eqref{def:zeta}, and where we differentiated Equation \eqref{eq:mesure equilibre} ---every term involved being differentiable by the results of Section \ref{section:regularity_eq_measure}--- to deduce 
$$ \int_\R\big( V'(x)\phi(x)-\phi'(x)\big)d\mu_V(x)-P\iint_{\R^2}\frac{\phi(x)-\phi(y)}{x-y}d\mu_V(x)d\mu_V(y) =0\,.$$
Let $F:\mathcal{P}(\R)\to \R$ be defined by
\begin{equation}
    \label{def:fonctionF}
    F(\mu)=-P\iint_{\R^2}\left(\frac{\phi(x)-\phi(y)}{x-y}\right)^2d\mu(x)d\mu(y)-\int_\R(V''\phi^2 + \phi'^2)(x)d\mu(x)\,.
\end{equation}
It is continuous for the topology of weak convergence since all the functions in the integrals are bounded continuous. So far we have established that 
\begin{equation}
    \label{eq:esperance=1}
    \lim_{N\to \infty}\E_N^{V,P}\left[ e^{t\sqrt{N}\Lambda_N + \frac{t^2}{2}F(\hat{\mu}_N)} \right]=1,
\end{equation}
with $\Lambda_N$ given by \eqref{def Lambda}. We now replace in the latter equation the term $F(\hat{\mu}_N)$ by its limiting expression, $F(\mu_V)$. Fix a metric that is compatible with the weak convergence of probability measures on $\R$. For example,
\begin{equation}
\label{def:FortetMourier}
    d_{\text{Lip}}(\mu,\nu) = \sup\left| \int_\R f(x) d\mu(x) - \int_\R f(x) d\nu(x) \right|\,, 
\end{equation}
where the supremum runs over $f:\R \to \R$ bounded and Lipschitz with $\|f\|_{\infty}\leq 1$ and Lipschitz constant $|f|_{\text{Lip}}\leq 1$. By the large deviations principle for $(\hat{\mu}_N)_{N}$ under the probability \eqref{coulomb} established by \cite[Theorem 1.1]{Garcia}, for all $\delta>0$ the event $A_N(\delta)\defi\{d_{\text{Lip}}(\hat{\mu}_N,\mu_V) >\delta \}$ has (for $N$ big enough) probability smaller than $e^{-Nc_\delta}$ where $c_\delta>0$. Hence, 
decomposing the expectation of equation \eqref{eq:esperance=1} into
$$ \E^{V,P}_N\left[ \mathbf{1}_{A_N(\delta)}e^{\sqrt{N}\Lambda_N+\frac{t^2}{2}F(\hat\mu_N)} \right]+\E^{V,P}_N\left[ \mathbf{1}_{A_N(\delta)^c}e^{\sqrt{N}\Lambda_N+\frac{t^2}{2}F(\hat\mu_N)} \right]=:\mcal{I}^{(1)}+\mcal{I}^{(2)}\,,$$
we can bound the second term by using Cauchy-Schwarz inequality:
$$ 0\leq \mcal{I}^{(2)} \leq \P(A_N(\delta)^c)^\frac{1}{2}\E^{V,P}_N\left[ e^{2\sqrt{N}\Lambda_N+t^2F(\hat\mu_N)}\right]^{1/2} = e^{-Nc_\delta/2}\underset{N\to +\infty}{O}\left(e^{\sqrt{N}C(\phi,V,t)}\right)\,;$$
where we used the rough bound
$$ |\sqrt{N}\Lambda_N + \frac{t^2}{2}F(\hat\mu_N)|\leq \sqrt{N}c_1(\phi,V) + \frac{t^2}{2}c_2(V,\phi)$$
for some $c_1(\phi,V)$, $c_2(\phi,V)>0$ independent of $N$.
\ \\
By taking the limit $N\to +\infty$, we conclude that

$$1=\lim_{N\to \infty} \E_N^{V,P}\left[ e^{t\sqrt{N}\Lambda_N + \frac{t^2}{2}F(\hat{\mu}_N)} \right] = \lim_{N\to \infty}\E_N^{V,P}\left[\mathbf{1}_{\{d_{\text{Lip}}(\hat{\mu}_N,\mu_V)\leq \delta\}} e^{t\sqrt{N}\Lambda_N + \frac{t^2}{2}F(\hat{\mu}_N)} \right].$$
By continuity of $F$ there is some $\varepsilon(\delta)$ which goes to $0$ as $\delta\to 0$ such that, for $d_{\text{Lip}}(\nu,\mu_V)\leq \delta$, we have $|F(\nu)-F(\mu_V)|\leq \varepsilon(\delta)$. Taking the (decreasing) limit as $\delta$ goes to zero we deduce
$$\lim_{N\to \infty} \E_N^{V,P}\left[ e^{t\sqrt{N}\Lambda_N + \frac{t^2}{2}F(\hat{\mu}_N)} \right] = \lim_{\delta\to 0}\lim_{N\to \infty}\E_N^{V,P}\left[\mathbf{1}_{\{d_{\text{Lip}}(\hat{\mu}_N,\mu_V)\leq \delta\}} e^{t\sqrt{N}\Lambda_N} \right]e^{\frac{t^2}{2}F(\mu_V)}.$$
But the same large deviations argument shows that
$$\lim_{\delta\to 0}\lim_{N\to \infty}\E_N^{V,P}\left[\mathbf{1}_{\{d_{\text{Lip}}(\hat{\mu}_N,\mu_V)\leq \delta\}} e^{t\sqrt{N}\Lambda_N} \right] = \lim_{N\to \infty}\E_N^{V,P}\left[ e^{t\sqrt{N}\Lambda_N} \right]. $$
Thus, we have shown that
\begin{equation}
    \label{presque TCL}
    \lim_{N\to \infty} \E_N^{V,P}\left[e^{t\sqrt{N}\left(2P\iint_{\R^2}\frac{\phi(x)-\phi(y)}{x-y}d\mu_V(x)d(\hat{\mu}_N-\mu_V)(y) + \int_\R (\phi'-V'\phi)(x)d(\hat{\mu}_N-\mu_V)(x) + P\zeta_N(\phi) \right)}\right] = e^{-\frac{t^2}{2}F(\mu_V)}\,,
\end{equation}
which establishes, that $\sqrt{N}\Lambda_N= \sqrt{N}\Big(\mrm{fluct}_N(\Xi[\phi]) + P\zeta_N(\phi)\Big)$ converges in law towards a centered Gaussian random variable with announced variance. We finally get rid of the remaining term $\zeta_N(\phi)$, using Corollary \ref{cor:controleZeta}: taking $\varepsilon = 1/4$ for example, we see in particular that $\sqrt{N}\zeta_N(\phi)$ converges in probability towards zero. The conclusion follows from Slutsky's lemma.
\end{proof}

We now extend the result of Proposition \ref{prop:transfo laplace} to a more general set of functions. With the notations of Proposition \ref{prop:transfo laplace}, we have

\begin{theorem}
\label{thm:LaplaceGeneral}
Let $\phi\in H^2(\R)\cap \mathcal{C}^2(\R)$ such that $\phi''$ is bounded. Additionally, suppose that $V^{(3)}\phi^2$, $V''\phi\phi'$, $V''\phi^2$ and $V'\phi$ are bounded. Then, recalling \eqref{eq:variance2}, we have the following convergence in distribution,
$$ \sqrt{N}\mrm{fluct}_N(\Xi[\phi]) \underset{N\rightarrow\infty}{\overset{\mrm{law}}{\longrightarrow}} \mathcal{N}(0,q(\phi))\,.$$
\end{theorem}

\begin{proof}
For $N\geq 1$, let $E_N^-, E_N^+$ be given by Corollary \ref{cor:edge}. Let $\chi_N:\R\to [0,1]$ be $\mathcal{C}^2$ with compact support such that 
$$\chi_N(x)=1\text{ for }x\in [E_N^--1,E_N^++1]\text{ and }\chi_N(x)=0\text{ for }x\in [E_N^--2,E_N^++2]^c$$
and such that, denoting $\phi_N=\phi\chi_N$, $\sup_N\|\phi_N^{(k)}\|_{\infty}+\|\phi_N^{(k)}\|_{L^2(\R)}<+\infty$ for $k=0,1,2$ (we assumed $\phi\in H^2(\R)$ and $\|\phi''\|_{\infty}<+\infty$, in particular $\phi$, $\phi'$ and $\phi''$ are bounded and such a $\chi_N$ exists). The point of cutting $\phi$ outside the set $[E_N^--1,E_N^++1]$ is that with high probability, the empirical measure $\hat{\mu}_N$ doesn't see the difference between $\phi$ and $\phi_N$.

The support of $\phi_N$ is then contained in $[E_N^--2,E_N^++2]$, and we now argue that the proof of Proposition \ref{prop:transfo laplace} can be adapted so that
\begin{equation}
\label{eq:convergenceIntermediaire}
    \sqrt{N}\mrm{fluct}_N(\Xi[\phi_N]) \to \mathcal{N}(0,q(\phi))\,.
\end{equation}

Similarly as in Proposition \ref{prop:transfo laplace}, we perform in $\mcal{Z}_N^{P}[V]$ the change of variables
$x_i=y_i+\frac{t}{\sqrt{N}}\phi_N(y_i)$, $1\leq i \leq N$, which is the same as before, but with $\phi$ replaced by $\phi_N$. First, with $I_N\defi [E_N^--2,E_N^++2]$, the error term 
$$K_N(t,\phi_N) \leq 2\frac{t^3}{3N^{1/2}}\|\phi_N'\|_{\infty}^3+\frac{t^3}{6N^{1/2}}\|\phi_N\|_{\infty}\sup_{d(x,I_N)\leq 1}|V^{(3)}(x)|$$ 
of the proof of Proposition \ref{prop:transfo laplace} is still going to zero, because of our choice of $\chi_N$ and Assumption \ref{assumption2}. As previously, we then have
\begin{equation}
\label{eq:limiteLaplacePhi_N}
    \lim_{N\to \infty}\E_N^{V,P}\left[ e^{t\sqrt{N}\Lambda_N(\phi_N) + \frac{t^2}{2}F_N(\hat{\mu}_N)} \right]=1
\end{equation}
with 
$$ \Lambda_N(\phi_N) \defi  2P\iint_{\R^2}\frac{\phi_N(x)-\phi_N(y)}{x-y}d\mu_V(x)d(\hat{\mu}_N-\mu_V)(y) + \int_\R (\phi_N'-V'\phi_N)(x)d(\hat{\mu}_N-\mu_V)(x) + P\zeta_N(\phi_N)\,,$$
where $\zeta_N$ is given by \eqref{def:zeta}, and 
$$ F_N(\hat\mu_N)=-P\iint_{\R^2}\left(\frac{\phi_N(x)-\phi_N(y)}{x-y}\right)^2d\hat\mu_N(x)d\hat\mu_N(y)-\int_\R(V''\phi_N^2 + \phi_N'^2)(x)d\hat\mu_N(x)\,.$$
Taking again the distance $d_{\text{Lip}}$ defined in \eqref{def:FortetMourier}, one can check that for $\mu$, $\nu$ probability measures over $\R$, 
$$ \left| F_N(\mu)-F_N(\nu) \right| \leq C_N d_{\text{Lip}}(\mu,\nu)\,, $$
where $C_N$ is a term depending on the norms $\|\phi_N'\|_{\infty},\|\phi_N''\|_{\infty}$, $\| V''\phi_N^2\|_{\infty}$ and $\|(V''\phi_N^2)'\|_{\infty}$. The choice of $\chi_N$ and the fact that $\phi$ is chosen so that $V^{(3)}\phi^2$ and $V''\phi\phi'$ are bounded guarantee that $\|(V''\phi_N^2)'\|_{\infty}$ is bounded in $N$. The other norms are easily bounded by hypothesis. Therefore $C_N$ can be seen to be uniformly bounded in $N$, and we find some $C\geq 0$ independent of $N$ such that 
$$ \left| F_N(\mu)-F_N(\nu) \right| \leq C d_{\text{Lip}}(\mu,\nu)\,. $$
As in proposition \ref{prop:transfo laplace}, we use the large deviation principle for $(\hat{\mu}_N)$ to deduce
$$ \lim_{N\to +\infty} \E^{V,P}_N\left[ e^{t\sqrt{N}\Lambda_N(\phi_N)+\frac{t^2}{2}F_N(\hat{\mu}_N)} \right] = \lim_{N\to +\infty} \E^{V,P}_N\left[ e^{t\sqrt{N}\Lambda_N(\phi_N)} \right]e^{\frac{t^2}{2}F_N(\mu_V)}\,.$$
By dominated convergence, $F_N(\mu_V)$ converges to $F(\mu_V)$, the function $F$ being given by \eqref{def:fonctionF}.
This shows the convergence as $N$ goes to infinity
$$ \lim_{N\to +\infty} \E^{V,P}_N\left[ e^{t\sqrt{N}\Lambda_N(\phi_N)} \right]= e^{-\frac{t^2}{2}F(\mu_V)}\,, $$ 
and we deduce that $\sqrt{N}\Big(\mrm{fluct}_N(\Xi[\phi_N]) + P\zeta_N(\phi_N)\Big)$ converges towards a centered Gaussian variable with variance $-F(\mu_V) = q(\phi)$. Because $\sup_N \|\phi_N\|_{H^2(\R)}$ is finite, we can apply again Corollary \ref{cor:controleZeta} to deduce the convergence in law \eqref{eq:convergenceIntermediaire}. We now have the ingredients to conclude, by showing that the characteristic function 
\begin{equation*}
   \E^{V,P}_N\left[e^{\ii t\sqrt{N}\mrm{fluct}_N(\Xi[\phi])}\right]
\end{equation*}
converges to the characteristic function of a Gaussian variable with appropriate variance. Denoting
$$ \widetilde \Xi[\phi] (x) \defi  2P\int_\R \frac{\phi(x)-\phi(y)}{x-y}d\mu_V(y) + \phi'(x)-V'(x)\phi(x) = \Xi[\phi](x) + 2P\int_\R \mathcal{H}[\phi\rho_V](y)d\mu_V(y)\,,$$
we have 
$$ \mrm{fluct}_N(\Xi[\phi]) = \mrm{fluct}_N(\widetilde \Xi[\phi])$$
and we aim at establishing that
$$
\E^{V,P}_N\left[\exp\Big(\ii t\sqrt{N}\mrm{fluct}_N(\widetilde\Xi[\phi])\Big)\right]=\E^{V,P}_N\left[\exp\Big(\ii t\sqrt{N}\int_\R \widetilde\Xi[\phi](x) d\hat{\mu}_N(x)\Big)\right].\exp\Big(-\ii t\sqrt{N}\int_\R\widetilde\Xi[\phi](x) d\mu_V(x)\Big)  $$
converges towards the appropriate characteristic function.

By Corollary \ref{cor:edge}, the probability under $\P^{V,P}_N$ of the event $\mathcal{E}_N=\bigg\{ x_1,\ldots,x_N \in [E_N^--1,E_N^++1] \bigg\}$ converges to $1$. Along with the convergence \eqref{eq:convergenceIntermediaire}, we deduce
\begin{multline*}
    e^{-\frac{t^2}{2}q(\phi)}=\lim_{N\rightarrow\infty}\E^{V,P}_N\left[\exp\Big(\ii t\sqrt{N}\int_\R \widetilde\Xi[\phi_N](x) d\hat{\mu}_N(x)\Big)\right].\exp\Big(-\ii t\sqrt{N}\int_\R\widetilde\Xi[\phi_N](x) d\mu_V(x)\Big)
    \\= \lim_{N\rightarrow\infty}\E^{V,P}_N\left[\mathbf{1}_{\mathcal{E}_N}\exp\Big(\ii t\sqrt{N}\int_\R \widetilde\Xi[\phi_N](x) d\hat{\mu}_N(x)\Big)\right].\exp\Big(-\ii t\sqrt{N}\int_\R\widetilde\Xi[\phi_N](x) d\mu_V(x)\Big)\,,
\end{multline*}
Where we used 
$$\left|\E^{V,P}_N\left[\mathbf{1}_{\mathcal{E}_N^c}\exp\Big(\ii t\sqrt{N}\int_\R \widetilde\Xi[\phi_N](x) d\hat{\mu}_N(x)\Big)\right].\exp\Big(-\ii t\sqrt{N}\int_\R\widetilde\Xi[\phi_N](x) d\mu_V(x)\Big)\right|\leq \P^{V,P}_N(\mathcal{E}_N^c) \underset{N\rightarrow\infty}{\longrightarrow} 0\,.$$
Using that $\phi_N=\phi$ on $J_N={[E_N^--1,E_N^++1]}$,
\begin{align*}
    \int_\R \widetilde\Xi[\phi_N](x) d\mu_V(x) &= 2P\iint_{\R^2}\frac{\phi_N(x)-\phi_N(y)}{x-y}d\mu_V(x)d\mu_V(y) + \int_\R(\phi_N'-V'\phi_N)(x)d\mu_V(x) \\
    &=2P\iint_{J_N^2}\frac{\phi(x)-\phi(y)}{x-y}d\mu_V(x)d\mu_V(y) + 2P\iint_{(J_N^2)^c}\frac{\phi_N(x)-\phi_N(y)}{x-y}d\mu_V(x)d\mu_V(y) 
    \\&+ \int_{J_N} (\phi'-V'\phi)(x) d\mu_V(x) + \int_{J_N^c} (\phi\chi_N'+\phi'\chi_N-V'\phi\chi_N)(x)d\mu_V(x) \,.
\end{align*}
By boundedness of $(\|\phi_N'\|_{\infty})_N$, the second term is bounded by $$C_P\iint_{(J_N^2)^c}d\mu_V(x)d\mu_V(y) \leq 2C_P\mu_V(J_N^c) = o(N^{-1/2})\,,$$
where we used the union bound and Lemma \ref{lem:mesureE_N}. By the same estimate and the fact that $\chi_N$ can be chosen so that $(\|\chi_N'\|_{\infty})_N$ is bounded, and because $\phi'$, $V'\phi$ are bounded, the last integral term is also $o(N^{-1/2})$.
By the previous arguments, we also conclude that
$$ 2P\iint_{(J_N^2)^c}\frac{\phi(x)-\phi(y)}{x-y}d\mu_V(x)d\mu_V(y) + \int_{J_N^c}(\phi'-V'\phi)(x)d\mu_V(x) = o(N^{-1/2})\,,$$
thus
\begin{align*}
    \int_{\R} \widetilde\Xi[\phi_N](x) d\mu_V(x) = \int_\R \widetilde\Xi[\phi](x) d\mu_V(x) + o(N^{-1/2})\,,
\end{align*}
and so far we have 
\begin{equation*}
    \exp\Big(-\frac{t^2}{2}q(\phi)\Big) =\lim_{N\rightarrow\infty}\E^{V,P}_N\left[\mathbf{1}_{\mathcal{E}_N}\exp\Big(\ii t\sqrt{N}\int_\R \widetilde\Xi[\phi_N](x) d\hat{\mu}_N(x)\Big)\right].\exp\Big(-\ii t\sqrt{N}\int_\R\widetilde\Xi[\phi](x) d\mu_V(x)\Big)\,. 
\end{equation*}
Finally, using that on $\mathcal{E}_N$ the measure $\hat{\mu}_N$ is supported on $J_N$ (because each $x_i\in J_N$), using $\phi_N=\phi$ on $J_N$,
\begin{align*}
    \int_\R\widetilde\Xi[\phi_N](x)d\hat{\mu}_N(x) &= 2P\iint_{J_N^2}\frac{\phi(x)-\phi(y)}{x-y}d\mu_V(x)d\hat{\mu}_N(y) + 2P\iint_{(J_N^2)^c} \frac{\phi_N(x)-\phi_N(y)}{x-y}d\mu_V(x)d\hat{\mu}_N(y)
    \\&+ \int_{J_N} (\phi' - V'\phi)(x)d\hat{\mu}_N(x)
    \\&= 2P\iint_{\R^2} \frac{\phi(x)-\phi(y)}{x-y}d\mu_V(x)d\hat{\mu}_N(y) + \int_\R (\phi' - V'\phi)(x)d\hat{\mu}_N(x) + o(N^{-1/2})\,,
\end{align*}
Where in the second line we used, using Lemma \ref{lem:mesureE_N} again, that
$$\iint_{(J_N^2)^c} \frac{\phi_N(x)-\phi_N(y)}{x-y}d\mu_V(x)d\hat{\mu}_N(y) = \iint_{J_N \times J_N^c} \frac{\phi_N(x)-\phi_N(y)}{x-y}d\mu_V(x)d\hat{\mu}_N(y) = o(N^{-1/2})\,,$$
and the same estimate holds for $\phi_N$ replaced by $\phi$.
Therefore,
\begin{equation*}
    \exp\Big(-\frac{t^2}{2}q(\phi)\Big) = \\\lim_{N\rightarrow\infty}\E^{V,P}_N\left[\mathbf{1}_{\mathcal{E}_N}\exp\Big(\ii t\sqrt{N}\int_\R \widetilde\Xi[\phi](x) d\hat{\mu}_N(x)\Big)\right].\exp\Big(-\ii t\sqrt{N}\int_\R\widetilde\Xi[\phi](x) d\mu_V(x)\Big)\,.
\end{equation*}
This establishes that
$$ \lim_{N\rightarrow+\infty}\E^{V,P}_N\left[ \exp\Big(\ii t\sqrt{N}\mrm{fluct}_N(\Xi[\phi])\Big)\right] = \exp\Big(-\frac{t^2}{2}q(\phi)\Big)\,, $$
which concludes the proof.
\end{proof}

\begin{remark}
\label{rem:thmValidePourT}
Taking $\phi$ such that $\phi'$ satisfies the conditions of Theorem \ref{thm:LaplaceGeneral}, we then have
\begin{equation}
    \label{prooffinalthm}
    \E_N^{V,P}\left[e^{t\sqrt{N}\mrm{fluct}_N(\mathcal{L}[\phi])} \right] \underset{N\rightarrow\infty}{\longrightarrow} \exp\left\{\frac{t^2}{2}q(\phi')\right\},
\end{equation}
where the operator $\mathcal{L}$ is defined as $\mathcal{L}[\phi] \defi  \Xi[\phi']$, \textit{i.e.}
\begin{equation}
    \label{eq:premiere def L}
    \mathcal{L}[\phi](x) =2P\int_\R \frac{\phi'(x)-\phi'(y)}{x-y}d\mu_V(y)+\phi''(x) - V'(x)\phi'(x)-2P\int_\R\mcal{H}[\phi'\rho_V](y)d\mu_V(y)\,.
\end{equation}
Note that $q(\phi')=\big(\sigma_V^P\big)^2(\mcal{L}[\phi])$ 
where $\sigma_V^P$ is defined in \eqref{eq:variance}.
By Theorem \ref{thmreg}, the class of functions in $\mcal{L}^{-1}(\mcal{T})$ where
\begin{comment}
    \begin{equation}
\label{def:espaceT}
\mcal{T}\defi \left\{f\in\mcal{C}^1(\R), \exists\varepsilon>0,\,f(x)=\underset{|x|\rightarrow\infty}{O}\left(x^{-\frac{1}{2}-\varepsilon}\right),\;f'(x)=\underset{|x|\rightarrow\infty}{O}\left(x^{-\frac{1}{2}-\varepsilon}\right), \int_\R f\rho_{P}=0\right\}
\end{equation}
\end{comment}
\begin{equation}
\label{def:espaceT}
\mcal{T}\defi \left\{f\in\mcal{C}^2(\R)\, |\, f,f',f''\  \mathrm{are\ bounded}, \int_\R f(x)d\mu_V(x)=0\right\}
\end{equation}
satisfies (\ref{prooffinalthm}). This proves Theorem \ref{thm:thmppal} when noticing that $\mcal{L}^{-1}[\phi]'=\Xi^{-1}[\phi]$ and that for any $\phi\in\mcal{C}^2(\R)$ such that $\, \phi,\phi',\phi''\  \mathrm{are\ bounded}$, $\mrm{fluct}_N(\phi)=\mrm{fluct}_N(\mcal{X}[\phi])$ and $\mcal{X}[\phi]\in\mcal{T}$. Hence, we can conclude about the fluctuations of linear statistics with general $\mcal{C}^2$ functions with bounded derivatives.
\end{remark}

\section{Inversion of \texorpdfstring{$\mathcal{L}$}{TEXT}}
\label{section:spectral_theory}

This section is dedicated to the definition of $\mcal{L}$ given by \eqref{def:operateurL} and its domain and then we focus on its inversion. We rely heavily on results of Appendix \ref{app A}: the diagonalization of the operator $\mcal{A}$ by the use of the theory of Schrödinger operators.\\

We reintroduce the operators $\mathcal{A}$ and $\mathcal{W}$, acting on sufficiently smooth functions of $L^2(\mu_V)$, by
\begin{equation}
    \label{def:operateurs A et W}
    \mathcal{A}[\phi]=-\dfrac{\Big(\phi'\rho_V\Big)'}{\rho_V} = - \left(\phi''+\frac{\rho_V'}{\rho_V}\phi'\right) \quad \text{and} \quad 
     \mathcal{W}[\phi]=-\mcal{H}\big[\phi'\rho_V\big]+\int_\R\mcal{H}\big[\phi'\rho_V\big](x)d\mu_V(x)\,.
\end{equation}
One can show that the operator $\mcal{A}$ corresponds to the operator verifying:
$$\braket{\phi,\psi}_\msf{H}=\int_\R\phi'(x)\psi'(x)d\mu_V(x)=\int_\R\phi(x)\mcal{A}[\psi](x) d\mu_V(x)=\braket{\phi,\mcal{A}[\psi]}_{L^2(\mu_V)}$$
We first show the following decomposition of $\mathcal{L}$.

\begin{lemma} For $\phi$ twice differentiable we have the following pointwise identity
\begin{equation} -\mathcal{L}[\phi]=\mathcal{A}[\phi]+2P\mathcal{W}[\phi]\,.
	\end{equation}
\end{lemma}

\begin{proof}
	We write for $x\in \R$
	\begin{equation}
	2P\int_\R \frac{\phi'(x)-\phi'(y)}{x-y}\rho_V(y)dy=-2P\phi'(x)\mcal{H}[\rho_V](x)+2P\mcal{H}[\phi'\rho_V](x)\,.
	\end{equation}
	Then,
	$$\mathcal{L}[\phi] = \phi''-V'\phi'-2P\phi'\mathcal{H}[\rho_V] +2P\mathcal{H}\big[\phi'\rho_V\big]-2P\int_\R \mathcal{H}[\phi'\rho_V](y)d\mu_V(y)\,.$$
	By \eqref{deriv1} we have
	$-V'-2P\mcal{H}[\rho_V]=\frac{\rho_V'}{\rho_V}$, which concludes the proof.
\end{proof}

In order to state the next theorem, whose proof we detail in the Appendix, we introduce the following Sobolev-type spaces.
Let
$$ H^1_{V'}(\R)\defi \Big\{u\in H^1(\R),\,uV'\in L^2(\R)\Big\}\,.$$
As explained in the Appendix, $\mcal{A}$, defined in \eqref{def:operateurs A et W}, is linked to a Schrödinger operator $\mcal{S}\defi -\Delta+w_V$ via $\mcal{A}=\rho_V^{-1/2}\mcal{S}\Big[.\rho_V^{1/2}\Big]$. Here $w_V$ is defined in \eqref{eq:w_V}, continuous bounded by below and grows at infinity like $\dfrac{V'\,^2}{4}$. This operator $\mcal{S}$ is defined on the following space:
$$\mcal{D}(\mcal{S})=\Big\{u\in H^1_{V'}(\R),\sqrt{\rho_V}\mcal{A}\Big[\dfrac{u}{\sqrt{\rho_V}}\Big]\in L^2(\R)\Big\}.$$
Now, let
$$ \mathcal{D}_{L^2(\mu_V)}(\mathcal{A})\defi  \rho_V^{-1/2}\mcal{D}(\mcal{S})$$
and its homogeneous counterpart
\begin{equation}\label{def:dl2_0(A)}
     \mcal{D}_{L^2_0(\mu_V)}(\mcal{A})\defi  \Big\{u\in \mcal{D}_{L^2(\R)}(\mcal{A}),\, \int_\R u(x)d\mu_V(x) = 0 \Big\}\,.
\end{equation}
Finally, we let $L_0^2(\mu_V)$ be the subset of $L^2(\mu_V)$ of zero mean functions with respect to $\rho_V$.

We detail the proof of the following theorem in Appendix \ref{app A} which is based on Schrödinger operators theory.
\begin{theorem}[Diagonalization of $\mathcal{A}$ in $L_0^2(\mu_V)$]
\label{thm:DiagA}
There exists a sequence $0<\lambda_1<\lambda_2<\ldots $ going to infinity, and a complete orthonormal set $(\phi_n)_{n\geq 1}$ of $L_0^2(\mu_V)$ of associated eigenfunctions for $\mathcal{A}$, meaning that
\begin{itemize}
    \item $\Span\{\phi_n,\, n\geq 1\}$ is dense in $L_0^2(\mu_V)$,
    \item For all $i,j$,  $\Braket{\phi_i,\phi_j}_{L^2(\mu_V)}=\delta_{i,j}$,
    \item For all $n\geq 1$, $\mathcal{A}[\phi_n] = \lambda_n \phi_n$.
\end{itemize}
Furthermore, each $\phi_n$ is in $\mcal{D}_{L^2_0(\mu_V)}(\mcal{A})$, $\mathcal{A}:\mcal{D}_{L^2_0(\mu_V)}(\mcal{A})\to L_0^2(\mu_V)$ is bijective, and we have the writing, for $u\in L_0^2(\mu_V)$
$$ \mathcal{A}^{-1}[u] = \sum_{n\geq 1} \lambda_n^{-1} \Braket{u,\phi_n}_{L^2(\mu_V)}\phi_n\,.$$
\end{theorem}

We see the operators $\mathcal{A}$ and $\mathcal{W}$ as unbounded operators on the space 
$$ \mathsf{H}=\Big\{ u\in H^1(\mu_V)\,|\,\int_\R u(x)d\mu_V(x)=0 \Big\} $$
endowed with the inner product $\Braket{u,v}_{\mathsf{H}}=\Braket{u',v'}_{L^2(\mu_V)}$. This defines an inner product on $\mathsf{H}$ and makes it a complete space: it can be seen that $H^1(\mu_V)$ is the completion of $\mathcal{C}_c^{\infty}(\R)$ with respect to the inner product $\Braket{u,v}_{L^2(\mu_V)}+\Braket{u',v'}_{L^2(\mu_V)}$, see \cite{zhikov1998weighted}. The space $\mathsf{H}$ is then the kernel of the bounded (with respect to $\|\cdot\|_\msf{H}$) linear form, $\langle \widetilde{1}, \cdot \rangle_{L^2(\mu_V)}$ on $H^1(\mu_V)$, and both inner products are equivalent on $\mathsf{H}$ because of the Poincaré inequality, Proposition \ref{prop:poincare}. The use of $\msf{H}$ is motivated by the fact that both $\mcal{A}$ and $\mcal{W}$ are self-adjoint positive on this space as we show in Lemma \ref{lem:Wpositif}.

In the next proposition, we deduce from Theorem \ref{thm:DiagA} the diagonalization of $\mathcal{A}$ in $\mathsf{H}$.
\begin{proposition}[Diagonalization of $\mathcal{A}$ in $\mathsf{H}$]
\label{prop:diagAdansH}
With the same eigenvalues $0<\lambda_1<\lambda_2<\ldots$ as in Theorem \ref{thm:DiagA}, there exists a complete orthonormal set $(\varphi_n)_{n\geq 1}$ of $\mathsf{H}$ formed by eigenfunctions of $\mathcal{A}$. %Furthermore, similarly as in Proposition \ref{prop:diagoL2}, we have the writing, for $u\in \mathsf{H}$
%$$ \mathcal{A}^{-1}u = \sum_{n\geq 1} \lambda_n^{-1} \Braket{u,\psi_n}_{\mathsf{H}}\psi_n\,.$$
\end{proposition}

\begin{proof}
With $(\phi_n)_{n\geq 1}$ of Theorem \ref{thm:DiagA},
    \begin{align*}
        \delta_{i,j}=\langle \phi_i,\phi_j \rangle_{L^2(\mu_V)} &= \frac{1}{\lambda_j}\langle \phi_i,\mathcal{A}[\phi_j] \rangle_{L^2(\mu_V)}\\
                              &=\frac{1}{\lambda_j}\langle \phi_i',\phi_j' \rangle_{L^2(\mu_V)}\\
                     &= \frac{1}{\lambda_j}\langle \phi_i,\phi_j \rangle_{\mathsf{H}}.
    \end{align*}
    With $\varphi_n=\frac{1}{\sqrt{\lambda_n}}\phi_n$, $(\varphi_n)_{n\geq 1}$ is then orthonormal with respect to the inner product of $\mathsf{H}$. 
    To show that $\Span\{\varphi_n,n\geq 1\}$ is dense in $\mathsf{H}$, let $u\in\msf{H}$ be such that for all $j\geq 1$, $\langle u,\phi_j\rangle_{\msf{H}}=0$. In the last series of equalities, replace $\phi_i$ by $u$: we see that $u$ is orthogonal to each $\phi_j$ in $L^2(\mu_V)$, thus $u$ is a constant as shown in the proof of Lemma \ref{lemma:ortho}, and because $u\in \mathsf{H}$ it has zero mean against $\rho_V$. This shows that $u=0$.
\end{proof}

We set for what follows $\mathcal{D}(\mathcal{A}) \defi  \left\{ u\in \mathsf{H}\,|\, \mathcal{A}[u]\in \mathsf{H}\right\}$ and $\mathcal{D}(\mathcal{W}) \defi  \left\{ u\in \mathsf{H}\,|\, \mathcal{W}[u]\in \mathsf{H}\right\}$.
\begin{remark}
    One can notice that $\mcal{D}(\mcal{A})\subset\mcal{D}(\mcal{W})$. Indeed, it just suffices to see that whenever $u\in\mcal{D}(\mcal{A})$, then $$\mcal{W}[u]'=\mcal{H}\big[\rho_V\mcal{A}[u]\big]\in L^2(\mu_V).$$
    Since for all $u\in\msf{H}$, $\mcal{W}[u]\in L^2_0(\mu_V)$, then $\mcal{W}[u]\in\msf{H}$.
\end{remark}

\begin{lemma}
\label{lem:Wpositif}
The following properties hold:
\begin{itemize}
    \item The operator $\mathcal{W}:\mathcal{D}(\mathcal{W})\to \mathsf{H}$ is positive: for all $\phi\in \mathcal{D}(\mathcal{W})$,
$$\langle \mathcal{W}[\phi],\phi \rangle_{\mathsf{H}}=\dfrac{1}{2}\|\phi'\rho_V\|_{1/2}^2\geq 0 \,,$$
with equality only for $\phi=0$, where the $1/2$-norm of $u$ is given by 
$$ \|u\|_{1/2}^2=\int_\R |x|.\left|\mathcal{F}[u](x)\right|^2dx\,.$$
    \item Both $\mathcal{A}$ and $\mathcal{W}$ are self-adjoint for the inner product of $\mathsf{H}$.
\end{itemize}

\end{lemma}

\begin{proof}
To prove the first point, let $\phi\in\mathcal{D}(\mathcal{W})$. Then,
\begin{multline*}
2\pi\Braket{\mcal{W}[\phi],\phi}_\msf{H}=-2\pi\Braket{\mcal{H}[\phi'\rho_V]',\phi'\rho_V}_{L^2(\R)}=-\Braket{ix\mcal{F}\big[\mcal{H}[\phi'\rho_V]\big],\mcal{F}[\phi'\rho_V]}_{L^2(\R)}\\=\pi\Braket{|x|\mcal{F}[\phi'\rho_V],\mcal{F}[\phi'\rho_V]}_{L^2(\R)}=\pi\|\phi'\rho_V\|_{1/2}^2\geq 0\,,
\end{multline*}
and because $\phi$ is in $\mathsf{H}$, this last quantity is zero if and only if $\phi$ vanishes.\\

For the second point, note that  both $\mcal{D}(\mcal{A})$ and $\mcal{D}(\mcal{W})$ contain $\mathrm{span}(\varphi_n,n\geq 1)$, so $\mcal{A}$ and $\mcal{W}$ are densely defined. Let $i,j\in \N$, then by orthonormality
$$\braket{\mcal{A}[\varphi_i],\varphi_j}_\msf{H}=\delta_{i,j}\lambda_i=\braket{\varphi_i,\mcal{A}[\varphi_j]}_\msf{H}$$ hence $\mcal{A}$ is symmetric.
Furthermore, we notice that $\mcal{D}(\mcal{A})=\{u\in\msf{H},\sum_{n\geq1}\lambda_n^2|\braket{u,\varphi_n}_\msf{H}|^2<+\infty\}$. Indeed, we extend $\mcal{A}$ on the space of the RHS by for all $v\in\{u\in\msf{H},\sum_{n\geq1}\lambda_n^2|\braket{u,\psi_n}_\msf{H}|^2<+\infty\}$, $$\mcal{A}[v]\defi \sum_{n\geq1}\lambda_n\braket{v,\psi_n}_\msf{H}\psi_n.$$
It is then clear that $\mcal{A}[v]\in\msf{H}$. The converse is straightforward. Moreover, it is standard that $\mcal{D}(\mcal{A})$ is therefore the completion of span $(\varphi_n,n\geq1)$ under the graph norm $\|u\|_\mcal{A}^2\defi \|u\|^2_\msf{H}+\|\mcal{A}[u]\|_\msf{H}^2$ which makes $\big(\mcal{A},\mcal{D}(\mcal{A})\big)$ a closed operator. Moreover, $\big(\mcal{A},\mcal{D}(\mcal{A})\big)$ is self-adjoint since for $v\in\mcal{D}(\mcal{A}^*)$, which is defined by:
$$\mcal{D}(\mcal{A}^*)\defi \Big\{v\in\msf{H}, \varphi\in\mcal{D}(\mcal{A})\mapsto\braket{\mcal{A}[\varphi],v}_\msf{H}\;\mrm{can}\;\mrm{be}\;\mrm{extended}\;\mrm{into}\;\mrm{a}\;\mrm{continuous}\;\mrm{linear}\;\mrm{form}\;\mrm{on}\;\msf{H}\Big\},$$ there exists, by Riesz theorem, an element $f_v\in\msf{H}$, such that for all $\varphi\in\mcal{D}(\mcal{A})$,
$$ \braket{\mcal{A}[\varphi],v}_\msf{H}=\braket{\varphi,f_v}_{\msf{H}}\,.$$
Thus for all $n\in\N^*$,
$$\lambda_n\braket{\varphi_n,v}_\msf{H}=\braket{\varphi_n,f_v}_\msf{H}.$$
The right hand side defines a sequence in $\ell^2(\N^*)$ because $f_v\in\msf{H}$ and the fact that $(\varphi_n)_n$ is a Hilbertian basis, thus $\sum_{n\in\N}\lambda_n^2|\braket{v,\psi_n}_\msf{H}|^2<+\infty$ which is the definition of $v\in\mcal{D}(\mcal{A})$. This concludes the fact that $\mcal{A}:\mcal{D}(\mcal{A})\rightarrow\msf{H}$ is self-adjoint.

\begin{comment}
\red{(Agreed !)}
Using the symmetry of $\mathcal{A}$ on $\msf{H}$, we get
$$\braket{u,\mcal{A}[v]}_\msf{H}=\braket{u'\rho,\mcal{A}[v]'}_{L^2(\R)}=\braket{u'\rho,f_v'}_{L^2(\R)}\,.$$
%$$\braket{(\mcal{A}u)',\rho_Vv'}_{L^2(\R)}=\braket{u',\rho_Vf_v'}_{L^2(\R)}=-\braket{(u'\rho_V)',f_v}_{L^2(\R)}.$$
\blue{This equation also holds for general functions $u\in\mcal{C}_c^\infty(\R)$ by setting $\mcal{A}[u]\defi \mcal{A}[u-\int_\R u(x)d\mu_V(x)]$.} \red{(but $u-\int_\R u(x)d\mu_V(x)$ is not complactly supported)} This implies that, in the sense of distributions, $\mcal{A}[u]'=f'_v\in L^2(\mu_V)$ and therefore $\mcal{A}[v]=f_v+C$ for a constant $C$. Finally, since we necessarily have, $\int_\R\mcal{A}[v]d\mu_V=0$, $\mcal{A}[v]\in\msf{H}$ \textit{ie} $v\in\mcal{D}(\mcal{A})$.
\end{comment}

Now let $u,v\in \mathcal{D}(\mathcal{W})$, using Plancherel's isometry and $\textit{i)}$ of Lemma \ref{lemma:HilbertPropriétés}, we see that
\begin{align*}
    \Braket{\mathcal{W}[u],v}_{\mathsf{H}} = \Braket{\mathcal{W}[u]',v'\rho_V}_{L^2(\R)}=\frac{1}{2}\Braket{|x|\mathcal{F}[u'\rho_V],\mathcal{F}[v'\rho_V]}_{L^2(\R)}\,,
\end{align*}
and this last expression is symmetric in $(u,v)$, so $\mcal{W}$ is symmetric. To see that $\mcal{W}$ is self-adjoint, one must show that $\mcal{D}(\mcal{W}^*)\subset\mcal{D}(\mcal{W})$. Let $v\in\mcal{D}(\mcal{W}^*)$, then the map $\varphi\in\mcal{D}(\mcal{W})\mapsto\braket{\mcal{W}[\varphi],v}_\msf{H}$ can be extended to a continuous linear functional on $\msf{H}$, so there exists $f_v\in\msf{H}$ such that for all $\varphi\in\mcal{D}(\mcal{W})$,
\begin{equation}\label{eq:Wauto-adjoint}
    \braket{-\mcal{H}[\varphi'\rho_V]',v'\rho_V}_{L^2(\R)}=\braket{\varphi'\rho_V,f_v'}_{L^2(\R)}.
\end{equation}
Let $w\in\mcal{C}_c^\infty(\R)$, we define $$\varphi_w(x)\defi\int_{-\infty}^x\dfrac{w(t)}{\rho_V(t)}dt-\int_\R d\mu_V(y)\int_{-\infty}^y\dfrac{w(t)}{\rho_V(t)}dt.$$ It is easy to check that $\varphi_w\in\msf{H}$ and that $$\mcal{W}[\varphi_w]=-\mcal{H}[w]+\int_\R\mcal{H}[w](t)d\mu_V(t)\in\msf{H},$$
therefore $\varphi_w\in \mathcal{D}(\mathcal{W})$.
Finally, by inserting the definition of $\varphi_w$ in \eqref{eq:Wauto-adjoint}, we obtain:
\begin{equation*}
    \braket{-\mcal{H}[w'],v'\rho_V}_{L^2(\R)}=\braket{w,f_v'}_{L^2(\R)}.
\end{equation*}
By using that $\mcal{H}$ is skew self-adjoint, we obtain that, in the sense of distributions, $\mcal{W}[v]'=-\mcal{H}[v'\rho_V]'=-f_v'\in L^2(\mu_V)$. Along with the fact that for all $\widetilde{v}\in\msf{H}$, $\mcal{W}[\widetilde{v}]\in L^2_0(\mu_V)$, we deduce that $\mcal{W}[v]\in\msf{H}$, therefore that $v\in\mcal{D}(\mcal{W})$. This allows to conclude that $\mcal{D}(\mcal{W}^*)\subset\mcal{D}(\mcal{W})$, hence that $\mcal{W}$ is self-adjoint.
This concludes the proof.
 
\end{proof}

\begin{definition}[Quadratic form associated to $-\mcal{L}$]
We define for all $u,v\in \msf{H}\cap\mcal{C}_c^\infty(\R)$ the quadratic form associated to $-\mcal{L}$ by
$$q_{-\mcal{L}}(u,v)=\braket{\mcal{A}[u],\mcal{A}[v]}_{L^2(\mu_V)}+P\braket{\mcal{F}[u'\rho_V],\mcal{F}[v'\rho_V]}_{L^2(|x|dx)}\,.$$
\begin{comment}
\sout{\blue{It is possible to extend the domain of this quadratic form also define to the so-called form domain of $\mathcal{L}$ given by the completion of $\msf{H}\cap\mcal{C}_c^\infty(\R)$ under the 
norm $u\mapsto q_{-\mcal{L}}(u,u)$ on $\mcal{D}_{L^2_0(\mu_V)}(\mcal{A})$ given in \eqref{def:dl2_0(A)}.}
$$Q(\mathcal{L})= \Big\{ u\in \mathsf{H}, q_{-\mathcal{L}}(u,u)<+\infty \Big\}\,.$$}
\end{comment}
\end{definition}
\begin{remark}
    Note that for all $u,v\in \msf{H}\cap\mcal{C}_c^\infty(\R)$, $q_{-\mcal{L}}(u,v)=\braket{-\mcal{L}[u],v}_\msf{H}$ and that whenever $u\in\mcal{D}(\mcal{A})$,
\begin{equation}
    \label{cont:qL}
    q_{-\mcal{L}}(u,u)=\braket{\mcal{A}[u],u}_\msf{H}+2P\braket{\mcal{W}[u],u}_\msf{H}\geq\lambda_1(\mcal{A})\|u\|_\msf{H}^2
\end{equation}
 by Proposition \ref{prop:diagAdansH} and Lemma \ref{lem:Wpositif}. We can now extend $q_{-\mcal{L}}$ to its form domain $$Q(\mcal{L})\defi\Big\{u\in\msf{H},\mcal{A}[u]\in L^2(\mu_V),\,\mcal{F}[u'\rho_V]\in L^2(|x|dx)\Big\}=\mcal{D}_{L^2_0(\mu_V)}(\mcal{A}).$$
\end{remark}
 We now define $\mcal{D}(\mcal{L})$ the domain of definition of $-\mcal{L}$ by:
$$\mcal{D}(\mcal{L})\defi \Big\{v\in Q(\mcal{L})\ |\ \varphi\in Q(\mcal{L})\mapsto q_{-\mcal{L}}(\varphi,v)\text{ can be extended into a continuous linear functional on }\msf{H}\Big\}.$$

\begin{proposition}
    $\mcal{D}(\mcal{L})=\mcal{D}(\mcal{A}).$
\end{proposition}

\begin{proof}
    Let $v\in\mcal{D}(\mcal{L})$. By Riesz theorem, there exists $f_v\in\msf{H}$ such that for all  $\varphi\in Q(\mcal{L})$, $q_{-\mcal{L}}(\varphi,v)=\braket{u,f_v}_\msf{H}$. We set $-\mcal{L}[v]\defi f_v$, it is called the Friedrichs extension of $-\mcal{L}$. Then for all $\varphi\in \mcal{D}(\mcal{A})\subset\mcal{Q}(\mcal{L}),$
$$ \braket{\varphi,-\mcal{L}[v]}_\msf{H}=q_{-\mcal{L}}(\varphi,v)=\braket{\mcal{A}[\varphi],v}_\msf{H}+2P\braket{\mcal{W}[\varphi],v}_\msf{H}. $$
Since $v\in\mcal{D}(\mcal{L})\subset Q(\mcal{L})$, then $\mcal{W}[v]\in\msf{H}$, hence $\braket{\mcal{W}[\varphi],v}_\msf{H}=\braket{\varphi,\mcal{W}[v]}_\msf{H}$. We deduce that,
$$\braket{\mcal{A}[\varphi],v}_\msf{H}=\braket{\varphi,-\mcal{L}[v]-2P\mcal{W}[v]}_\msf{H}$$
Taking $\varphi=\phi_n$, one deduces that:
$$\lambda_n\braket{\phi_n,v}_\msf{H}=\braket{\phi_n,-\mcal{L}[v]-2P\mcal{W}[v]}_\msf{H}.$$
Since $-\mcal{L}[v]-2P\mcal{W}[v]\in\msf{H}$, the sequence $\Big(\lambda_n\braket{\phi_n,v}_\msf{H}\Big)_n=\Big(\braket{\phi_n,-\mcal{L}[v]-2P\mcal{W}[v]}_\msf{H}\Big)_n\in\ell^2(\N^*)$ which is the definition of $v\in\mcal{D}(\mcal{A})$. Hence $\mcal{D}(\mcal{L})\subset\mcal{D}(\mcal{A})$. Conversely, if $v\in\mcal{D}(\mcal{A})$, then for all $\varphi\in Q(\mathcal{L})$,
$$q_{-\mcal{L}}(\varphi,v)=\braket{\varphi,\mcal{A}[v]+2P\mcal{W}[u]}_\msf{H}$$
which provides a continuous extension of $\varphi\in Q(\mcal{L})\mapsto q_{-\mcal{L}}(\varphi,v)$ to $\msf{H}$, thus $v\in\mcal{D}(\mcal{L})$.

\begin{comment}
\blue{It is easy to see first that, whenever $u\in\mcal{D}(\mcal{A})$,}
hence \blue{just as before this equation only depends on the derivative of $v$ and hence can be extended to $\mcal{C}_c^\infty(\R)$.} Thus we deduce the distributional identity $-\mcal{L}[u]=\mcal{A}[u]+2P\mcal{W}[u]$. Since $u\in\msf{H}$, $u'\rho_V\in L^2(\R)$, hence $\mcal{H}[u'\rho_V]\in L^2(\R)$ so $\mcal{W}[u]\in L^2(\R)$. Furthermore $u\in\mcal{D}(\mcal{L})\subset Q(\mcal{L})=\mcal{D}_{L^2_0(\mu_V)}(\mcal{A})$,
so $\mcal{W}[u]'=\mcal{H}\big[\rho_V\mcal{A}[u]\big]\in L^2(\R)$ which implies that  $\mcal{W}[u]\in\msf{H}$. Finally, since $\mcal{L}[u]\in\msf{H}$, $\mcal{A}[u]\in \msf{H}$ so $u\in\mcal{D}(\mcal{A})$.
Similarly, if $u\in\mcal{D}(\mcal{A})$, for the same reasons as before $\mcal{W}[u]\in\msf{H}$ thus $\mcal{L}[u]=-\mcal{A}[u]-2P\mcal{W}[u]\in\msf{H}$ \textit{ ie } $u\in\mcal{D}(\mcal{L})$.
\end{comment}

\end{proof}

We are now ready to state the main theorem of this section, that is the inversion of $\mathcal{L}$ on $\mcal{D}(\mcal{L})$.

\begin{theorem}[Inversion of $\mcal{L}$]\label{thm:InverL} $-\mcal{L}:\mcal{D}(\mcal{L})\longrightarrow\msf{H}$ is bijective. Furthermore, $(-\mcal{L})^{-1}$ is positive, continuous from $(\msf{H},\|.\|_\msf{H})$ to $(\mcal{D}(\mcal{L}),\|.\|_\msf{H})$. More precisely, for all $f\in\msf{H}$,
$$\|\mcal{L}^{-1}[f]\|_\msf{H}\leq\lambda_1(\mcal{A})^{-1}\|f\|_\msf{H}.$$
\end{theorem}

\begin{proof}
    Let $f\in\msf{H}$, since $\braket{f,.}_\msf{H}$ is a linear form on $Q(\mcal{L})=\mcal{D}_{L^2_0(\mu_V)}(\mcal{A})$ which is, by (\ref{cont:qL}), continuous with respect to $q_{-\mcal{L}}$, one can applies Riesz theorem so there exists a unique $u_f\in \mcal{D}_{L^2_0(\mu_V)}(\mcal{A})$, such that for all $v\in Q(\mcal{L})$, $\braket{f,v}_\msf{H}=q_{-\mcal{L}}(u_f,v)$. Since, $u_f$ is clearly in $\mcal{D}(\mcal{L})$ by definition of the Friedrichs extension of $-\mcal{L}$, we have $-\mcal{L}[u_f]=f$.
    Finally, the continuity of $\mcal{L}^{-1}$ with respect to $\|.\|_\msf{H}$ follows from \eqref{cont:qL}.
\end{proof}

\begin{remark}
    We can diagonalize $\mcal{L}$ by the same argument we used in Appendix \ref{app A} to diagonalize $\mcal{A}$ in $L_0^2(\mu_V)$.
\end{remark}

We now prove, following Remark \ref{rem:thmValidePourT}, a more compact formula for the variance such as the one appearing in \cite{hardy2021clt}. With $\mathcal{T}$ defined in \eqref{def:espaceT}, we have

\begin{lemma}
\label{lem:formuleVarianceL}
	The following equality holds for all $\phi\in\mcal{T}$
		\begin{multline}
			\braket{\mcal{L}^{-1}[\phi],\phi}_\msf{H}=\big(\sigma_V^P\big)^2(\phi)\defi \int_\R\Bigg(\mcal{L}^{-1}[\phi]''(x)^2+V''(x)\mcal{L}^{-1}[\phi]'(x)^2\Bigg)d\mu_V(x)
			\\+P\iint_{\R^2}\Bigg(\frac{\mcal{L}^{-1}[\phi]'(x)-\mcal{L}^{-1}[\phi]'(y)}{x-y}\Bigg)^2d\mu_V(x)d\mu_V(y)
		\end{multline}
\end{lemma}

\begin{proof}
	It suffices to show that $\big(\sigma_V^P\big)^2(\mcal{L}[\phi])=\braket{\mcal{L}[\phi],\phi}_\msf{H}$ for all $\phi\in \mathcal{D}(\mathcal{L})$.
			$$\braket{\mcal{L}[\phi],\phi}_\msf{H}=-\int_\R\Big(\dfrac{(\phi'\rho_V)'}{\rho_V}\Big)'(x)\phi'(x)d\mu_V(x)-2P\int_\R\mcal{H}[\phi'\rho_V]'(x)\phi'(x)d\mu_V(x)$$
Proceeding to integration by parts in the first integral leads to 
\begin{align*}
-\int_\R\Big(\dfrac{(\phi'\rho_V)'}{\rho_V}\Big)'(x)\phi'(x)d\mu_V(x)=\int_\R\Big(\dfrac{(\phi'\rho_V)'}{\rho_V}\Big)^2(x)d\mu_V(x)
	=\int_\R\Big[\phi''^2-\phi'^2\dfrac{\rho_V''}{\rho_V}+\phi'^2\big(\dfrac{\rho_V'}{\rho_V}\big)^2\Big](x)d\mu_V(x).
\end{align*}
Since $$\dfrac{\rho_V''}{\rho_V}=\Big(-V''-2P\mcal{H}[\rho_V]'+V'^2+4P^2\mcal{H}[\rho_V]^2+4PV'\mcal{H}[\rho_V]\Big)=-V''-2P\mcal{H}[\rho_V]'+\Big(\dfrac{\rho_V'}{\rho_V}\Big)^2,$$
we obtain
$$\braket{\mcal{L}[\phi],\phi}_\msf{H}
=\int_\R\big(\phi''^2+V''\phi'^2\big)(x)d\mu_V(x)-2P\int_\R\mcal{H}[\phi'\rho_V]'(x)\phi'(x)d\mu_V(x)+2P\int_\R\mcal{H}[\rho_V]'(x)\phi'^2(x)d\mu_V(x).$$
We conclude using the fact that for all $\phi\in \mathcal{D}(\mathcal{L})$, we have
$$\iint_{\R^2}\Big(\dfrac{\phi'(x)-\phi'(y)}{x-y}\Big)^2d\mu_V(x)d\mu_V(y)=2\int_\R\big(\mcal{H}[\rho_V]'\phi'^2-\mcal{H}[\phi'\rho_V]'\phi'\big)(x)d\mu_V(x).$$
\end{proof}

    We now state a result that could allow one to obtain an explicit formulation for $\mcal{L}^{-1}$ by the help of Fredholm determinant theory for Hilbert-Schmidt operators. The reader can refer to \cite{gohberg2012traces} for more details.
\begin{definition}[Fredholm determinant]
    Let $\mcal{K}$ be a kernel Hilbert-Schmidt operator on $L^2(\mu_V)$, we define the Fredholm 2-determinant of $\msf{id}+\mcal{K}$ by
    $$\det_2\left(\msf{id}+\mcal{K}\right)\defi 1+\sum_{n\geq1}\dfrac{1}{n!}\int_{\R^{n}}\begin{vmatrix}
		0&k(t_1,t_2)&\dots&k(t_1,t_n)
		\\k(t_2,t_1)&0&\dots&k(t_2,t_n)
		\\\vdots& & &\vdots
		\\k(t_n,t_1)&k(t_n,t_2)&\hdots&0
	\end{vmatrix}\prod_{i=1}^{n}d\mu_{V}(t_i).$$
    where $|.|$ is the usual determinant for matrices.
\end{definition}

\begin{theorem}[Explicit expression for $\mcal{L}^{-1}$]

    For all $u\in \msf{H}$, such that $x\mapsto \dfrac{1}{\rho_V(x)}\displaystyle\int_x^{+\infty}u(t)\rho_V(t)dt$ is integrable at $+\infty$, we have:
    \begin{comment}
    \begin{equation}
        \mcal{L}^{-1}u=\mcal{A}^{-1}u-\rho_V^{-1/2}\mcal{R}\big[\rho_V^{1/2}\mcal{A}^{-1}u\big]
    \end{equation}
    \end{comment}
    \begin{equation}
        -\mcal{L}^{-1}[u]=\mcal{A}^{-1}[u]-\mcal{A}^{-1}\mcal{R}\big[u\big].
    \end{equation}
    In this expression $\mcal{A}^{-1}[u]$ is given by 
    $$\mcal{A}^{-1}[u](x)=-\int_x^{+\infty}\dfrac{ds}{\rho_V(s)}\int_s^{+\infty}u(t)\rho_V(t)dt+C$$
    where $C=\displaystyle\int_\R\rho_V(y)dy\int_y^{+\infty}\dfrac{ds}{\rho_V(s)}\int_s^{+\infty}u(t)\rho_V(t)dt$. The operator $\mcal{R}$ is the kernel operator defined for all $v\in L^2(\mu_V)$ by:
    \begin{equation*}
        \mcal{R}[v](x)=\int_\R r(x,y)v(y)d\mu_V(y)
    \end{equation*}
    where 

    \begin{comment}
    $$R(x,y)=\displaystyle\dfrac{1}{\det(I+\mcal{K})}\sum_{n\geq0}\dfrac{1}{n!}\int_{\R^n}\det_{n+1}\begin{bmatrix}
        K(x,y)&K(x,\lambda_b)
        \\K(\lambda_a,y)&K(\lambda_a,\lambda_b)
    \end{bmatrix}_{a,b=1\dots n}d\lambda_1\dots d\lambda_n$$
    \end{comment}

    \begin{equation}\label{eq:kernel R}
        r(x,y)=\dfrac{1}{\underset{2}{\det}(\msf{id}-\mcal{K})}\sum_{n\geq1}\dfrac{(-1)^n}{n!}\int_{\R^{n}}\begin{vmatrix}
		k(x,y)&k(x,t_1)&\dots&k(x,t_n)
		\\k(t_1,y)&0&\dots&k(t_1,t_n)
		\\\vdots& & &\vdots
		\\k(t_n,y)&k(t_n,t_1)&\hdots&0
	\end{vmatrix}\prod_{i=1}^{n}d\mu_{V}(t_i),
    \end{equation}
where $\mcal{K}$ is the kernel operator defined for all $w\in L^2(\mu_V)$ by:
\begin{equation}
    \label{def:oper K}
    \mcal{K}[w](x)=\int_\R k(x,y)w(y)dy
\end{equation}
with 
\begin{equation}
    \label{def:kernel K}k(x,y)=2P\ln|x-y|-2P\displaystyle\int_\R\ln|z-y|d\mu_V(z).
\end{equation}
Finally, $2P\mcal{W}\circ\mcal{A}^{-1}=-\mcal{K}$.
\end{theorem}
\begin{proof}
    Let $f\in\msf{H}$, there exists a unique $u\in\mcal{D}(\mcal{A})$ such that $\mcal{A}[u]=f$. Since $(u'\rho_V)'=-\rho_V\mcal{A}[u]\in L^2(\R)$, we get $u'\rho_V\in H^1(\R)$, so $u'(x)\rho_V(x)\underset{|x|\rightarrow+\infty}{\longrightarrow}0$. By definition, $-\dfrac{(u'\rho_V)'}{\rho_V}=f\text{ hence }$
       \begin{equation}\label{formule:dérivée}
        \mcal{A}^{-1}[f]'(x)\rho_V(x)=u'(x)\rho_V(x)=\int_x^{+\infty}f(t)\rho_V(t)dt.
    \end{equation}
    Using the fact that $\int_\R u(x)\rho_V(x)dx=0$, integrating again we get:
 
    $$u(x)=-\int_x^{+\infty}\dfrac{ds}{\rho_V(s)}\int_s^{+\infty}f(t)\rho_V(t)dt+C$$
    where $C=\displaystyle\int_\R\rho_V(x)dx\int_x^{+\infty}\dfrac{ds}{\rho_V(s)}\int_s^{+\infty}f(t)\rho_V(t)dt$.
    Now using that that for all $g\in \msf{H}$ $$-\mcal{L}^{-1}=\Big(\mcal{A}+2P\mcal{W}\Big)^{-1}=\mcal{A}^{-1}\Big(\msf{id}+2P\mcal{W}\circ\mcal{A}^{-1}\Big)^{-1}$$
    We only need to compute the expression of $\Big(\msf{id}+2P\mcal{W}\circ\mcal{A}^{-1}\Big)^{-1}$.
    %let $g\in\msf{H}$, there exists a unique $v\in\mcal{D}(\mcal{L})$, such that $-\mcal{L}v=\mcal{A}v+2P\mcal{W}v=g$ and then $v+2P\mcal{W}\mcal{A}^{-1}v=\mcal{A}^{-1}g$.
    Using (\ref{formule:dérivée}), let $v\in\mcal{A}\big(\mcal{D}(\mcal{L})\big)=\mathsf{H}$. We obtain:
    \begin{equation*}
        \mcal{W}\circ\mcal{A}^{-1}[v](x)=-\fint_\R\dfrac{ds}{s-x}\int_s^{+\infty}dtv(t)\rho_V(t)+\int_\R d\mu_V(y)\fint_\R\dfrac{ds}{s-y}\int_s^{+\infty}dtv(t)\rho_V(t).
    \end{equation*}
    By Sokhotski-Plejmel formula, we have:
    $$\fint_\R\dfrac{ds}{s-x}\int_s^{+\infty}dtv(t)\rho_V(t)=\lim_{\varepsilon\rightarrow{0}^+}\underset{M\rightarrow+\infty}{\lim}\int_{-M}^M\dfrac{ds}{2}\Big\{\dfrac{1}{s-x+\ii\varepsilon}+\dfrac{1}{s-x-\ii\varepsilon}\Big\}\int_s^{+\infty}dtv(t)\rho_V(t).$$
    We then proceed to an integration by part:
    \begin{multline*}
        \fint_\R\dfrac{ds}{s-x}\int_s^{+\infty}dtv(t)\rho_V(t)=\lim_{\varepsilon\rightarrow{0}^+}\underset{M\rightarrow+\infty}{\lim}\Big[\dfrac{\ln\big((s-x)^2+\varepsilon^2\big)}{2}\int_s^{+\infty}dtv(t)\rho_V(t)\Big]_{-M}^M\\+\int_\R ds\ln|x-s|v(s)\rho_V(s).
    \end{multline*}
    To conclude that $\mcal{W}\circ\mcal{A}^{-1}[v](x)=-\displaystyle\int_\R ds\ln|x-s|v(s)\rho_V(s)+\int_\R d\mu_V(y)\int_\R ds\ln|y-s|v(s)\rho_V(s)$, we just need to show that $$\ln|x|\int_x^{+\infty}dtv(t)\rho_V(t)\underset{|x|\rightarrow\infty}{\longrightarrow}0$$ which can be seen by Cauchy-Schwarz inequality:
    $$\Big|\ln(x)\int_x^{+\infty}dt v(t)\rho_V(t)\Big|\leq|\ln(x)|\|v\|_{L^2(\mu_V)}.\rho_V(x)^{1/4}\Big(\int_\R\rho_V(t)^{1/2}dt\Big)^{1/2}.$$
    In this inequality, we used that $\rho_V$ is decreasing exponentially fast in a neighborhood of $+\infty$, hence $$\ln(x)\int_x^{+\infty}dt v(t)\rho_V(t)\underset{x\rightarrow+\infty}{\longrightarrow}0.$$ 
    Using that $\displaystyle\int_\R v(t)\rho_V(t)dt =0$, we have $\displaystyle \int_x^{+\infty} v(t)\rho_V(t)dt=-\int_{-\infty}^x  v(t)\rho_V(t)dt$, therefore the exact same argument allows us to conclude when $x$ goes to $-\infty$. \begin{comment}Using the fact that $\int_\R v(t)\rho_V(t)dt=0$ \blue{and thus that $\int_x^{+\infty}v(t)\rho_V(t)dt=-\int_{-\infty}^xv(t)\rho_V(t)dt$},\end{comment} 
    We obtain the following equality:
    \begin{equation*}
       (\msf{id}+2P\mcal{W}\circ\mcal{A}^{-1})[v](x)=v(x)-2P\int_\R ds\ln|x-s|v(s)\rho_V(s)+2P\int_\R d\mu_V(y)\int_\R ds\ln|y-s|v(s)\rho_V(s).
    \end{equation*}
    \begin{comment}
    \blue{Anciennement: Now setting  $\Tilde{v}(t)=\rho_V^{1/2}(t)v(t)$ and $\tilde{h}=\rho_V^{1/2}(t)h(t)$, we obtain $\tilde{v}+\mcal{K}[\tilde{v}]=\Tilde{h}$ where $\mcal{K}$ is defined in (\ref{def:oper K}). Since its kernel (defined in (\ref{def:kernel K})) belongs to $L^2(\R^2)$, $\mcal{K}$ is Hilbert-Schmidt. Hence by Fredholm determinant theory: 
    $$\Tilde{v}=\Tilde{h}-\mcal{R}[\tilde{h}]$$ or $\mcal{L}^{-1}g=\mcal{A}^{-1}g-\rho_V^{-1/2}\mcal{R}\big[\rho_V^{1/2}\mcal{A}^{-1}g\big]$ as expected.}
    \end{comment}
    With the operator $\mcal{K}$ given in \eqref{def:oper K}, by \cite[Section XII.2]{gohberg2012traces}, we can conclude that the inverse of the kernel operator $\msf{id}-\mcal{K}$ is $\msf{id}-\mcal{R}$ where $\mcal{R}$ is given in \eqref{eq:kernel R}. This concludes the proof.
\end{proof}

\section{Regularity of the inverse of \texorpdfstring{$\mathcal{L}$}{TEXT} and completion of the proof of Theorem \ref{thm:thmppal}}
\label{s7}
Since we have proven the central limit theorem for functions of the type $\mcal{L}[\phi]$ with $\phi$ regular enough and satisfying vanishing asymptotic conditions at infinity, we exhibit a class of functions $f$ for which $\mcal{L}^{-1}[f]$ is regular enough to satisfy conditions of Theorem \ref{thm:LaplaceGeneral}.
 Define the subset $\mcal{T}$ of $\msf{H}$ by
$$\mcal{T}\defi \left\{f\in\mcal{C}^2(\R)\,|\,f, f', f'' \text{ are bounded}, \int_\R f(x)d\mu_V(x)=0\right\}.$$

\begin{theorem}\label{thmreg}
For $f\in \mcal{T}$, there exists a unique $u\in \mcal{C}^3(\R)$ such that $u'\in H^2(\R)$ with $u^{(3)}$ bounded which verifies:
 \begin{itemize}
     \item $u'(x)=\underset{|x|\rightarrow\infty}{O}\left(V'(x)^{-1}\right)$
     \item $u''(x)=\underset{|x|\rightarrow\infty}{O}\left(V'(x)^{-1}\right)$
     \item $u^{(3)}(x)=\underset{|x|\rightarrow\infty}{O}\left(V'(x)^{-1}\right)$
 \end{itemize}
 and such that $f=\mcal{L}[u]$.
\end{theorem}

\begin{proof}
	Let $f\in \mcal{T}\subset \msf{H}$, then since $-\mcal{L}:\mcal{D}(\mcal{L})\rightarrow\msf{H}$ is bijective, there exists a unique $u\in \mcal{D}(\mcal{L})$ such that $-\mcal{L}[u]=f$ \textit{i.e}:
	\begin{equation}\label{equation resolution}
		-u''-\dfrac{\rho_V'}{\rho_V}u'-2P\mcal{H}[u'\rho_V]+2P\int_\R\mcal{H}[u'\rho_V](t)d\mu_V(t)=f
	\end{equation}
	Hence we have
	\begin{equation}\label{premiere derivée}
		-(u'\rho_V)'=\rho_V\Big(f+2P\mcal{H}[u'\rho_V]-2P\int_\R\mcal{H}[u'\rho_V](t)d\mu_V(t)\Big).
	\end{equation} 
 Since $u\in\mcal{D}(\mcal{L})=\mcal{D}(\mcal{A})=\{u\in \mathsf{H}\ |\ \mathcal{A}[u]\in \mathsf{H}\}$, the functions $u'\rho_V$ and its distributional derivatives $(u'\rho_V)'=-\rho_V\mcal{A}[u]$ and $(u'\rho_V)''=-\dfrac{\rho_V'}{\rho_V}\rho_V^{1/2}.\rho_V^{1/2}\mcal{A}[u]-\rho_V\mcal{A}[u]'$ are in $L^2(\R)$. In particular $u'\rho_V$ goes to zero at infinity, and $\mcal{H}[u'\rho_V]\in H^2(\R)\subset \mcal{C}^1(\R)$. So we can integrate (\ref{premiere derivée}) on $[x,+\infty[$ , since by Lemma \ref{lemme:boundedHilbert}, the right-hand side is a $\underset{|x|\rightarrow\infty}{O}\left(\rho_V(x)\right)$,  to get the following expression for $x$ large enough
	
	\begin{equation}
		u'(x)\rho_V(x)=\int_{x}^{+\infty}\dfrac{\rho_V(t)}{\rho_V'(t)}\Big(f+2P\mcal{H}[u'\rho_V]-2P\int_\R\mcal{H}[u'\rho_V](s)d\mu_V(s)\Big).\rho_V'(t)dt
	\end{equation}
From this expression, we see that $u'\in\mcal{C}^2(\R)$. We now check the integrability condition at infinity and the claimed boundedness properties of $u',u'',u^{(3)}$.
After proceeding to an integration by parts, which is permitted by the previous arguments, we obtain:
	\begin{equation}\label{decomp premiere derivee}
	u'(x)=-\dfrac{\rho_V(x)}{\rho_V'(x)}\Big(f(x)+2P\mcal{H}[u'\rho_V](x)-2P\int_\R\mcal{H}[u'\rho_V](t)d\mu_V(t)\Big)-R_1(x)
	\end{equation}
	where we set $R_1(x)\defi \displaystyle\dfrac{1}{\rho_{V}(x)}\int_{x}^{+\infty}\Bigg[\dfrac{\rho_{V}(t)}{\rho_{V}'(t)}\Big(f+2P\mcal{H}[u'\rho_V]-2P\int_\R\mcal{H}[u'\rho_V](s)d\mu_V(s)\Big)\Bigg]'\rho_{V}(t)dt$. We will need to show that $R_1$ is a remainder of order $\underset{x\rightarrow+\infty}{O}\left(V'(x)^{-2}\right)$ at infinity. In this case, we will have $u'(x)=\underset{x\rightarrow+\infty}{O}\left(V'(x)^{-1}\right)$ which will be useful for the following. If we reinject (\ref{decomp premiere derivee}) in (\ref{equation resolution}), we find:
	
	\begin{multline}\label{eq:ippreg}
		u''=-\Big(f+2P\mcal{H}[u'\rho_V]-2P\int_\R\mcal{H}[u'\rho_V](t)d\mu_V(t)\Big)
  \\-\dfrac{\rho'_V}{\rho_V}\Bigg[-\dfrac{\rho_{V}}{\rho_{V}'}\Big(f+2P\mcal{H}[u'\rho_V]-2P\int_\R\mcal{H}[u'\rho_V](t)d\mu_V(t)\Big)-R_1\Bigg]=\dfrac{\rho_{V}'}{\rho_{V}}R_1
	\end{multline}
	Hence
 \begin{multline*}
     u''(x)=\dfrac{\rho_{V}'}{\rho_{V}^2}(x)\int_{x}^{+\infty}\rho_V(t)dt\Bigg\{\underbrace{\Big(\dfrac{\rho_{V}}{\rho_{V}'}\Big)'(t)}_{=\underset{t\rightarrow+\infty}{O}\Big(\frac{V''(t)}{V'(t)^2}\Big)}\underbrace{\Big[f+2P\mcal{H}[u'\rho_V]-2P\int_\R\mcal{H}[u'\rho_V](s)d\mu_V(s)\Big](t)}_{=\underset{t\rightarrow+\infty}{O}(1)}
 \\+\underbrace{\dfrac{\rho_{V}}{\rho_{V}'}(t)}_{=\underset{t\rightarrow+\infty}{O}\big(V'(t)^{-1}\big)}\underbrace{\Big(f'-2P\mcal{H}[\rho_V\mcal{A}[u]]\Big)(t)}_{=\underset{t\rightarrow+\infty}{O}(1)}\Bigg\}.
 \end{multline*}
We recall that because of assumption  \textit{\ref{ass4}}, $\underset{t\rightarrow+\infty}{O}\Big(\dfrac{V''(t)}{V'(t)^2}\Big)=\underset{t\rightarrow+\infty}{O}\big(V'(t)^{-1}\big)$. The fact that $\mcal{H}[\rho_V\mcal{A}[u]]$ is bounded comes again from lemma \ref{lemme:boundedHilbert}.
	\begin{comment}
		Hence we obtain that after an integration by parts, if we set $h\defi \Bigg(\dfrac{\rho_V}{\rho_{P}'}(f+2P\mcal{H}[u'\rho_V])\Bigg)'$
		
		$$u''(x)=-\dfrac{\rho_{P}'}{\rho_{P}^2}(x)h(x)\dfrac{\rho_{P}}{\rho_{P}'}(x)\rho_{P}(x)-\dfrac{\rho_{P}'}{\rho_{P}^2}(x)\int_{x}^{+\infty}\rho_{P}(t)\Bigg(h\dfrac{\rho_{P}}{\rho_{P}'}\Bigg)'(t)dt=-h(x)-R_2(x)$$ where $h(x)\sim O(\dfrac{V''(x)}{xV'^2(x)}+\dfrac{1}{x^2V'(x)})$ and $R_2(x)\defi \dfrac{\rho_{P}'}{\rho_{P}^2}(x)\displaystyle\int_{x}^{+\infty}dt\Bigg(h\dfrac{\rho_{P}}{\rho_{P}'}\Bigg)'(t)\rho_{P}(t)$.
	\end{comment}
	Differentiating yields

\begin{equation*}
    u^{(3)}(x)=\Big(\dfrac{\rho_V'}{\rho_V}\Big)'(x)R_1(x)-\Big(\dfrac{\rho_V'}{\rho_V}(x)\Big)^2R_1(x)-\dfrac{\rho_V'}{\rho_V}(x)g(x)
\end{equation*}
where we have set $g\defi \Bigg[\dfrac{\rho_{V}}{\rho_{V}'}\Big(f+2P\mcal{H}[u'\rho_V]-2P\int_\R\mcal{H}[u'\rho_V](t)d\mu_V(t)\Big)\Bigg]'$. Now, by the same integration by parts argument as in \eqref{eq:ippreg}, we obtain:

\begin{multline*}
    -\Big(\dfrac{\rho_V'(x)}{\rho_V(x)}\Big)^2R_1(x)-\dfrac{\rho_V'(x)}{\rho_V(x)}g(x)
    \\=-\dfrac{\rho_V'(x)^2}{\rho_V(x)^3}\Big[g(t)\dfrac{\rho_V(t)}{\rho_V'(t)}\rho_V(t)\Big]_x^{+\infty}+\dfrac{\rho_V'(x)^2}{\rho_V(x)^3}\int_\R\Big(g\dfrac{\rho_V}{\rho_V'}\Big)'(t)\rho_V(t)dt
    -\dfrac{\rho_V'(x)}{\rho_V(x)}g(x)
    \\=\Big(\dfrac{\rho_V'(x)}{\rho_V(x)}\Big)^2R_2(x)
\end{multline*}
where 
\begin{equation}\label{eq:def R_2}
    R_2(x)\defi \dfrac{1}{\rho_V(x)}\displaystyle\int_x^{+\infty}\Bigg\{\Bigg[\Big(f+2P\mcal{H}[u'\rho_V]-2P\int_\R\mcal{H}[u'\rho_V](t)d\mu_V(t)\Big)\dfrac{\rho_{V}}{\rho_{V}'}\Bigg]'\dfrac{\rho_V}{\rho_V'}\Bigg\}'(t)\rho_V(t)dt.
\end{equation}
The above integration by part is justified by the fact that $g$ goes indeed to zero at $+\infty$. Hence
$$u^{(3)}(x)=\Big(\dfrac{\rho_V'}{\rho_V}\Big)'(x)R_1(x)+\Big(\dfrac{\rho_V'}{\rho_V}(x)\Big)^2R_2(x).$$
Since $\left(\dfrac{\rho_V'}{\rho_V}\right)'(x)=\underset{x\rightarrow+\infty}{O}\Big(V'(x)\Big)$ using Lemma \ref{lem:regularitedensite} and assumption \textit{\ref{ass4}}, and $\dfrac{\rho_V'}{\rho_V}(x)=\underset{x\rightarrow+\infty}{O}\Big(V'(x)\Big)$, it just remains to check that $R_1(x)=\underset{x\rightarrow+\infty}{O}\Big(V'(x)^{-2}\Big)$ and that $R_2(x)=\underset{x\rightarrow+\infty}{O}\Big(V'(x)^{-3}\Big)$.
\begin{comment}
    	\begin{multline*}
		u^{(3)}(x)=\Big(\dfrac{\rho_V'}{\rho_V^2}\Big)'(x)\rho_V(x)R_1(x)-\Big(\dfrac{\rho_V'}{\rho_V^2}\Big)(x)\Bigg(\dfrac{\rho_{P}}{\rho_{P}'}(f+2P\mcal{H}[u'\rho_V]\red{-2P\int_\R\mcal{H}[u'\rho_V](t)d\mu_V(t)})\Bigg)'(x)\rho_{P}(x)\\=\Big(\underbrace{\dfrac{\rho_V''}{\rho_V}-2\dfrac{\rho_V'^2}{\rho_V^2}\Big)(x)}_{=\underset{x\rightarrow+\infty}{O}\Big(V'(x)^2\Big)}R_1(x)-\underbrace{\Big(\dfrac{\rho_V'}{\rho_V}\Big)(x)\Bigg(\dfrac{\rho_{P}}{\rho_{P}'}(f+2P\mcal{H}[u'\rho_V]\red{-2P\int_\R\mcal{H}[u'\rho_V](t)d\mu_V(t)})\Bigg)'(x)}_{=\underset{x\rightarrow+\infty}{O}\Big(\dfrac{V''(x)}{x^{\frac{1}{2}+\varepsilon}V'(x)}+x^{-\frac{1}{2}-\varepsilon}\Big)}
	\end{multline*}
 	\red{\sout{The second term is $\underset{x\rightarrow+\infty}{O}\Big(x^{-\frac{1}{2}-\varepsilon}\Big)\red{\underset{x\rightarrow+\infty}{O}\Big(x^{-\frac{1}{2}-\varepsilon}\Big)CHANGER}$ by the assumption that $\dfrac{V''}{V'}(x)=\underset{|x|\rightarrow\infty}{O}(1)$.}} \sout{\red{Hence, we just have to check that $R_1(x)=\underset{x\rightarrow+\infty}{O}\Big(\dfrac{1}{x^{\frac{1}{2}+\varepsilon}V'(x)^2}\Big)\red{\underset{x\rightarrow+\infty}{O}\Big(\dfrac{1}{V'(x)^2}\Big)}$ to establish that $u'$, $u''$, $u^{(3)}$ are in $L^2(\R)$.}}
\end{comment}

 Since the integrand in $R_1$ is a $\underset{t\rightarrow+\infty}{O}\Big(V'(t)^{-1}\Big)$, by a comparison argument, we get
	\begin{equation*}
		R_1(x)\defi \underset{x\rightarrow+\infty}{O}\big(I_1(x)\big)\hspace{0,5cm}\text{where}\hspace{0,5cm}I_1(x)\defi \dfrac{1}{\rho_V(x)}\int_{x}^{+\infty}\dfrac{\rho_V(t)}{V'(t)}dt
	\end{equation*}
Furthermore, by integration by parts:

\begin{multline}
	I_1(x)\defi -\dfrac{\rho_{V}(x)}{V'(x)\rho_{V}'(x)}-\dfrac{1}{\rho_V(x)}\int_{x}^{+\infty}\rho_{V}(t)\Big(\dfrac{\rho_{V}}{V'\rho_{V}'}\Big)'(t)dt
 \\=\underset{x\rightarrow+\infty}{O}\Big(V'(x)^{-2}\Big)-\dfrac{1}{\rho_V(x)}\int_{x}^{+\infty}\rho_{V}(t)dt\Bigg\{-\dfrac{V''(t)}{V'(t)^2}\dfrac{\rho_{V}}{\rho_{V}'}+\dfrac{1}{V'(t)}\Big(\dfrac{\rho_{V}}{\rho_{V}'}\Big)'(t)\Bigg\}
\end{multline}
By the same argument as before, the last integral is of the form $\dfrac{1}{\rho_V(x)}\displaystyle\int_{x}^{+\infty}\underset{t\rightarrow+\infty}{O}\Big(\dfrac{\rho_{V}(t)}{V'(t)^2}\Big)dt$ so we can conclude that
$$\dfrac{1}{\rho_V(x)}\int_{x}^{+\infty}\rho_{V}(t)dt\Bigg\{-\dfrac{V''(t)}{V'(t)^2}\dfrac{\rho_{V}}{\rho_{V}'}+\dfrac{1}{V'(t)}\Big(\dfrac{\rho_{V}}{\rho_{V}'}\Big)'(t)\Bigg\}=\underset{x\rightarrow+\infty}{O}\Bigg(\dfrac{1}{\rho_V(x)V'(x)^2}\int_{x}^{+\infty}\rho_{V}(t)dt\Bigg)$$
where we used that $t\mapsto V'(t)^{-2}$ is decreasing on $[x,+\infty[$ for $x$ big enough. In order to conclude that $R_1(x)=\underset{x\rightarrow+\infty}{O}\Big(V'(x)^{-2}\Big)$, one just needs to check that $\dfrac{1}{\rho_V(x)}\displaystyle\int_{x}^{+\infty}\rho_{V}(t)dt=\underset{x\rightarrow+\infty}{o}(1)$.
By two integration by parts, we obtain:
\begin{multline*}
    \dfrac{1}{\rho_V(x)}\displaystyle\int_{x}^{+\infty}\rho_{V}(t)dt=-\dfrac{\rho_V(x)}{\rho_V'(x)}+\Big(\dfrac{\rho_V}{\rho_V'}\Big)'(x)\dfrac{\rho_V(x)}{\rho_V'(x)}+\dfrac{1}{\rho_V(x)}\displaystyle\int_{x}^{+\infty}\rho_{V}(t)\Bigg[\Big(\dfrac{\rho_V}{\rho_V'}\Big)'\dfrac{\rho_V}{\rho_V'}\Bigg]'(t)dt
    \\=\underset{x\rightarrow+\infty}{o}(1)+\dfrac{1}{\rho_V(x)}\displaystyle\int_{x}^{+\infty}\underset{t\rightarrow+\infty}{O}\Big(\dfrac{\rho_{V}(t)}{V'(t)^2}\Big)dt
    \\=\underset{x\rightarrow+\infty}{o}(1)+\underset{x\rightarrow+\infty}{O}\Bigg(\dfrac{1}{\rho_V(x)}\displaystyle\int_{x}^{+\infty}\dfrac{\rho_{V}(t)}{V'(t)^2}dt\Bigg)=\underset{x\rightarrow+\infty}{o}(1)
\end{multline*}
since $t\mapsto\rho_V(t)$ is decreasing on $[x,+\infty[$ for $x$ big enough and that $t\mapsto V'(t)^{-2}$ is integrable at infinity by assumption \textit{\ref{ass4}}. This allows to conclude indeed that $R_1(x)=\underset{x\rightarrow+\infty}{O}\Big(V'(x)^{-2}\Big)$.

Now, we have to check that $R_2(x)=\underset{x\rightarrow+\infty}{O}\Big(V'(x)^{-3}\Big)$. By differentiating \eqref{eq:def R_2}, by the same arguments as before, 
\begin{equation*}
		R_2(x)\defi \underset{x\rightarrow+\infty}{O}\big(I_2(x)\big)\hspace{0,5cm}\text{where}\hspace{0,5cm}I_2(x)\defi \dfrac{1}{\rho_V(x)}\int_{x}^{+\infty}\dfrac{\rho_V(t)}{V'(t)^2}dt.
	\end{equation*}
Moreover, by integration by part, we get just as before:
\begin{multline}
I_2(x)\defi -\dfrac{\rho_{V}(x)}{V'(x)^2\rho_{V}'(x)}-\dfrac{1}{\rho_V(x)}\int_{x}^{+\infty}\rho_{V}(t)\Big(\dfrac{\rho_{V}}{V'^2\rho_{V}'}\Big)'(t)dt
 \\=\underset{x\rightarrow+\infty}{O}\Big(V'(x)^{-3}\Big)-\dfrac{1}{\rho_V(x)}\int_{x}^{+\infty}\underset{t\rightarrow+\infty}{O}\Big(\dfrac{\rho_{V}(t)}{V'(t)^3}\Big)dt\Bigg\}
 \\=\underset{x\rightarrow+\infty}{O}\Big(V'(x)^{-3}\Big)-\underset{x\rightarrow+\infty}{O}\Bigg(\dfrac{1}{\rho_V(x){V'(x)^3}}\int_{x}^{+\infty}\rho_V(t)dt\Bigg)=\underset{x\rightarrow+\infty}{O}\Big(V'(x)^{-3}\Big).
 \end{multline}
Finally, we can conclude that, by doing the same thing near $-\infty$, \begin{equation}\label{behaviour}
u'(x)=\underset{|x|\rightarrow+\infty}{O}\Big(V'(x)^{-1}\Big),\hspace{1cm} u''(x)=\underset{|x|\rightarrow+\infty}{O}\Big(V'(x)^{-1}\Big) \hspace{0,4cm}\text{and}\hspace{0,4cm} u^{(3)}(x)=\underset{|x|\rightarrow+\infty}{O}\Big(V'(x)^{-1}\Big)
\end{equation} which establishes that these functions are in $L^2$ in a neighborhood of $\infty$, again by assumption \textit{\ref{ass4}}. Since we already showed that $u\in\mcal{C}^3(\R)\subset H^3_{\text{loc}}(\R)$, it establishes that $u\in H^3(\R)\cap\mcal{C}^{3}(\R)$ with $u^{(3)}$ bounded. To complete the proof we just have to show that $(u')^2V^{(3)}$, $u'u''V''$, $(u')^2V''$ and $u'V'$ are bounded which is easily checked by (\ref{behaviour}) and assumption \textit{\ref{ass4}}.	
\end{proof}

\begin{comment}
    \begin{remark}
    We choose here the functions that vanishes at infinity at worst like $|x|^{-1/2-\varepsilon}$, but functions like $x\mapsto|x|^{-1/2}\ln^{-1/2-\varepsilon}|x|$ or $x\mapsto|x|^{-1/2}\ln^{-1/2}|x|\ln^{-1/2-\varepsilon}\ln|x|$ also work, the proof being the same. The only hypotheses that we use is that $f\in H^1(\R)\cap \mcal{C}^1(\R)$, that $f'(x)=\underset{|x|\rightarrow+\infty}{O}\left(f(x)\right)$ and that $f$ is decreasing (resp. increasing) in a neighborhood of + (resp.-) $\infty$.
\end{remark}
\end{comment}
\appendix

{
\section{Appendix: proof of Theorem \ref{thm:DiagA}}
\label{app A}

In order to analyze $\mathcal{A}$, we let, for $u\in L^2(\R)$, 
$$\mathcal{S}u \defi \rho_V^{1/2}\mathcal{A}\Big[\rho_V^{-1/2}u\Big]\,.$$
Note that $u \in \left(L^2(\R),\|.\|_{L^2(\R)}\right)\mapsto \rho_V^{-1/2}u \in \left(L^2(\mu_V),\|.\|_{L^2(\mu_V)}\right)$ is an isometry. It turns out that it will be easier to study first the operator $\mcal{S}$ in order to get the spectral properties of $\mcal{A}$.

\begin{proposition}
\label{prop:potentiel schrodinger}The operator $\mcal{S}$ is a Schrödinger operator: it admits the following expression for all $u\in C_c^2(\R)$: $\mcal{S}[u]=-u''+w_V u$ with 

\begin{equation}
    \label{eq:w_V}
    w_V=\frac{1}{2}\left( \frac{1}{2}V'^2-V'' + 2PV'\mathcal{H}[\rho_V]-2P\mathcal{H}[\rho_V']+2P^2\mathcal{H}[\rho_V]^2 \right)=\dfrac{1}{2}\Big[(\ln\rho_V)''+\dfrac{1}{2}(\ln\rho_V)'^2\Big]\,.
\end{equation}
	Furthermore, $w_V$ is continuous and we have $w_V(x)\underset{\infty}{\sim}\dfrac{V'(x)^2}{4}\underset{|x|\rightarrow\infty}{\longrightarrow}+\infty$.
\end{proposition}

\begin{proof}
	We compute directly
	\begin{align*}
		\dfrac{\Big(\rho_V\big(\rho_V^{-1/2}u\big)'\Big)'}{\rho_V}&=\big(\rho_V^{-1/2}u\big)''+\dfrac{\rho_V'}{\rho_V}\big(\rho_V^{-1/2}u\big)'
		\\&=\big(\rho_V^{-1/2}u'-\dfrac{1}{2}\rho_V^{-3/2}\rho_V'u\big)'+\rho_V'\rho_V^{-3/2}u'-\dfrac{1}{2}\rho_V^{-5/2}\big(\rho_V'\big)^2u
  \\&=\rho_V^{-1/2}u''+\dfrac{1}{4}\rho_V^{-5/2}\big(\rho_V'\big)^2u-\dfrac{1}{2}\rho_V^{-3/2}\rho_V''u
		\\&=\rho_V^{-1/2}\Big[u''+\dfrac{1}{4}\rho_V^{-2}\big(\rho_V'\big)^2u-\dfrac{1}{2}\rho_V^{-1}\rho_V''u\Big]
		\\&=\rho_V^{-1/2}\Bigg(u''-\dfrac{1}{2}\Big[\big(\dfrac{\rho_V''}{\rho_V}\big)-\dfrac{1}{2}\big(\dfrac{\rho_V'}{\rho_V}\big)^2\Big]u\Bigg)
		\\&=\rho_V^{-1/2}\Bigg(u''-\dfrac{1}{2}\Big[(\ln\rho_V)''+\dfrac{1}{2}(\ln\rho_V)'^2\Big]u\Bigg) = \rho_V^{-1/2}\Bigg(u''-w_V u\Bigg)\,.
	\end{align*}
Now, using Lemma \ref{lem:regularitedensite}, we have $$w_V=\frac{1}{2}\left(\frac{1}{2}V'^2-V''+2PV'\mathcal{H}[\rho_V]-2P\mathcal{H}[\rho_V']+2P^2\mathcal{H}[\rho_V]^2 \right)\,.$$ Notice that $\mathcal{H}[\rho_V']$ and $\mathcal{H}[\rho_V]$ are bounded since they belong to $H^1(\R)$, as we showed in Lemma \ref{lem:regularitedensite} that $\rho_V$ is $H^2(\R)$.
Along with assumption \textit{\ref{ass3}} and Lemma \ref{lemme:boundedHilbert}, we deduce $w_V(x)\underset{\infty}{\sim}\dfrac{1}{4}V'^2(x)$.
\end{proof}

\begin{remark}
	Note that the function $w_V$ need not be positive on $\R$. In fact, neglecting the terms involving the Hilbert transforms of $\rho_V$ and $\rho_V'$, $w_V$ would only be positive outside of a compact set. However, using the positivity of $\mathcal{A}$, which will be shown further in the article, we can show that the operator $-u''+w_Vu$ is itself positive on $L^2(\R)$. It can also be checked by integration by parts that $\mcal{S}$ is symmetric on $\mathcal{C}_c^\infty(\R)$ with respect to the inner product of $L^2(\R)$.
\end{remark}

\begin{comment}

\begin{lemma}
For $u,v\in\mathcal{C}^2(\R)$ with compact support, we have 
$$ \langle \mathcal{A}u,v\rangle_{L^2(\mu_V)} = \langle u',v' \rangle_{L^2(\rho_P)}\,.$$
As a consequence, $\langle Su,u \rangle_{L^2(d x)}\geq 0$ for all $u$ of class $\mathcal{C}^2$ with compact support.
\end{lemma}

\begin{proof}
The first point follows from integration by partss:
$$\Braket{\mathcal{A}u,v}_{L^2(\rho_P)} = \Braket{-\frac{(u'\rho_P)'}{\rho_P},v\rho_P}_{L^2(d x)} = \langle u',v'\rho_P\rangle_{L^2(d x)}= \langle u',v' \rangle_{L^2(\rho_P)} \,.$$
For the second point, letting $v\defi \rho_P^{-1/2}u$, 
\begin{align*}
    \langle \mathcal{S}u,u \rangle_{L^2(d x)} = \langle \rho_P^{1/2}\mathcal{A}\rho_P^{-1/2}u ,u \rangle_{L^2(d x)} &=\langle \mathcal{A}v, \rho_P v \rangle_{L^2(d x)}=\langle \mathcal{A}v,v \rangle_{L^2(\rho_P)} \geq 0\,.
\end{align*}
\end{proof}
\end{comment}

We now introduce an extension of $\mathcal{S}$ by defining its associated bilinear form.

\begin{definition}[Quadratic form associated to $\mcal{S}$]~\\Let $\alpha\geq0$ such that $w_V+\alpha\geq1$. We define the quadratic form associated to $\mcal{S}+\alpha \msf{id}$, defined for all $u\in\mathcal{C}_c^\infty(\R)$ by
$$q_\alpha(u,u)\defi \int_\R u'^2(x)dx+\int_\R u^2(x)(w_V(x)+\alpha)dx$$
\end{definition}

This quadratic form can be extended to a larger domain denoted by $Q(\mcal{S}+\alpha \msf{id})$, called the \textit{form domain} of the operator $\mcal{S}+\alpha \msf{id}$. By the theory of Schrödinger operators, it is well-known (see  \cite[Theorem 8.2.1]{davies1996spectral}) that such a domain is given by $$Q(\mcal{S}+\alpha \msf{id})=\left\{u\in H^1(\R), u(w_V+\alpha)^{1/2}\in L^2(\R) \right\}=\left\{u\in H^1(\R), u V'\in L^2(\R)  \right\}=: H^1_{V'}(\R)\,.$$
The space $H^1_{V'}(\R)$ can be seen to be the completion of $\mathcal{C}_c^{\infty}(\R)$ under the norm $q_\alpha$.
Now that the quadratic form associated to $\mcal{S}+\alpha \msf{id}$ has been extended to its form domain, it is possible to go back to the operator and extend it by its Friedrichs extension.

\begin{theorem}[Friedrichs extension of $\mcal{S}+\alpha \msf{id}$]~\\
There exists an extension $(\mcal{S}+\alpha\msf{id})_F$ of the operator $\mcal{S}+\alpha \msf{id}$, called the \textit{Friedrichs extension} of $\mcal{S}+\alpha \msf{id}$ defined on $\mcal{D}\Big((\mcal{S}+\alpha\msf{id})_F\Big)=\Big\{u\in H^1_{V'}(\R),-u''+(w_V+\alpha)u\in L^2(\R)\Big\}$.
\end{theorem}

\begin{proof}
We denote \begin{multline*}
    \mcal{D}\Big((\mcal{S}+\alpha\msf{id})_F\Big)\\\defi \Big\{v\in H^1_{V'}(\R),u\in H^1_{V'}(\R)\longmapsto q_\alpha(u,v) \text{ can be extended to a continuous linear functional on }L^2(\R)\Big\}
\end{multline*} 
If $v\in \mcal{D}\Big((\mcal{S}+\alpha\msf{id})_F\Big)$, by Riesz theorem, there exists a unique $f_v\in L^2(\R)$ such that $q_\alpha(u,v)=\braket{u,f_v}_{L^2(\R)}$ holds for all $u\in L^2(\R)$ and we can set $(\mcal{S}+\alpha\msf{id})_F[v]\defi f_v$. Note that it is indeed a way of extending $\mcal{S}+\alpha \msf{id}$ since for all $u,v\in\mathcal{C}_c^\infty(\R)$, $q_\alpha(u,v)=\braket{u,(\mcal{S}+\alpha\msf{id})[v]}_{L^2(\R)}$.

We want to show that $\mcal{D}\Big((\mcal{S}+\alpha\msf{id})_F\Big)=\Big\{u\in H^1_{V'}(\R),-u''+(w_V+\alpha)u\in L^2(\R)\Big\}$. Let $f\in \mcal{D}\Big((\mcal{S}+\alpha\msf{id})_F\Big)$ and $g\defi (\mcal{S}+\alpha\msf{id})_F[f]\in L^2(\R)$. By definition of $q_\alpha$, for all $u\in\mathcal{C}_c^\infty(\R)$:
\begin{equation*}
    \int_\R g(x)u(x)dx=\int_\R \Big[f'(x)u'(x)+(w_V(x)+\alpha)f(x)u(x)\Big]dx=\int_\R \Big[-f(x)u''(x)+(w_V(x)+\alpha)f(x)u(x)\Big] dx
\end{equation*}
Therefore in the sense of distributions, we get $-f''+(w_V+\alpha)=g$ which is a function in $L^2(\R)$, hence $f\in\Big\{u\in H^1_{V'}(\R),-u''+(w_V+\alpha)u\in L^2(\R)\Big\}$. Conversely, if $f\in H^1_{V'}(\R)$ such that $-f''+(w_V+\alpha)f\in L^2(\R)$, it is possible to extend $u\mapsto q_\alpha(f,u)$ to a continuous linear form on $L^2(\R)$ by
$$u\mapsto\int_\R u\Big[-f''(x)+f(x)(w_V(x)+\alpha)\Big]dx$$
which concludes the fact that $\mcal{D} \Big((\mcal{S}+\alpha\msf{id})_F\Big)=\Big\{u\in H^1_{V'}(\R),-u''+(w_V+\alpha)u\in L^2(\R)\Big\}$.
\end{proof}

\textbf{In the following, we will deal only with $(\mcal{S}+\alpha\msf{id})_F:\mcal{D} \Big((\mcal{S}+\alpha\msf{id})_F\Big)\longrightarrow L^2(\R)$ and denote it $\mcal{S}+\alpha \msf{id}:\mcal{D}(\mcal{S}+\alpha\msf{id})\longrightarrow L^2(\R)$.}

\begin{remark}
Note that in the previous proof, the application of Riesz theorem doesn't allow to say that $(\mcal{S}+\alpha\msf{id}):v\in \Big(\mcal{D}(\mcal{S}+\alpha\msf{id}),\|.\|_{q_\alpha}\Big)\mapsto f_v\in \Big(L^2(\R),\|.\|_{L^2(\R)}\Big)$, where $\|.\|_{q_\alpha}$ stands for the norm associated to the bilinear positive definite form $q_\alpha$, is continuous. It can be seen by the fact that \\$v\in \Big(\mcal{D}(\mcal{S}+\alpha\msf{id}),\|.\|_{q_\alpha}\Big)\mapsto q(.,v)\in \Big(L^2(\R)',\|.\|_{L^2(\R)'}\Big)$, where $L^2(\R)'$ stands for the topological dual of $L^2(\R)$ equipped with its usual norm, is not continuous. Indeed the $\|.\|_{q_\alpha}$ norm doesn't control the second derivative of $v$ and hence doesn't provide any module of continuity for the $L^2(\R)$-extended linear form $q(.,v)$.

Also note that, even though it would be convenient that $\mcal{D}\Big((\mcal{S}+\alpha \msf{id}\Big)=L^2(\R,(w_V+\alpha)^2dx)\cap H^2(\R)$ it is not true without more properties on $w_V$. Such a result holds, for example when $w_V$ belongs to $B_2$, the class of reverse Hölder weights, see \cite[Theorem 1.1]{auscher2007maximal}.
\end{remark}

\begin{theorem}[Inversion of $\mcal{S}+\alpha \msf{id}$]~\\
For every $f\in L^2(\R)$, there exists a unique $u\in \mcal{D}\Big(\mcal{S}+\alpha \msf{id}\Big)$ such that $(\mcal{S}+\alpha\msf{id})[u]=f$. Furthermore, the map $(\mcal{S}+\alpha\msf{id})^{-1}$ is continuous from $\big(L^2(\R),\|.\|_{L^2(\R)}\big)$ to $\big(\mcal{D}(\mcal{S}+\alpha\msf{id}),\|.\|_{q_\alpha}\big)$.
\end{theorem}

\begin{proof}
Let $f\in L^2(\R)$, the map $u\longmapsto\braket{u,f}_{L^2(\R)}$ is continuous on $\big(H_{V'}^1(\R),\|.\|_{q_\alpha}\big)$ which is a Hilbert space. Therefore, by Riesz theorem, there exists a unique $v_f\in H^1_{V'}(\R)$ such that for all $u\in H^1_{V'}(\R)$, $$\braket{f,u}_{L^2(\R)}=q_\alpha(v_f,u)$$ from which we deduce that, in the sense of distributions, $f=-v_f''+(w_V+\alpha)v_f$. This implies that $v_f\in \mcal{D}(\mcal{S}+\alpha\msf{id})$. Since $v_f\in \mcal{D}(\mcal{S}+\alpha\msf{id})$, we have then for all $u\in L^2(\R)$, $$\braket{f,u}_{L^2(\R)}=q_\alpha(v_f,u)=\Braket{(\mcal{S}+\alpha)v_f,u}_{L^2(\R)},$$ hence $(\mcal{S}+\alpha\msf{id})[v_f]=f$. Finally, by Riesz theorem, $f\in L^2(\R)\mapsto v_f\in H^1_{V'}(\R)$ is continuous hence so is $(\mcal{S}+\alpha\msf{id})^{-1}$.
\end{proof}

\begin{remark}
It would be tempting to use Banach's isomorphism theorem to say that since $(\mcal{S}+\alpha\msf{id})^{-1}$ is bijective and continuous, so must be $\mcal{S}+\alpha \msf{id}$. But since $\big(\mcal{D}(\mcal{S}+\alpha\msf{id}),\|.\|_{q_\alpha}\big)$ is not a Banach space (it's not closed in $H^1_{V'}(\R)$) we can't apply it.
\end{remark}

We are now able to diagonalize the resolvent of $\mathcal{S}$.

\begin{theorem}[Diagonalization of $(\mcal{S}+\alpha\msf{id})^{-1}$]~\\
\label{thm:diagSchrodinger}
There exists a complete orthonormal set $(\psi_n)_{n\geq 0}$ of $L^2(\R)$ (meaning that $$\overline{\Span\{\psi_n,\ n\geq 0\}}^{\|.\|_{L^2(\R)}}=L^2(\R)$$ and $\langle \psi_i,\psi_j \rangle_{L^2(\R)}=\delta_{i,j}$), where each $\psi_n \in \mcal{D}(\mcal{S}+\alpha\msf{id})$ and $\big(\mu_n(\alpha)\big)_{n\geq0}\in [0,1]^\N$ with $\mu_n(\alpha)\underset{N\rightarrow\infty}{\longrightarrow}0$ such that $(\mcal{S}+\alpha\msf{id})^{-1}[\psi_n]=\mu_n(\alpha)\psi_n$ for all $n\geq0$. We also have for all $v\in L^2(\R)$, $$\Big\|\big(\mcal{S}+\alpha \msf{id}\big)^{-1}[v]\Big\|_{L^2(\R)}\leq \|v\|_{L^2(\R)}.$$
\end{theorem}

\begin{proof}
By Proposition \ref{prop:potentiel schrodinger}, $w_V+\alpha$ is continuous and goes to infinity at infinity. By Rellich criterion \cite[Theorem XIII.65]{reed1978methods}, the unit ball of $\mcal{D}(\mcal{S}+\alpha\msf{id})$, \textit{i.e.} the set $$\Big\{u\in \mcal{D}(\mcal{S}+\alpha\msf{id}),\,\int_\R u'^2(x)dx+\int_\R (w_V(x)+\alpha)u^2(x)dx\leq 1\Big\}$$ considered as a subset of $L^2(\R)$ is relatively compact in $\big(L^2(\R),\|.\|_{L^2(\R)}\big)$. Hence, we can conclude that the injection $\iota:\big(\mcal{D}(\mcal{S}+\alpha\msf{id}),\|.\|_{q_\alpha}\big)\longrightarrow\big(L^2(\R),\|.\|_{L^2(\R)}\big)$ is a compact operator. Since $(\mcal{S}+\alpha\msf{id})^{-1}:\big(L^2(\R),\|.\|_{L^2(\R)}\big)\longrightarrow\big(\mcal{D}(\mcal{S}+\alpha\msf{id}),\|.\|_{q_\alpha}\big)$ is continuous then $(\mcal{S}+\alpha\msf{id})^{-1}$ is compact from $\big(L^2(\R),\|.\|_{L^2(\R)}\big)$ to itself. The fact that $(\mcal{S}+\alpha\msf{id})^{-1}$ is self-adjoint and positive allows us to apply the spectral theorem to obtain $\big(\mu_n(\alpha)\big)_{n\geq0}$ positive eigenvalues verifying $\mu_n(\alpha)\underset{N\rightarrow\infty}{\longrightarrow}0$ by compactness and a Hilbertian basis $(\psi_n)_{n\geq 0}\in L^2(\R)^\N$, such that for all $n\geq0$, $(\mcal{S}+\alpha\msf{id})^{-1}\psi_n=\mu_n(\alpha)\psi_n$. It is then easy to see that for all $n$, $\psi_n\in \mcal{D}(\mcal{S}+\alpha\msf{id})$ since they belong to the range of $(\mcal{S}+\alpha\msf{id})^{-1}$. Finally, since for all $\phi\in L^2(\R)$, $\Braket{(\mcal{S}+\alpha\msf{id})\phi,\phi}_{L^2(\R)}\geq \|\phi\|_{L^2(\R)}^2$, the spectrum of $(\mcal{S}+\alpha\msf{id})^{-1}$ is contained in $[0,1]$. It allows us to conclude that $\vertiii{(\mcal{S}+\alpha\msf{id})^{-1}}_{L^2(\R)}\leq 1$.
\end{proof}

\begin{comment}
goes to infinity at infinity. By \cite{reed1978methods}[Theorem XIII.65], $Q_\alpha\subset \mathcal{Q}$ is relatively compact in $\big(L^2(\R),\|.\|_{L^2(\R)}\big)$. This implies that the weak inverse $(\mathcal{S}+\alpha\msf{id})^{-1} : L^2(\R) \to Q_\alpha$, seen as an operator from $\Big(L^2(\R),\|.\|_{L^2(\R)}\Big)$ to itself, is compact. Thus, by the spectral theorem for compact, self-adjoint operators, the set of eigenvalues of $\mathcal{S}+\alpha \msf{id}$ is discrete, and there exists a complete orthonormal set (meaning that $\text{Span}\{\psi_n,\ n\geq 0\}$ is dense in $\Big(L^2(d x),\|.\|_{L^2(d x)}\Big)$ and $\langle \psi_i,\psi_j \rangle_{L^2} = \delta_{i,j}$) of eigenfunctions $(\psi_n)_{n\geq 0}$ of $L^2(d x)$. Furthermore, $\psi_n \in Q_\alpha$. In symbols, for $u\in L^2(d x)$, we have, denoting the eigenvalues of $\mathcal{S}+\alpha \msf{id}$ by $0<\mu_0(\alpha)<\mu_1(\alpha)<\ldots$
$$(\mathcal{S}+\alpha\msf{id})^{-1}u = \sum_{n\geq 0} \mu_n(\alpha)\langle u,\psi_n\rangle_{L^2(d x)}\psi_n\,. $$
This establishes that the spectrum $\lambda_0\leq \lambda_1 \leq \ldots$ of $\mathcal{S}$ is discrete, and because $\mathcal{S}$ is non negative, $\lambda_0 \geq 0$.
\end{comment}
Since for all $u\in H^1_{V'}(\R)$, $(\mcal{S}+\alpha)[u]\in L^2(\R)$ \textit{iff} $\mcal{S}[u]\in L^2(\R)$, if we define $\mcal{D}(\mcal{S})$ in the same manner that we did before, $\mcal{D}(\mcal{S})=\mcal{D}(\mcal{S}+\alpha\msf{id})$.
It is now straightforward to see that we can extend $\mcal{A}=\rho_V^{-1/2}\mcal{S}\Big[\rho_V^{1/2}\cdot\Big]$ on $\mcal{D}_{L^2(\mu_V)}(\mcal{A})\defi \rho_V^{-1/2} \mcal{D}(\mcal{S})$ equipped with the norm $\|.\|_{q_\alpha,\rho_V}$ defined for all $u\in \mcal{D}_{L^2(\mu_V)}(\mcal{A})$ by
$$\|u\|_{q_\alpha,\rho_V}=\int_\R u'^2(x)d\mu_V(x)+\int_\R u^2(x)(w_V(x)+\alpha)d\mu_V(x)\, .$$ 
It is easy to see that $(\mcal{A}+\alpha \msf{id})^{-1}$ is continuous.

\begin{remark}
The kernel of $\mcal{A}$ is generated by the function $\widetilde{1}$. Indeed if $\phi\in \mcal{D}_{L^2(\mu_V)}(\mcal{A})$ is in the kernel of $\mcal{A}$ then

$$0=-\dfrac{\big(\phi'\rho_V\big)'}{\rho_V}\Rightarrow\exists c\in\R,\,\phi'=\dfrac{c}{\rho_V}$$
But since $\phi'$ is in $L^2(\mu_V)$ then $c=0$ which implies that $\phi$ is constant. We must restrict $\mcal{A}$ to the orthogonal of $Ker\mcal{A}$ with respect to the inner product of $L^2(\mu_V)$, \textit{i.e.}
$$ \mcal{D}_{L^2_0(\mu_V)}(\mcal{A})\defi \left\{ u\in \mcal{D}_{L^2(\mu_V)}(\mcal{A})\,|\,\int_\R u(x)d\mu_V(x) = 0 \right\}\,.$$
Doing so makes $\mcal{A}$ injective.
\end{remark}

Before inverting $\mcal{A}$, we need the following lemma:

\begin{lemma} The following equality holds
\label{lemma:ortho}
$$(\mcal{A}+\alpha \msf{id})\Big(\mcal{D}_{L^2_0(\mu_V)}(\mcal{A})\Big)=L_0^2(\mu_V)\defi \Big\{u\in L^2(\mu_V),\,\int_\R u(x)d\mu_V(x)=0\Big\}$$
\end{lemma}

\begin{proof}
\label{prooflemmaa.10}
Let $\phi=\widetilde{c}$ for $c\in\R$, $(\mcal{A}+\alpha \msf{id})[\phi]=\widetilde{\alpha c}$ then $(\mcal{A}+\alpha\msf{id})(\R.\widetilde{1})=\R \widetilde{1}$. Hence since $\mcal{A}+\alpha \msf{id}$ is self-adjoint with respect to the inner product of $L^2(\mu_V)$ and that $\R\widetilde{1}$ is stable by $\mcal{A}+\alpha \msf{id}$, then  $(\mcal{A}+\alpha\msf{id})\Big((\R.\widetilde{1})^\perp \cap \mcal{D}_{L^2(\mu_V)}(\mcal{A})\Big)\subset(\R.\widetilde{1})^\perp$. For the converse, let $u\in(\R.\widetilde{1})^\perp$, since $\mcal{A}+\alpha \msf{id}$ is bijective, there exists $v\in \mcal{D}_{L^2(\mu_V)}(\mcal{A})$ such that $u=(\mcal{A}+\alpha\msf{id})[v]$. For all $w\in \R.\widetilde{1}$,
$$0=\Braket{u,w}_{L^2(\mu_V)}=\Braket{(\mcal{A}+\alpha\msf{id})[v],w}_{L^2(\mu_V)}=\Braket{v,(\mcal{A}+\alpha\msf{id})[w]}_{L^2(\mu_V)}$$
Hence $v\in \big[(\mcal{A}+\alpha\msf{id})(\R\widetilde{1})\big]^\perp=\R\widetilde{1}^\perp$ and so $(\R.\widetilde{1})^\perp\subset(\mcal{A}+\alpha\msf{id})\Big((\R.\widetilde{1})^\perp\Big)$.
\end{proof}

\begin{remark}
    It is easy to see that $L_0^2(\mu_V)$ is a closed subset of $L^2(\mu_V)$ as it is the kernel of the linear form $\phi\in L^2(\mu_V)\mapsto\Braket{\phi,\widetilde{1}}_{L^2(\mu_V)}$, making it a Hilbert space.
\end{remark}

We now are able to diagonalize $\mcal{A}$ and deduce that it is invertible.

\begin{proposition}[Diagonalization and invertibility of $\mathcal{A}$]
\label{prop:diagoL2}
    Let $L_0^2(\mu_V)$ be given in Lemma \eqref{lemma:ortho}.
     There exists a complete orthonormal set of $\left(L_0^2(\mu_V),\Braket{.,.}_{L^2(\mu_V)}\right)$, $(\phi_n)_{n\in\N}\in \mcal{D}_{L^2_0(\mu_V)}(\mcal{A})^\N$ (meaning that $$\overline{\Span\{\phi_n,\ n\geq 0\}}^{\|.\|_{L^2(\mu_V)}}=L_0^2(\mu_V)$$ and $\langle \phi_i,\phi_j \rangle_{L^2(\mu_V)}=\delta_{i,j}$), and an increasing sequence $(\lambda_n)_{n\geq 1}$ such that $\mcal{A}[\phi_n]=\lambda_n\phi_n$. Furthermore, $\mcal{A}:\mcal{D}_{L^2_0(\mu_V)}(\mcal{A})\longrightarrow L_0^2(\mu_V)$ is bijective, $\mcal{A}^{-1}$ is continuous and verifies for all $v\in L_0^2(\mu_V)$, $$\|\mcal{A}^{-1}[v]\|_{L^2(\mu_V)}\leq\lambda_1^{-1}\|v\|_{L^2(\mu_V)}.$$
\end{proposition}

\begin{proof}
Since $(\mcal{S}+\alpha\msf{id})^{-1}$ considered as an operator of $L^2(\R)$, is compact so is $(\mcal{A}+\alpha\msf{id})^{-1}$ on $L^2(\mu_V)$ and since $\mcal{A}$ is self-adjoint, by the spectral theorem, $(\mcal{A}+\alpha\msf{id})^{-1}$ is diagonalizable. With the notations of Theorem \ref{thm:diagSchrodinger}, $(\mcal{A}+\alpha\msf{id})^{-1}$ has eigenvalues $\big(\mu_n(\alpha)\big)_{n\geq0}$ and corresponding eigenfunctions $\phi_n=\rho_V^{-1/2}\psi_n\in \mcal{D}_{L^2(\mu_V)}(\mcal{A})$. Hence for all $n\in\N$, $\mcal{A}[\phi_n]=\lambda_n\phi_n$ with $\lambda_n\defi \big(\dfrac{1}{\mu_n(\alpha)}-\alpha\big)$. Now,
$$\lambda_n\|\phi_n\|^2_{L^2(\mu_V)}=\int_\R\mcal{A}[\phi_n](x)\phi_n(x)d\mu_V(x)=-\int_\R(\rho_V\phi_n')'(x)\phi_n(x)dx=\int_\R\phi_n'(x)^2d\mu_V(x)\geq\,0\,.$$
Furthermore, the kernel of $\mcal{A}$ is $\R.\widetilde{1}$, thus $\mcal{A}$ restricted to $\mcal{D}_{L^2_0(\mu_V)}(\mcal{A}) $ is positive. Finally, $\mcal{A}$ is surjective in $L^2_0(\mu_V)$:
First, notice that since $\mcal{A}$ is self-adjoint
\begin{equation}
    \label{egalitedomaines}
    \mcal{D}_{L^2_0(\mu_V)}(\mcal{A})=\left\{ u\in L^2_0(\mu_V)\ |\ \sum_{n=1}^{+\infty}\lambda_n^2\braket{u,\phi_n}_{L^2(\mu_V)}^2 <+\infty \right\}\,.
\end{equation}
\begin{comment}
Indeed, if $u\in\mathcal{D}_{L^2(\R),0}(\mathcal{A})$ then for $n\geq 1$ we have by symmetry
$$ \braket{\mathcal{A}u, \phi_n}_{L^2(\mu_V)} = \braket{u, \mathcal{A}\phi_n}_{L^2(\mu_V)}=\lambda_n\braket{u,\phi_n}_{L^2(\mu_V)},$$
therefore by orthogonality
$$\|\mathcal{A}u\|^2_{L^2(\mu_V)}=\sum_{n=1}^{+\infty}\lambda_n^2\braket{u,\phi_n}_{L^2(\mu_V)}^2 <+\infty\,.$$
Conversely, if $u\in L^2_0(\mu_V)$ is such that the above sum is finite, the same computation shows that $\|\mathcal{A}u\|^2_{L^2(\mu_V)}$ is finite.
\end{comment}
Now, for $v\in L^2_0(\mu_V)$, we have
$$\mcal{A}[u]=v\text{,\quad where \quad}u\defi \sum_{n\geq1}\dfrac{\braket{v,\phi_n}_{L^2(\mu_V)}}{\lambda_n}\phi_n.$$
$u$ is indeed an element of $\mcal{D}_{L^2_0(\mu_V)}(\mcal{A})$ because of \eqref{egalitedomaines}, the fact that $v\in L^2(\mu_V)$ and that $(\lambda_n)_n$ is bounded from below. Finally, for all $n\in\N$, by orthonormality of $(\phi_n)_n$, $$\braket{\mcal{A}[u],\phi_n}_{L^2(\mu_V)}=\lambda_n\braket{u,\phi_n}_{L^2(\mu_V)}=\braket{v,\phi_n}_{L^2(\mu_V)}$$
which allows to conclude on the bijectivity of $\mcal{A}:\mcal{D}_{L^2_0(\mu_V)}(\mcal{A})\longrightarrow L_0^2(\mu_V)$. Finally, it is easy to see that
$$\|\mcal{A}^{-1}[v]\|_{L^2(\mu_V)}^2=\sum_{n\geq1}\dfrac{|\braket{v,\phi_n}_{L^2(\mu_V)}|^2}{\lambda_n^2}\leq\lambda_1^{-2}\|v\|_{L^2(\mu_V)}^2.$$

\begin{comment}
\sout{But since $(\mcal{A}+\alpha\msf{id})^{-1}$ is a compact operator of $L^2(\mu_V)$ and that $(\mcal{A}+\alpha\msf{id})$ maps $\R.\widetilde{1}^\perp$ to $\R.\widetilde{1}^\perp$ with respect to the inner product of $L^2(\mu_V)$ (see lemma \ref{lemma:ortho}), then $\big(\mcal{A}+\alpha \msf{id}\big)^{-1}$ is compact as an operator from $L_0^2(\mu_V)$ to itself. By Fredholm alternative, for every $\lambda\in\R$ $\lambda\neq 0$, either $(\mcal{A}+\alpha\msf{id})^{-1}-\lambda I$ is bijective either $\lambda\in Sp\big((\mcal{A}+\alpha\msf{id})^{-1}\big)$. These conditions are equivalent to: either $\mcal{A}+(\alpha -\dfrac{1}{\lambda})I$ is bijective as an operator from $\mcal{D}_{L^2_0(\mu_V)}(\mcal{A})$ to $L^2_0(\mu_V)$, either $-\alpha+\dfrac{1}{\lambda}\in Sp\big(\mcal{A}\big)$. If we set $\lambda=\dfrac1\alpha$ then either $\mcal{A}$ is bijective either $0\in Sp(\mcal{A})$, since the latter is wrong then $\mcal{A}:\mcal{D}_{L^2_0(\mu_V)}(\mcal{A})\rightarrow L_0^2(\mu_V)$ is bijective. The spectrum of $\mcal{A}$ is $\left(\dfrac{1}{\mu_n(\alpha)}-\alpha\right)_{n\geq0}\subset(\lambda_1,+\infty)\subset(0,+\infty)$, where $\lambda_1$ is the smallest eigenvalue, hence we deduce that $\vertiii{\mcal{A}^{-1}}_{L^2(\mu_V)}\leq \lambda_1^{-1}$.}
\end{comment}
\end{proof}

\bibliographystyle{alpha}

\bibliography{bibAbel3}

\end{document}